\documentclass[final,1p,times,number]{elsarticle}

\journal{Advances in Mathematics}
\usepackage{amsmath,amsthm,amscd,amssymb,bbm,titlesec}

\usepackage[colorlinks]{hyperref}

\usepackage[arrow]{xy}

\theoremstyle{plain}
\newtheorem{theorem}{Theorem}[section]
\newtheorem{lemma}[theorem]{Lemma}
\newtheorem{proposition}[theorem]{Proposition}

\newtheorem{corollary}[theorem]{Corollary}

\theoremstyle{definition}
\newtheorem{definition}[theorem]{Definition}

\newtheorem{example}[theorem]{Example}

\theoremstyle{remark}
\newtheorem{remark}[theorem]{Remark}

\def\Mc{{\mathcal M}}
\def\Ec{{\mathcal E}}
\def\Qr{{P}^{2}}
\def\Homc{\mathcal Hom}

\def\Ob{{\mathrm Ob}}

\def\Bm{{\bf B}}
\def\dd{{\bf d}}
\def\Dop{{\mathcal D}}

\def\Z{{\mathbb Z}}
\def\C{{\mathbb C}}
\def\Q{{\mathbb Q}}
\def\N{{\mathbb N}}

\def\P{{P}}

\newcommand\uno{{\mathbbm 1}}
\DeclareMathOperator{\GL}{GL}
\DeclareMathOperator{\Vect}{\bf Vect}
\DeclareMathOperator{\Rep}{\bf Rep}
\DeclareMathOperator{\Ker}{Ker}
\DeclareMathOperator{\Spec}{Spec}
\DeclareMathOperator{\Hom}{Hom}
\DeclareMathOperator{\End}{End}
\DeclareMathOperator{\Aut}{Aut}
\DeclareMathOperator{\Gal}{Gal}
\newcommand{\Sets}{\mathbf{\bf Sets}}
\newcommand{\Fun}{\mathbf{Fun}}
\DeclareMathOperator{\id}{id}

\DeclareMathOperator{\spec}{Spec}
\newcommand{\gr}{{\rm gr}}
\DeclareMathOperator{\At}{At}
\DeclareMathOperator{\Img}{Im}

\DeclareMathOperator{\Char}{char}

\DeclareMathOperator{\fIsom}{Isom}

\DeclareMathOperator{\trdeg}{tr.deg}
\DeclareMathOperator{\Sym}{{\rm Sym}}
\DeclareMathOperator{\Der}{Der}

\DeclareMathOperator{\p}{\mathfrak{p}}

\DeclareMathOperator{\Mod}{\bf Mod}
\DeclareMathOperator{\Modf}{\bf Mod^{\it fg}}

\DeclareMathOperator{\Comod}{\bf Comod}
\DeclareMathOperator{\Vectf}{\bf Vect^{\it fg}}
\DeclareMathOperator{\Repf}{\bf Rep^{\it fg}}
\DeclareMathOperator{\Comodf}{\bf Comod^{\it fg}}

\DeclareMathOperator{\DModf}{\bf DMod^{\it fg}}
\DeclareMathOperator{\DMod}{\bf DMod}

\DeclareMathOperator{\DAlg}{\bf DAlg}
\DeclareMathOperator{\Alg}{\bf Alg}

\DeclareMathOperator{\PPV}{\bf PPV}
\DeclareMathOperator{\Mat}{Mat}
\DeclareMathOperator{\Isom}{\underline{Isom}}

\DeclareMathOperator{\Cat}{\mathcal{C}}
\DeclareMathOperator{\Frac}{Frac}

\newcommand{\Le}{\leqslant}
\newcommand{\Ge}{\geqslant}

\titlespacing\subsection{\parskip}%
           {12pt plus 6pt minus 3pt}%
           {8pt plus 6pt minus 3pt}%

\begin{document}

\begin{frontmatter}

\begin{keyword}
Differential algebra\sep Tannakian category\sep
Parameterized differential Galois theory\sep Atiyah extension
\MSC[2010] primary 12H05\sep secondary 12H20\sep 13N10\sep 20G05\sep 20H20\sep 34M15  
\end{keyword}

\title{Parameterized Picard--Vessiot extensions and Atiyah extensions}

\author{Henri Gillet}
\ead{gillet@uic.edu}
\address{University of Illinois at Chicago, Department of Mathematics, Statistics, and Computer Science, 851 S Morgan Street,
Chicago, IL 60607-7045}

\author{Sergey Gorchinskiy}
\ead{gorchins@mi.ras.ru}
\address{Steklov Mathematical Institute, Gubkina str. 8, Moscow, 119991, Russia}

\author{Alexey Ovchinnikov}
\ead{aovchinnikov@qc.cuny.edu}
\address{CUNY Queens College, Department of Mathematics, 65-30 Kissena Blvd,
Queens, NY 11367-1597, USA\\
CUNY Graduate Center, Department of Mathematics, 365 Fifth Avenue,
New York, NY 10016, USA}


\begin{abstract}
Generalizing Atiyah extensions, we introduce and study differential
abelian tensor categories over differential rings. By a
differential ring, we mean a commutative ring with an action of a
Lie ring by derivations. In particular, these derivations act on a differential category. A differential Tannakian theory is developed. The main application is to the
Galois theory of linear differential equations with parameters.
Namely, we show the existence of a parameterized Picard--Vessiot
extension and, therefore, the Galois correspondence for many differential fields with, possibly,
non-differentially closed fields of constants, that is, fields of
functions of parameters. Other applications include a substantially simplified test for a system of linear differential equations with parameters to be isomonodromic, which will appear in a separate paper. This application is based on differential categories developed in the present paper, and not just differential algebraic groups and their representations.
\end{abstract}

\end{frontmatter}

\tableofcontents

\section{Introduction}

Classical differential Galois theory studies symmetry
groups of solutions of linear differential equations, or
equivalently the groups of automorphisms of the corresponding
extensions of differential fields; the groups that arise are linear
algebraic groups over the field of constants. This theory, started in the 19th century by Picard and Vessiot, was put on a firm
modern footing by Kolchin~\cite{Kolchin1948}. In \cite{Landesman},
Landesman initiated a generalized differential Galois theory that uses
Kolchin's axiomatic approach~\cite{KolDAG} and realizes differential
algebraic groups as Galois groups. The parameterized Picard--Vessiot Galois theory considered by Cassidy and Singer in~\cite{PhyllisMichael} is a special case of the Landesman generalized differential Galois theory and studies symmetry groups of the solutions of linear differential equations whose coefficients contain parameters. This is done by
constructing a differential field containing the solutions and their
derivatives with respect to the parameters, called a parameterized
Picard--Vessiot (PPV) extension, and studying its group of
differential symmetries, called a parameterized differential Galois
group. The Galois groups that arise are linear differential algebraic
groups, which are defined by polynomial differential
equations in the parameters. Another approach to the Galois theory of systems of linear differential equations with parameters is given in~\cite{TG}, where the authors study Galois groups for generic values of the parameters.

The tradition in classical differential Galois theory has been to assume that the field of constants of the coefficient field is algebraically closed~\cite{Kolchin1948,Michael,Magid}. Cassidy and Singer follow the spirit of this tradition.
For example, as in \cite[Section~3]{PhyllisMichael}, consider
the differential equation $\partial_x f=\frac{t}{x}f$. The solutions
of this equation will be functions of $x$, which also depend on the
parameter $t$. If $x$ and $t$ are complex variables, these solutions
are of the form $a\cdot x^t$, $a\in\mathbb{C}(t)$, and the field
generated by the solutions together with their derivatives with
respect to both $x$ and~$t$ is
$\mathbb{C}{\left(x,t,x^t,\log(x)\right)}$.  The automorphisms of this
field over $\mathbb{C}(x,t)$ are given by non-zero elements $a$ in
$\mathbb{C}(t)$ that satisfy the differential equation
$\partial_t{\left(\frac{\partial_t(a)}{a}\right)}=0$. However, as explained in
\cite{PhyllisMichael}, this group does not have enough elements to
give a Galois correspondence between subgroups of the group of
automorphisms and intermediate differential fields. This leads
Cassidy and Singer to require that the field of
$\partial_x$-constants is a $\partial_t$-closed differential field
(or, more generally, that the field of functions of the parameters is
differentially closed).

Recall that a differential field is differentially
closed if it contains solutions of consistent systems of polynomial differential equations with coefficients in the field. However, this requirement is an obstacle to the practical applicability of the methods of the parameterized theory. A
similar phenomenon occurs in the classical differential Galois theory:
if the field of constants is not algebraically closed,
a Picard--Vessiot extension might not exist at all
(see the famous counterexample of
Seidenberg~\cite{SeidenbergExample}); therefore, there are no differential Galois group and Galois correspondence if this happens. Since the beginning of the theory~\cite{Kolchin1948}, it has been a major open problem in Picard--Vessiot
theory to determine to what extent one can avoid taking the
algebraic closure of the field of constants. In the present paper, we are able to
remove the assumption that the field of constants has to be
differentially closed in order to have a Galois correspondence in
the parameterized case.

With this aim, following~\cite{DeligneFS}, we use the Tannakian
approach to linear differential equations. In particular, in the
usual non-parameterized case \cite{Michael}, we show in
Theorem~\ref{theor-nonparmain} that, under a relatively existentially closed
assumption on the
field of constants (which includes the case of formally real fields with real closed fields of constants, as well as fields that
are purely transcendental extensions of the fields of constants),
one can always construct a Picard--Vessiot extension for a system of
linear differential equations. To treat the parameterized case,
which is our main interest, we develop a theory of differential
categories over differential rings and the corresponding theory of
differential Tannakian categories. Here, by a differential ring, we
mean a commutative ring together with a Lie ring acting on it by
derivations (this is also often called a Lie algebroid). The theory of
differential Tannakian categories allows us to show that a PPV
extensions exists under a much milder assumption (relatively differentially closed)
on the field of constants than being differentially closed,
Theorem~\ref{theor-main}. This assumption is satisfied by
many differential fields used in practice, Theorem~\ref{thm-main}.

The importance of the existence of a PPV extension is that it leads
to a Galois correspondence, Section~\ref{section-Galois}. The Galois
group is a differential  algebraic group
\cite{Cassidy,CassidyRep,KolDAG,OvchRecoverGroup,PhyllisMichaelJH,diffreductive,MinOvRepSL2,MichaelGmGa}    defined over the field of
constants, which, after passing to the differential closure,
coincides with the parameterized differential Galois group
from~\cite{PhyllisMichael}, Corollary~\ref{corol-Michael}. The
Galois correspondence, as usual, can be used to analyze how one may
build the extension, step-by-step, by adjoining solutions of
differential equations of lower order, corresponding to taking
intermediate extensions of the base field. For example, consider the
special function known as the incomplete Gamma function $\gamma$,
which is the solution of a second-order parameterized differential
equation \cite[Example~7.2]{PhyllisMichael} over $\Q(x,t)$. Knowing
the relevant Galois correspondence, one could show how to build the
differential field extension of $\Q(x,t)$ containing $\gamma$
without taking the (unnecessary and unnatural) differential closure
of $\Q(t)$.

The general nature of our approach will allow in the future to adapt
it to the Galois theory of linear difference equations, which has
numerous applications. Differential algebraic dependencies among
solutions of difference equations were studied in~\cite{CharlotteArxiv,CharlotteComp,CharlotteMichael,CharlotteLucia11,CharlotteLucia12,CharlotteLucia13,DiVizioHardouin:DescentandConfluence,HDV,LuciaSMF,Granier}.
Among many applications of the Galois theory, one has an algebraic
proof of the differential algebraic independence of the
Gamma function over $\C(x)$, \cite{CharlotteMichael} (the Gamma function
satisfies the difference equation $\Gamma(x+1)=x\cdot
\Gamma(x)$). Moreover, such a method leads to algorithms, given in
the above papers, that test differential algebraic dependency with
applications to solutions of even higher order difference equations
(hypergeometric functions, etc.). General results on the subject can
be found in~\cite{TakeuchiHA,Malgrange2,Pillay2004,Umemura2,Umemura,TrushinSplitting}.

It turns out that the results of the present paper, including the new theory of differential categories not restricted to the case of just one derivation, lead (see  \cite{GO}) to a new
understanding of isomonodromic systems of parameterized linear differential equations
\cite{ClaudineMichael2,ClaudineMichael,PhyllisMichael,MalgrangeSing2,MalgrangeSing,Sibuya} allowing one to substantially simplify the test for isomonodromicity generalizing the classical results \cite{JM,JMU}. This, in turn, has become a part of a new algorithm \cite{MiOvSi} in the PPV theory.

Let us compare the present paper with some previously known results.
The existence of a PV extension with a non-algebraically closed field
of constants was considered by a number of authors. In particular,
the case when the Galois group is~$\GL_n$, the base field is
formally real, and the constants are real closed was solved
positively in \cite{Sowa}, while the case of the
field~$\mathbb{R}(z)$ has been also studied in~\cite{Tobias2008}. In
the case of one derivation, differential Tannakian categories were
defined and studied in~\cite{OvchTannakian,difftann,Moshe}. In the
present paper, we define differential Tannakian categories over
fields that may have many derivations. A similar approach in the case of one derivation was independently developed in~\cite{Moshe2011}. Also, we do not choose a
basis of the space of derivations, allowing us to give a functorial
description of the constructions involved. One reason that this
generalization is needed was explained in~\cite{Besser}, in the
context of Coleman integration. The paper~\cite{diffreductive}
considers the case of several derivations but chooses a basis in the
space of derivations and uses a fiber functor to give the axioms of
a differential Tannakian category. On the contrary, the axioms in
the present paper need to be and are given independently of the
fiber functor.

It turns out  that \cite{Wibmer3}, in the case of one derivation, one
can relax the differentially closed assumption,  and just ask that
the field of constants be algebraically closed in order to guarantee
the existence of a PPV extension, by using the more straightforward
method of differential kernels~\cite{Lando2}. This approach was
initiated by M.~Wibmer who first applied difference
kernels~\cite{Cohn,Levin} to study differential equations with difference
parameters \cite{Wibmer2}. While not including all the cases
from~\cite{Wibmer3}, the method presented in our paper gives the
existence of PPV extensions in many other new situations important
for applications. For instance, in the case of the incomplete
Gamma function, if one used the differential kernels approach, one
would have to take the algebraic closure $\overline{\Q(t)}$ instead
of just $\Q(t)$.

Now we give more details about our method. To apply the Tannakian
approach in the case of parameterized linear differential equations,
one needs to develop a theory of differential Tannakian categories
over differential fields. For this, one needs first to describe what
a differential abelian tensor category is,
Definition~\ref{defin-intdifftens}. In other words, one needs to
define what  it means for a Lie ring of derivations of a field~$k$
to act on an abelian $k$-linear category. The main subtlety here is
that one cannot ``subtract'' functors in order to give a
straightforward definition. There are two ingredients needed to
overcome this difficulty. First, one uses the equivalence
established by Illusie~\cite{Illusie} between complete formal Hopf
algebroids and differential rings, Section~\ref{subsection-illusie}.
Then, one uses the formalism of the extension of scalars for
categories, Section~\ref{subsection-prelextscal}
and~\cite{Gaits},~\cite{Stalder},
in order to define the
action of a complete formal Hopf algebroid over $k$ on an abelian
$k$-linear category.  This leads to the notion of a differential
category. For example, the category of all modules over a
differential ring is a differential category. In this case, the
differential structure is given by the Atiyah
extension~\cite{AtiyahC}.

The approach to differential categories via the action of Hopf
algebroids on categories can be generalized to many other
situations, including the difference case, when the corresponding
Hopf algebroid is given by the difference ring itself. For the
purposes of this paper, it is in fact enough to consider only the
degree two quotient of the formal Hopf algebroid. Having introduced
differential categories, one defines differential Tannakian
categories, Definition~\ref{defin-diffTancat}, and proves a
differential version of the Tannaka duality between differential
Hopf algebroids and differential Tannakian categories,
Proposition~\ref{prop-reprdiffalg} and
Theorem~\ref{theor-diffTanncorr}.

The main non-trivial example of a differential category in this
paper is the category formed by parameterized systems of linear
differential equations, Section~\ref{sec:paramAtiyah}. In this case,
the differential structure is given by what could be called a
parameterized Atiyah extension. Based on this construction, one
shows that the category of PPV extensions is equivalent to the
category of  differential fiber functors,
Theorem~\ref{theor-PPVdff}. Thus, the problem of constructing a PPV
extension is equivalent to the problem of constructing a
differential fiber functor. For the latter, we use a geometric
approach. The main technical difficulty here is to obtain
flatness of a certain differential algebra over a differential ring
after localizing this ring by a non-zero element. In general, this
seems to be unknown, however we prove this result in the special
case of a Hopf algebroid, Theorem~\ref{theor-defindiffHopfalg},
which is enough for our purpose. As an auxiliary result, we prove that a differentially finitely generated differential Hopf algebra is a quotient of the ring of differential polynomials by a differentially finitely generated ideal (one does not need to take a radical), Lemma~\ref{lemma-finiteHopf}. Besides, Theorem~\ref{theor-defindiffHopfalg} implies the existence of a differential fiber functor for a differential Tannakian category over a differentially closed field. Finally, using simple
algebro-geometric considerations, we construct a differential fiber
functor, thus, providing a PPV extension in the case of
Theorem~\ref{thm-main}.

The paper is organized as follows. We start by describing our main
results in the non-parameterized  case,
Section~\ref{sec:mainnonparam}, and the main parameterized case,
Section~\ref{sec:mainparam}. The proofs for the parameterized case
are postponed until Section~\ref{sec:proofsofthemainresults}. In the
intermediate sections, we develop our main technique as follows. In
Section~\ref{section-maindefs}, we fix most notation used in the
paper (Section~\ref{subsection-notation}) and introduce differential
rings, algebras, modules, PPV extensions, and jet-rings using the
invariant language convenient for the proofs of the main results. We
then recall facts about extensions of scalars for categories and
introduce differential abelian tensor categories and differential functors in
Section~\ref{sec:differentialcategories}. We use this to define
parameterized Atiyah extensions in Section~\ref{sec:paramAtiyah} and
prove in Theorem~\ref{theor-PPVdff} that the categories of PPV
extensions and differential functors are equivalent.
Section~\ref{sec:diffHopfalgebroids} contains the main technical
ingredient, Theorem~\ref{theor-defindiffHopfalg}, needed for the
proofs of the main results shown in
Section~\ref{sec:proofsofthemainresults}. Finally, in
Section~\ref{section-nonclosedconstants}, we discuss the
parameterized differential Galois correspondence for arbitrary
fields of constants and the behavior of the Galois group under the
extensions of constants (see also~\cite{ClaudineMichael}). For the
convenience of the reader, we finish by giving the necessary
background on Hopf algebroids and the usual Tannakian categories in
the appendix, 
Section~\ref{sec:appendix}.

\section{Statement of the main results}\label{sec:allmainresults}

\subsection{Non-parameterized case}\label{sec:mainnonparam}

Following P.\:Deligne \cite{DeligneFS}, let us recall how Tannakian
categories can be used to construct (non-parameterized)
Picard--Vessiot extensions for systems of linear differential
equations. For simplicity, we consider differential fields with only
one derivation and we use a more common notation $(K,\partial)$
instead of $(K,K\cdot\partial)$ as in
Definition~\ref{defin-intdiffrings}. So, let $(K,\partial)$ be a
differential field with a derivation $\partial$ and the field of
constants $k:=K^\partial$ of characteristic zero. 

A system of linear
$\partial$-differential equations over~$K$ is the same as a
finite-dimensional differential module $M$ over the differential
field $(K,\partial)$. A Picard--Vessiot extension for $M$ is a
differential field extension $(K,\partial)\subset (L,\partial)$
without new $\partial$-constants such that there is a basis of
horizontal vectors in~$L\otimes_K\nolinebreak M$ over~$L$ and $L$ is
generated by their coordinates in a basis of $M$ over $K$ (see also
Definition~\ref{defin-PV}).

\begin{definition}\label{defin-arsimp}
A field $k$ is {\it existentially closed} in a field $F$ over $k$
if, for any finitely generated subalgebra~$R$ in $F$ over $k$, there
exists a morphism of $k$-algebras $R\to k$ (see \cite[Proposition 3.1.1]{Ershov} for the equivalence with a more
standard definition).
\end{definition}

Note that, if $F=k(X)$ for an irreducible variety $X$ over $k$, then
$k$ is existentially closed in $F$ if and only if the set of
$k$-rational points is Zariski dense in $X$. In particular, $k$ is
existentially closed in $F$ in the following cases:
\begin{itemize}
\item
the field $k$ is algebraically closed and $F$ is any field over $k$;
\item
the field $k$ is pseudo algebraically closed and is algebraically closed in $F$;
\item
the field $F$ is a subfield in a purely transcendental extension of $k$;
\item
the field $F$ is real with $k$ being real closed (in this case, one applies the
Artin--Lang homomorphism theorem,~\cite[Theorem 4.1.2]{BoCo}).
\end{itemize}
Also, there is a range of non-trivial examples coming from
various special geometrical considerations.
In the case when~$K$ is real, $k$ is real closed and the differential Galois group is $\GL_n$, the following result is also proved in~\cite{Sowa} by explicit methods.

\begin{theorem}\label{theor-nonparmain}
Suppose that $k$ is existentially closed in $K$. Then, for any
finite-dimensional differential module $(M,\nabla_M)$ over~$(K,\partial)$, there exists a Picard--Vessiot extension.
\end{theorem}

The construction of a Picard--Vessiot extension is based on the theory
of Tannakian categories (Section~\ref{subsection-prelTann}) and uses the following two results from \cite{DeligneFS}.

\begin{proposition}\cite[Proof of
Corollaire~6.20]{DeligneFS}\label{prop-Deligne1}
Let $\Cat$ be a Tannakian category over a field $k$ such that $\Cat$
is tensor generated by one object and there is a fiber functor
$\Cat\to \Vect(K)$
for a field extension $K\supset k$. Then there exists a finitely generated subalgebra
$R$ in $K$ over $k$ and a fiber functor
$\Cat\to\Mod(R)$.
\end{proposition}

According to the notation of Section~\ref{subsection-prelTann}, $\langle M\rangle_\otimes$ is a full subcategory in the category of all differential
modules over~$(K,\partial)$ generated by subquotients of objects of type $M^{\otimes m}\otimes (M^{\vee})^{\otimes n}$. The following statement uses that  $\Char k = 0$, which implies that any algebraic group scheme over $k$ is smooth.

\begin{proposition}\cite[9.5, 9.6]{DeligneFS}\label{prop-Deligne2}
If there exists a fiber functor
\mbox{$\omega_0:\langle M\rangle_\otimes\to\Vect(k)$},
then there exists a Picard--Vessiot extension for $(M,\nabla_M)$.
\end{proposition}

\begin{proof}[Proof of Theorem~\ref{theor-nonparmain}]
We put $\Cat:=\langle M\rangle_{\otimes}$. By definition, the
category $\Cat$ is tensor generated by the object~$(M,\nabla_M)$.
Consider the fiber functor $\Cat\to\Vect(K)$ that forgets the
differential structure on a differential module over $(K,\partial)$.
By Proposition~\ref{prop-Deligne1}, there exist a finitely generated
subalgebra $R$ in~$K$ over~$k$ and a fiber functor
$\omega:\Cat\to\Mod(R)$. Since $k$ is existentially closed in~$K$,
there exists a homomorphism of $k$-algebras $R\to k$. As shown
\mbox{in~\cite[1.9]{DeligneFS}}, for any object $X$ in
$\Cat$, the $R$-module~$\omega(X)$ is finitely generated and
projective. Hence,
$$
\omega_0:\Cat\to\Vect(k),\quad X\mapsto
k\otimes_R\omega(X)
$$
is a fiber functor on $\Cat$. We conclude the proof by Proposition~\ref{prop-Deligne2}.
\end{proof}

The main goal of the present paper is to make a
parameterized analogue of the above reasoning.
As an application, we obtain a
construction of a parameterized Picard--Vessiot extension in
a range of cases when the constants are not differentially closed.

\subsection{Main results: parameterized case}\label{sec:mainparam}

The following is a parameterized analogue of
Theorem~\ref{theor-nonparmain}. We use notions and notation from Section~\ref{section-maindefs}.

\begin{theorem}\label{theor-main}
Let $(K,D_K)$ be a parameterized differential field
(Definition~\ref{defin-paramdifffield}) over a differential
field~$(k,D_k)$ (Definition~\ref{defin-intdiffrings}) with
\mbox{$\Char k = 0$}. Suppose that there is a splitting
$\widetilde{D}_k$ (Definition~\ref{defin-split}) of~$(K,D_K)$ over
$(k,D_k)$ such that $(k,D_k)$ is relatively differentially closed in
$\big(K,K\otimes_k \widetilde{D}_k\big)$
(Definition~\ref{defin-diffclosed}, Remark~\ref{rem-split}).

Then, for any finite-dimensional differential module (Definition~\ref{defin-diffmod}) over $(K,D_{K/k})$ (Definition~\ref{defin-paramdifffield}), there exists a parameterized Picard--Vessiot extension (Definition~\ref{defin-PPV}).
\end{theorem}

\begin{remark} The existence of a PPV extension implies the existence of a parameterized differential Galois group, which is a linear differential algebraic group, together with the Galois correspondence (Section~\ref{section-Galois}).
\end{remark}

\begin{remark}\label{rem-allow}
According to our definition of a parameterized differential field,
derivations from $D_k$ do not act on the field~$K$. Having the
splitting $\widetilde{D}_k$ from Theorem~\ref{theor-main}, we can
replace the differential field~$(k,D_k)$ with the differential field
${\left(k,\widetilde{D}_k\right)}$ so that derivations from
$\widetilde{D}_k$ act on $K$ (Remark~\ref{rem-split}). This allows
us to consider $\widetilde{D}_k$-Hopf algebroids of type $(K,H)$
over~$k$ and produce an analogue of the proof of
Theorem~\ref{theor-nonparmain}.
\end{remark}

Theorem~\ref{theor-main} is proved in Section~\ref{subsection-prooftheor-main}.
The following result describes two rather broad cases when the hypotheses of
Theorem~\ref{theor-main} are satisfied.

\begin{theorem}\label{thm-main}
Let $(K,D_K)$ be a parameterized differential field over a
differential field $(k,D_k)$ with $\Char k = 0$. Suppose that one of the following
conditions is satisfied:
\begin{enumerate}
\item\label{it:552}
There exists a splitting $\widetilde{D}_k$ of $(K,D_K)$ over $(k,D_k)$ such that
\begin{itemize}
\item the structure map \mbox{$D_K\to K\otimes_k
D_k$} induces an isomorphism between $\widetilde{D}_k$ and $1\otimes
D_k$,
\item The field $K$ is generated as a field by
$K_0:=K^{\widetilde{D}_k}$ and $k$,
\item the field $k_0 := k^{D_k}$ is existentially closed in $K_0$ (Definition~\ref{defin-arsimp}).
\end{itemize}
\item\label{it:551}
The field $k$ is existentially closed in $K$ and the map
\mbox{$D_{K/k}\to \Der_k(K,K)$} is an isomorphism.
\end{enumerate}
Then the parameterized differential field $(K,D_K)$ over $(k,D_k)$ satisfies the hypotheses of Theorem~\ref{theor-main}. Thus, for any finite-dimensional differential module over $(K,D_{K/k})$, there
exists a PPV extension.
\end{theorem}

Theorem~\ref{thm-main} is proved in Section~\ref{subsection-proofthm-main}.

\begin{remark}\label{rem-mainthm}
\hspace{0cm}
\begin{enumerate}
\item\label{rmk:333}
In general, fields generated by two subfields may have a
complicated structure. However, condition~\ref{it:552} in
Theorem~\ref{thm-main} implies that $K_0\otimes_{k_0}k$ is a domain
and $K=\Frac(K_0\otimes_{k_0}k)$. Indeed, by
Lemma~\ref{lemma-diffidealsR}, the dif\-fe\-ren\-tial algebra
$K_0\otimes_{k_0}k$ over $(k,D_k)$ is $D_k$-simple, that is, contains no $D_k$-ideals, whence the morphism
${K_0\otimes_{k_0}k\to K}$ is injective, which yields the required
statement.
\item\label{rmk:328}
Condition~\ref{it:551} in Theorem~\ref{thm-main} is equivalent to
requiring that $k$ be existentially closed in $K$, $\dim_K(D_{K/k})=\trdeg(K/k)$, and map the $D_{K/k}\to \Der(K,K)$ be injective.
\end{enumerate}
\end{remark}

Here is a series of examples that satisfy the hypotheses of Theorem~\ref{thm-main}.

\begin{example}
Let the bar over a field denote the algebraic closure. All fields
$K$ below are subfields in the algebraic closure of the field
$\C(x_1,\ldots,x_m,t_1,\ldots,t_n)$; all fields $k$ below are
subfields in the algebraic closure of the field
$\C(t_1,\ldots,t_n)$, and, except for Examples~\ref{ex:566} and~\ref{ex:584}, we put
$$
D_K:=K\cdot\partial_{x_1}+\ldots+K\cdot\partial_{x_m}+K\cdot\partial_{t_1}+
\ldots+K\cdot\partial_{t_n}, \quad D_k:=k\cdot
\partial_{t_1}+\ldots+k\cdot\partial_{t_n}.
$$
We obtain
\mbox{$D_{K/k}=K\cdot\partial_{x_1}+\ldots+K\cdot\partial_{x_m}$}.
In Examples~\ref{ex:372},~\ref{ex:566},~\ref{ex:373},~\ref{ex:584}, and~\ref{ex:374}, we put
$$\widetilde{D}_k:=k\cdot
\partial_{t_1}+\ldots+k\cdot\partial_{t_n}\subset D_K.$$ The
following parameterized differential fields $(K,D_K)$ over $(k,D_k)$
satisfy the hypotheses of Theorem~\ref{thm-main}:
\begin{enumerate}
\item\label{ex:372}
If $K=\Frac{\big(K_0\otimes_{\overline{\Q}} k\big)}$, where $K_0$ is a finite
extension of $\overline{\Q}(x_1,\ldots,x_m)$ and $k$ is an algebraic
extension of $\overline{\Q}(t_1,\ldots, t_n)$, then $(K,D_K)$
satisfies condition~\ref{it:552} with $k_0=\overline\Q$ being algebraically closed.
\item\label{ex:566}
If $K=\Frac{\big(K_0\otimes_{\overline{\Q}} k\big)}$, where $K_0$ is a finite
extension of $\overline{\Q}(x_1,x_2)$ and $k$ is an algebraic
extension of $\overline{\Q}(t_1,\ldots, t_n)$, then $(K,D_K)$
satisfies condition~\ref{it:552} with
$$
D_K:=K\cdot(\partial_{x_1}+x_2\partial_{x_2})+K\cdot\partial_{t_1}+
\ldots+K\cdot\partial_{t_n},\quad
D_{K/k}=K\cdot(\partial_{x_1}+x_2\partial_{x_2}),
$$
and with $k_0=\overline\Q$ being algebraically closed.
\item\label{ex:584}
If $K=k(x_1,\ldots,x_m)$, where $k$ is an algebraic extension of
$\Q(t_1,\ldots,t_n)$ such that $\Q$ is algebraically closed in $k$,
then $(K,D_K)$ satisfies condition~\ref{it:552} with
$K_0=\Q(x_1,\ldots,x_m)$, $k_0=\Q$.
\item\label{ex:374}
If $K=\Frac{\left(K_0\otimes_{\mathbb R} k\right)}$, where $K_0$ is
a finite extension of ${\mathbb R}(x_1,\ldots,x_m)$ such that $K_0$
a real field, and $k$ is an algebraic extension of ${\mathbb
R}(t_1,\ldots,t_n)$ such that $\mathbb R$ is algebraically closed in
$k$, then $(K,D_K)$ satisfies condition~\ref{it:552} with
$k_0={\mathbb R}$.
\item\label{ex:373}
If $K=\Frac{\left(K_0\otimes_{\mathbb R} k\right)}$, where $K_0$ is
a finite extension of ${\mathbb R}(x_1,\ldots,x_m)$ such that $K_0$
a real field, and $k$ is an algebraic extension of ${\mathbb
R}(t_1,t_2,t_3)$ such that $\mathbb R$ is algebraically closed in
$k$, then $(K,D_K)$ satisfies condition~\ref{it:552} with
\begin{align*}
&D_K:=K\cdot\partial_{x_1}+\ldots+K\cdot\partial_{x_m}+K\cdot{\big(t_1\partial_{t_1}+
\sqrt{2}t_2\partial_{t_2}+\sqrt{3}t_3\partial_{t_3}\big)},\\
&D_k:=k\cdot{\big(t_1\partial_{t_1}+
\sqrt{2}t_2\partial_{t_2}+\sqrt{3}t_3\partial_{t_3}\big)},
\end{align*}
and with $k_0={\mathbb R}$, \cite[Remark~4.9]{deSalas}.
\item\label{ex:371}
If $k$ is an algebraic closure of ${\mathbb Q}(t_1,\ldots,t_n)$ and
$K$ is a finite extension of $k(x_1,\ldots,x_m)$, then~$(K,D_K)$
satisfies condition~\ref{it:551}.
\item\label{ex:379}
If $k$ is a real closure of ${\mathbb Q}(t_1,\ldots,t_n)$ with
respect to some ordering and $K$ is a real finite extension of
$k(x_1,\ldots,x_m)$, then $(K,D_K)$ satisfies
condition~\ref{it:551}.
\end{enumerate}
\end{example}

\section{Differential rings and jet rings}\label{section-maindefs}

We do not claim any originality of most of the definitions and
constructions in this section, for example, see
\cite[Section~1.1]{Illusie}, \cite[\S1]{BeiBer},
\cite[9.9]{DeligneFS}, \cite{hubert03} for
Section~\ref{subsection-diffrings}, see any standard reference about
modules with connections for Section~\ref{subsection-diffmod}, see
\cite{PhyllisMichael} for Section~\ref{subsection-PPV},
\cite[Section~1.2,1.3]{Illusie}, \cite[\S2]{BerthelotOgus},
\cite[\S16]{EGAIV4}, \cite{RahimTom1,RahimTom2,Pillay} for
Section~\ref{subsection-defAtiyah} and
Section~\ref{subsection-diffobjexamps}, \cite{Illusie} for
Section~\ref{subsection-illusie}, and see any standard reference
about the Lie derivative for Section~\ref{subsection-Lieder}.
The definition of a differential object
(Definition~\ref{defin-diffobj}) generalizes the well-known notion
of a stratification on a sheaf, \cite{BerthelotOgus}. 

Only the
definition of a parameterized differential algebra
(Definition~\ref{defin-paramdifffield}) seems to be new. However,
we have decided to fix the notation and notions concerning
differential rings, differential modules over them, PPV extensions,
and jet rings. Note that the more commonly used name for the notion
from Definition~\ref{defin-intdiffrings} is a {\it Lie algebroid},
but we use the term {\it differential ring}, which seems to be more
standard in differential algebra. There is a direct generalization
of differential rings as defined below from rings to schemes
replacing modules by quasi-coherent sheaves.

\subsection{Notation}\label{subsection-notation}

First let us fix the notation that we use in the paper.

\begin{itemize}
\item
Given data $D$, we say that an object $O$ associated with $D$ is
{\it canonical} if its construction does not depend on the choice of
any additional structure on $D$ (for example, the choice of a basis
in a vector space). Usually, this implies that $O$ is functorial in
$D$ in the reasonable sense.
\item
All rings are assumed to be commutative and having a unit element.
\item
Denote the category of sets by $\Sets$.
\item
Given a non-zero element $f$ in a ring $R$, denote the localization of $R$ over
the multiplicative set formed by all natural powers of $f$ by $R_f$.
\item
Given two rings $R$ and $S$, denote their tensor product over $\Z$ by $R\otimes S$.
\item
Given a ring $R$ and two $R$-bimodules $M$ and $N$, their tensor product is denoted by $M\otimes_R
N$, where $M$ and~$N$ are considered with the
right and left $R$-module structures, respectively.
\item
For rings $R$ and $S$, denote the set of all derivations from $R$ to
$S$, that is, additive homomorphisms that satisfy the Leibniz rule,
by $\Der(R,S)$. If $R$ and $S$ are algebras over a ring $\kappa$,
denote the set of all $\kappa$-linear derivations from $R$ to $S$ by
$\Der_{\kappa}(R,S)$. Note that $\Der(R,S)$ and
$\Der_{\kappa}(R,S)$ have canonical $S$-module structures. Also,
$\Der(R,R)$ and $\Der_{\kappa}(R,R)$ are Lie rings.
\item
Given a ring homomorphism $R\to S$ and an $R$-module $M$, denote the
extension of scalars $S\otimes_R M$ also by $M_S$. If only one
$R$-module structure on $S$ is considered, we put the new scalars on
the left in the tensor product, that is, we use the notation
$S\otimes_R M$. If two $R$-module structures on $S$ are considered,
then we usually refer to them as right and left structures and use the
notations $S\otimes_R M$ and $M\otimes_R S$ for the corresponding
extensions of scalars.
\item
Given a ring homomorphism $R\to S$ and a morphism $f:M\to N$ of
$R$-modules, we denote 
the extension of scalars for $f$ from $R$ to $S$ by $\id_S\otimes f$, $S\otimes_R f$, or $f_S$, that is, we have
$$
S\otimes_R f:S\otimes_R M\to S\otimes_R N\quad\text{or}\quad
f_S:M_S\to N_S.
$$
\item
For a field $K$, denote the category of
vector spaces over $K$ by $\Vect(K)$. Denote the full subcategory of finite-dimensional $K$-vector spaces by~$\Vectf(K)$.
\item For a ring $R$, denote the category of
$R$-modules by $\Mod(R)$. Denote the full subcategory of finitely generated $R$-modules by~$\Modf(R)$.
\item
For a ring $R$, denote the category of $R$-algebras by $\Alg(R)$.
\item
For a Hopf algebra $A$ over a ring $R$, denote the category of
comodules over $A$ by $\Comod(A)$. Denote the full subcategory of
comodules over $A$ that are finitely generated as $R$-modules
by~$\Comodf(A)$.
\item
For an affine group scheme $G$ over a field~$k$, denote the
category of algebraic representations of $G$ over~$k$ by $\Rep(G)$
(they correspond to comodules over a Hopf algebra). Denote the full
subcategory of finite-dimensional representations of $G$ over $k$
by~$\Repf(G)$.
\item
Given a category $\Cat$ and objects $X$, $Y$ in $\Cat$, denote the set of morphisms from $X$ to $Y$ by $\Hom_{\Cat}(X,Y)$. Put
$
\End_{\Cat}(X):=\Hom_{\Cat}(X,X)$.
\item
Given exact sequences
\begin{align*}
\begin{CD}
0@>>> X@>{\alpha}>> Y@>{\beta}>> Z@>>> 0
\end{CD}\\
\begin{CD}
0@>>> X@>{\alpha'}>> Y'@>{\beta'}>> Z@>>> 0
\end{CD}
\end{align*}
in an abelian category, denote their Baer sum by $Y+_{\Bm} Y'$, that is, we have
$$
Y+_{\Bm} Y'=\Ker\big(\beta-\beta':Y\oplus Y'\to Z\big)\big/\Img\big(\alpha\oplus-\alpha':X\to Y\oplus Y'\big).
$$
\end{itemize}

\subsection{Differential rings}\label{subsection-diffrings}
\begin{definition}\label{defin-intdiffrings}
A {\it differential ring} is a triple $(R,D_R,\theta_R)$, where $R$
is a ring, $D_R$ is a finitely generated projective
$R$-mo\-du\-le together with a Lie bracket
$[\:\cdot\:,\cdot\:]:\nolinebreak D_R\times D_R\to D_R$, and
$\theta_R:D_R\to \Der(R,R)$ is a morphism of both $R$-modules and
Lie rings such that, for all $a\in R$ and $\partial_1,\partial_2\in
D_R$, we have
$$
[\partial_1,a\partial_2]-a[\partial_1,\partial_2]=\theta_R(\partial_1)(a)\,\partial_2.
$$
\end{definition}
For short, we usually omit $\theta_R$ in the notation. Thus, a
differential ring is denoted just by $(R,D_R)$, and
$$\partial(a):=\theta_R(\partial)(a)\quad a\in R,\ \partial\in
D_R.$$ Let~$R^{D_R}$ denote the subring of {\it $D_R$-constants},
that is, the set of all $a\in R$ such that, for any $\partial\in
D_R$, we have $\partial(a)=0$.

\begin{remark}
In most of the situations that we have here, it is enough to
consider differential rings $(R,D_R)$ with~$D_R$ being a finitely
generated free $R$-module.
\end{remark}

Recall that, for an $R$-module $M$, its second wedge power
$\wedge_R^2M$ is the quotient of $M\otimes_R M$ over the submodule
generated by all elements $m\otimes m$, where $m\in M$. Given
$m,n\in M$, the image of $m\otimes n$ under the natural map
$$M\otimes_R M\to \wedge^2_R M$$ is denoted by $m\wedge n$. There is a canonical morphism of $R$-modules
$$
\wedge^2_R{\left(M^\vee\right)}\to{\left(\wedge^2_R M\right)}^\vee\,,\quad p\wedge q\mapsto\left\{m\wedge n\mapsto p(m)q(n)-p(n)q(m)\right\},
$$
where $M^{\vee}:=\Hom_R(M,R)$. If $M$ is
finitely generated and projective, then~$\wedge_R^2M$ is also
finitely generated and projective and the above morphism $\wedge^2_R{\left(M^\vee\right)}\to {\left(\wedge^2_R M\right)}^\vee$ is an isomorphism.

\begin{definition}\label{defin-omega}
For a differential ring $(R,D_R)$, we put $\Omega_R:=D_R^{\vee}$ and
define additive maps
$$
\dd:R\to \Omega_R,
\quad
a\mapsto\{\partial\mapsto \partial(a)\}
$$
\begin{equation}\label{eq:dd}
\dd:\Omega_R\to {\wedge}_R^2\Omega_R,\quad
\omega\mapsto\{\partial_1\wedge\partial_2\mapsto
\partial_1(\omega(\partial_2))-\partial_2(\omega(\partial_1))-
\omega([\partial_1,\partial_2])\}
\end{equation}
for all $a\in R$, $\omega\in\Omega_R$ and $\partial_1,\partial_2\in
D_R$.
\end{definition}

In the notation of Definition~\ref{defin-omega}, for all $a,b\in R$ and $\omega\in\Omega_R$, we have $$\dd(ab)=a\,\dd b+b\,\dd a,\quad
\dd(a\omega)=a\,\dd\omega+\dd a\wedge\omega,\quad \text{and}\quad \dd(\dd(a))=0.$$

\begin{remark}\label{rmk-Ill1}
The map $\dd$ is well-defined for all wedge powers of $\Omega_R$,
$$
\dd:{\wedge}^i_R\Omega_R\to {\wedge}^{i+1}_R\Omega_R,
$$
and this defines a dg-ring structure
on~$\wedge^{\bullet}_R\Omega_R$. Actually, to define a differential
ring structure on $R$ with $D_R$ being a finitely generated
projective $R$-module is the same as to define a dg-ring structure
on~$\wedge^{\bullet}_R\Omega_R$ with the natural product structure
and grading, where, as above, $\Omega_R=\nolinebreak D_R^{\vee}$,
~\cite[Remarques~1.1.9\, b)]{Illusie}. Namely, given $\dd$, we
put $$\partial(a):= (\dd a)(\partial),$$ and define the Lie bracket
$[\partial_1,\partial_2]$ such that it satisfies the condition
$$
\omega([\partial_1,\partial_2])=\partial_1(\omega(\partial_2))
-\partial_2(\omega(\partial_1))-(\dd\omega)(\partial_1\wedge\partial_2)
$$
for all $a\in R$, $\partial,\partial_1,\partial_2\in D_R$, and $\omega\in\Omega_R$.
\end{remark}

\begin{example}\label{examp-diffring}
\hspace{0cm}
\begin{enumerate}
\item\label{en:650}
Let $R$ be the coordinate ring of a smooth affine variety $X$
over a field $k$ and put $D_R:=\Der_k(R,R)$. Then the pair $(R,D_R)$
is a differential ring with $\Omega_R$, $\wedge^2_R\Omega_R$, and
$\dd$ being the modules of differential $1$-, $2$-forms on $X$, and
the de Rham differential, respectively.
\item
Let $\partial_1,\ldots,\partial_n$ be formal symbols that denote commuting derivations
from a ring $R$ to itself (possibly, some of the $\partial_i$'s correspond to the zero derivation). Then the pair
$$
(R,R\cdot\partial_1\oplus\ldots \oplus R\cdot\partial_n)
$$
defines a differential ring.
\item
The data
$$(K,K\cdot(z\partial_x+\partial_y)+K\cdot\partial_z)$$ with
$K:=\C(x,y,z)$ and natural $\theta_K$ do not define a differential ring because of the lack of a Lie bracket.
\item
Let $\mathfrak g$ be a finite-dimensional Lie algebra over a field $K$. Then $(K,{\mathfrak g})$ is a differential field with the zero $\theta_K$.
\item
Let $R\hookrightarrow S$ be an embedding of rings and  $D_R$ be a
finitely generated projective $R$-sub\-mo\-dule and a Lie subring in
the $R$-module of all derivations $\partial:S\to S$ with
$\partial(R)\subset R$. Let $$\theta_R:D_R\to\nolinebreak\Der(R,R)$$
be defined by the restriction to $R$ of derivations from $S$ to
itself. Then~$(R,D_R,\theta_R)$ is a differential ring with,
possibly, non-trivial kernel and image of $\theta_R$.
\item\label{ex:455}
Let $(R,A)$ be a Hopf algebroid
(Section~\ref{subsection-prelHopfalg}). Put $$I:={\rm Ker}(e:A\to R)\quad\text{and}\quad \Omega_R:=I/I^2.$$ Then the cosimplicial ring structure on the
tensor powers of $A$ as an $R$-bimodule defines a dg-ring structure
on~$\wedge_R^{\bullet}\Omega_R$. Explicitly, for any $a\in R$, the
element $\dd a\in \Omega_R=I/I^2$ is the class of $$r(a)-l(a)\in I.$$
For any $\omega\in\Omega_R$, the element $$\dd \omega\in
\wedge_R^2\Omega_R$$ is defined as follows. Let $\tilde\omega\in I$
be such that its class in~$\Omega_R$ equals $\omega$. One takes the
class of the element
$$
\tilde\omega\otimes 1-\Delta(\tilde\omega)+1\otimes\tilde\omega\in I\otimes_R I
$$
in the quotient $\Omega_R\otimes_R\Omega_R$ and then one applies the
canonical map $$\Omega_R\otimes_R\Omega_R\to\wedge^2_R\Omega_R$$ to
obtain~$\dd\omega$. By Remark~\ref{rmk-Ill1}, the dg-ring structure
on $\Omega_R^{\bullet}$ defines a differential ring $(R,D_R)$ with
$D_R=\Omega_R^{\vee}$. See more details about this example
in~\cite[Proposition~1.2.8]{Illusie}.
\end{enumerate}
\end{example}

Note that, for any differential field $(K,D_K)$ with $\Char K= 0$
and injective \mbox{$\theta_K:D_K\to \Der(K,K)$}, there exists a
commuting basis for $D_K$ as shown in \cite[p.~12,
Proposition~6]{KolDAG} and Proposition~\ref{prop-commbasis}.
However, we prefer not to choose such a basis and to give
coordinate-free definitions and constructions. In particular, here
is a definition of a morphism between differential rings.

\begin{definition}\label{defin-morphdiffrings}
A {\it morphism between differential rings} $(R,D_R)\to(S,D_S)$ is a
pair $(\varphi,\varphi_*)$, where $\varphi:\nolinebreak R\to S$ is a
ring homomorphism and $ \varphi_*:\Omega_R\to \Omega_S $ is an
$R$-linear map such that $\varphi_*$ commutes with~$\dd$, that is,
for all $a\in R$, $\omega\in\Omega_R$, we have
$$\dd(\varphi(a))=\varphi_*(\dd a)\in\Omega_S\quad
\text{and}\quad\dd(\varphi_*(\omega))=\nolinebreak\varphi_*(\dd
\omega)\in\nolinebreak{\wedge}^2_S\Omega_S,$$ where we denote  the $R$-linear map
$${\wedge}^2_R\Omega_R\to{\wedge}^2_S\Omega_S$$ induced by
$\varphi_*$ for
short also by $\varphi_*$. The second condition,
$$\dd(\varphi_*(\omega))=\nolinebreak\varphi_*(\dd \omega),$$ is
called the {\it integrability condition}. For short, we sometimes omit $\varphi_*$ in the notation.
A morphism $(\varphi,\varphi_*)$ is {\it strict} if the $S$-linear morphism $$S\otimes_R\Omega_R\to \Omega_S$$ induced by $\varphi_*$ is an isomorphism.
\end{definition}

Taking the dual modules, one obtains an explicit definition of a morphism between differential rings in terms of derivations. The pair $(\varphi,\varphi_*)$ from Definition~\ref{defin-morphdiffrings} corresponds to a
pair $(\varphi,D_{\varphi})$, where $\varphi:\nolinebreak R\to\nolinebreak S$ is a ring homomorphism and
$$D_\varphi:D_S\to S\otimes_R D_R$$ is a morphism of $S$-modules. Sometimes we refer to $D_{\varphi}$ as a {\it structure map} associated with a morphism between differential rings. The first condition, $$\dd(\varphi(a))=\varphi_*(\dd a),$$ is equivalent to the equality
\begin{equation}\label{eq-diffmorph}
\partial(\varphi (a))=\sum _i
b_i\cdot\varphi(\partial_i(a))
\end{equation}
for all $a\in R$ and $\partial\in D_S$, where $$D_\varphi(\partial)=\sum _i b_i\otimes \partial_i,\quad b_i\in S,\ \partial_i\in D_R.$$ The integrability condition
is equivalent to the equality
\begin{equation}\label{eq-explicintegrmorph}
D_{\varphi}([\partial,\delta])=\sum _j \partial(c_j)\otimes\delta_j-
\sum _i \delta(b_i)\otimes\partial_i+
\sum _{i,j}b_ic_j\otimes {\left[\partial_i,\delta_j\right]}
\end{equation}
for all $\partial,\delta\in D_S$, where $$D_{\varphi}(\delta)=\sum _j c_j\otimes \delta_j,\quad c_j\in S,\ \delta_j\in D_R.$$ The morphism $(\varphi,\varphi_*)$ is strict if and only if~$D_{\varphi}$ is an isomorphism.

\begin{remark}\label{rmk-morphdiffring}
\hspace{0cm}
\begin{enumerate}
\item
In the notation of Definition~\ref{defin-morphdiffrings}, assume the injectivity of the canonical map $$S\otimes_R D_R\to\Der(R,S)$$ induced by the ring homomorphism $\varphi:R\to S$. Then it follows from~\eqref{eq-diffmorph} that the morphism $D_{\varphi}$, as well as $\varphi_*$, is unique if it exists. In particular, the above injectivity assumption holds if $R$ is a field and $\theta_R$ is injective.
\item\label{i:472}
In the notation of Definition~\ref{defin-morphdiffrings}, it follows from~\eqref{eq-explicintegrmorph} that  the $S$-submodule $D_{S/R}:=\nolinebreak{\rm Ker}(D_\varphi)$ in $D_S$ is closed under the Lie bracket, that is, if $D_\varphi(\partial)=D_\varphi(\delta)=0$, then
$$D_\varphi([\partial,\delta])=0.$$ Therefore, we obtain a differential ring $(S,D_{S/R})$ with the map $D_{S/R}\to\Der(S,S)$ induced by $\theta_S$.
\item
If $(R,\partial_R)$ and $(S,\partial_S)$ are two rings with
derivations, then a morphism of differential rings
$$(R,R\cdot\nolinebreak\partial_R)\to\nolinebreak (S,S\cdot\partial_S)$$
is given by a ring homomorphism $\varphi:R\to S$ and an element $b\in S$
such that, for any $a\in R$, we have
$$
\partial_S(\varphi (a))=b\cdot\varphi(\partial_R(a)).
$$
Thus, up to a rescaling, this is the usual definition of a morphism between
differential rings with one derivation.
\end{enumerate}
\end{remark}

\begin{example}\label{exampmorphdifffrings}
For a field $k$, consider the rings
$R:=k[x,y,z]$, $S:=k[x,y]$, the modules
$$
D_R:=R\cdot\partial_x+R\cdot\partial_y+R\cdot z\partial_z,\quad
D_S:=S\cdot\partial_x+S\cdot\partial_y,
$$
and the ring homomorphism $\varphi:R\to S$
being the quotient by the ideal $(z)\subset R$. Then we have
$$
\Omega_R=R\cdot \dd x+R\cdot\dd y+R\cdot(1/z)\,\dd z,
\quad
\Omega_S=S\cdot\dd x+S\cdot\dd y.
$$
Given polynomials $f,g\in S$, consider the morphism of $R$-modules
$$
\varphi_*:\Omega_R\to \Omega_S,
\quad
\dd x\mapsto \dd x,\quad \dd y\mapsto \dd y,\quad (1/z)\,\dd z\mapsto f\dd x+g\dd y.
$$
Then $(\varphi,\varphi_*)$ satisfies $$\varphi_*(\dd(a))=\dd(\varphi_*(a))$$
for all $a\in R$. Further, $(\varphi,\varphi_*)$ satisfies the integrability condition
if and only if $\partial_y f=\partial_x g$, because
$$
\dd{\left((1/z)\dd z\right)}=0,\quad \dd{\left(\varphi_*{\left((1/z)\dd z\right)}\right)}={\left(-\partial_y f+\partial_x g\right)}\cdot\dd x\wedge\dd y.
$$
\end{example}

\subsection{Differential algebras}

In the present paper, we consider several types of algebras over differential rings. The first type is the most general one.

\begin{definition}\label{defin-diffalg}
\hspace{0cm}
\begin{itemize}
\item
Given a morphism of differential rings $(R,D_R)\to (S,D_S)$, we say that $(S,D_S)$ is a {\it differential algebra} over $(R,D_R)$.
\item A~{\it morphism between differential algebras} over $(R,D_R)$
is a morphism between differential rings that commutes with the
given morphisms from $(R,D_R)$.
\end{itemize}
\end{definition}

\begin{definition}\label{defin-finitediff}
Given a differential ring $(S,D_S)$ and a morphism of rings $R\to S$, we say that $(S,D_S)$ is {\it differentially finitely generated} over $R$ if there are finite subsets
$\Sigma\subset S$ and $\Delta\subset D_S$ such that any element in $S$ can be represented as a polynomial with coefficients from $\Img(R\to S)$ in elements
of the form $$\left(\partial_{1}\cdot\ldots\cdot\partial_{n}\right)a,\ \ \partial_i\in \Delta,\ a\in\Sigma,$$ and the product stands for the composition of derivations.
\end{definition}

The following is a differential version of Definition~\ref{defin-arsimp}.

\begin{definition}\label{defin-diffclosed}
Let $(k,D_k)\to (K,D_K)$ be a morphism between differential fields. We say that $(k,D_k)$ is {\it relatively differentially closed} in $(K,D_K)$ if, for any differential subalgebra $(R,D_R)$ in $(K,D_K)$ over $(k,D_k)$ such that $(R,D_R)$ is differentially finitely generated over $k$ and the morphism $(R,D_R)\to (K,D_K)$ is strict, there is a morphism $(R,D_R)\to (k,D_k)$ of differential algebras over~$(k,D_k)$.
\end{definition}

The following type of algebras corresponds to the usual notion of a differential algebra.

\begin{definition}\label{defin-strictalg}
\hspace{0cm}
\begin{itemize}
\item
Given a strict morphism of differential rings $(R,D_R)\to (S,D_S)$, we say that $(S,D_S)$ is a $D_R$-{\it algebra} over~$(R,D_R)$ (or simply
over $R$).
\item
Denote the category of
$D_R$-algebras over $(R,D_R)$ by $\DAlg(R,D_R)$.
\item
If a $D_R$-algebra $(S,D_S)$ over a differential ring $(R,D_R)$ is differentially finitely generated over $R$, then we say that $S$ is
{\it $D_R$-finitely generated} over $R$.
\item
Denote the $D_R$-algebra freely $D_R$-generated over $R$ by the finite set $T_1,\ldots,T_n$, that is, the ring of $D_R$-polynomials in the differential indeterminates $T_1,\ldots,T_n$, by
$
R\{T_1,\ldots,T_n\}$.
\end{itemize}
\end{definition}

For short, we usually omit $D_S$ in the notation of a $D_R$-algebra over~$R$, because it is reconstructed by the isomorphism $$D_{\varphi}:D_S\stackrel{\sim}\longrightarrow S\otimes_R D_R.$$ Given $\partial\in D_R$ and $b\in S$, we put
$$
\partial(b):=\theta_S{\left(D_{\varphi}^{-1}(1\otimes\partial)\right)}(b).
$$
We have that $S$ is $D_R$-finitely generated if and only if
there is a finite subset $\Sigma\subset S$ such that any element in $S$
can be represented as a polynomial with coefficients from $\Img(R\to S)$
in elements of the form $$\left(\partial_{1}\ldots\partial_{n}\right)a,\quad \partial_i\in D_R,\ a\in\Sigma.$$ Equivalently, there is no smaller
$D_R$-subalgebra over $R$ in $S$ containing $\Sigma$.

\begin{definition}\label{defin-difffinpres}
A $D_R$-algebra $S$ over a differential ring $(R,D_R)$ is {\it of $D_R$-finite presentation over $R$} if there is an isomorphism of $D_R$-algebras over $R$
$$
S\cong R\{T_1,\ldots,T_n\}/I,
$$
where $I$ is a $D_R$-finitely generated ideal.
\end{definition}

The following type of algebras is needed to work with
parameterized differential equations.

\begin{definition}\label{defin-paramdifffield}
A differential algebra $(R,D_R)$ over a differential field $(k,D_k)$ is called {\it parameterized} if the structure map
$D_R\to R\otimes_k D_k$ is surjective and we have
$k=R^{D_{R/k}}$, where $D_{R/k}$ is the kernel of the structure map.
\end{definition}

Given a parameterized differential algebra $(R,D_R)$ over $(k,D_k)$,
one has the differential ring~$(R,D_{R/k})$
(Remark~\ref{rmk-morphdiffring}\eqref{i:472}).

\begin{definition}\label{defin-split}
A {\it splitting} of a parameterized differential algebra $(R,D_R)$
over a differential field~$(k,D_k)$ is a finite-dimensional
$k$-subspace $\widetilde{D}_k$ in $D_R$ closed under the Lie bracket
on $D_R$ such that the structure map $D_R\to R\otimes_k D_k$ induces
a surjection $$\widetilde{D}_k\to D_k\cong 1\otimes D_k.$$
\end{definition}

\begin{remark}\label{rem-split}
In the notation of Definition~\ref{defin-split}, put
$\widetilde{D}_R:=R\otimes_k \widetilde{D}_k$ and consider the
differential field~$\big(k,\widetilde{D}_k\big)$, where
$\widetilde{D}_k\to\Der(k,k)$ is defined as the composition
$$
\widetilde{D}_k\to D_k\to\Der(k,k).
$$
We obtain a commutative diagram of differential rings with the bottom horizontal morphism being strict:
$$
\begin{CD}
(k,D_k)@>>>(R,D_R)\\
@VVV@VVV\\
\big(k,\widetilde{D}_k\big)@>>>\big(R,\widetilde{D}_R\big).
\end{CD}
$$
\end{remark}

\begin{example}\label{exmp-commparamdf}
\hspace{0cm}
\begin{enumerate}
\item\label{ex:836}
Let
$$\big\{\partial_{x,1},\ldots\partial_{x,m},\tilde{\partial}_{t,1},\ldots,
\tilde{\partial}_{t,n}\big\}$$
be formal symbols that denote commuting derivations from a field $K$ to itself and let $k$ be the field of $\left\{\partial_{x,1},\ldots,\partial_{x,m}\right\}$-constants. Denote the restriction of $\tilde{\partial}_{t,i}$ from $K$ to $k$ by $\partial_{t,i}$, $1\Le i\Le n$. Then
$(K,D_K)$ is a parameterized differential field over $(k,D_k)$ with
\begin{align*}
&D_K:=K\cdot\partial_{x,1}\oplus\ldots\oplus
K\cdot\partial_{x,m}\oplus K\cdot\tilde{\partial}_{t,1}\oplus
\ldots\oplus K\cdot\tilde{\partial}_{t,n},\\
D_k&:=k\cdot\partial_{t,1}\oplus\ldots\oplus k\cdot\partial_{t,n},\quad
D_{K/k}=K\cdot\partial_{x,1}\oplus\ldots\oplus K\cdot\partial_{x,n}.
\end{align*}
\item
Let $\big(k,\widetilde{D}_k\big)$ be a differential field and 
$D_k$ be the image of the map
$$\theta_k:\widetilde{D}_k\to\Der(k,k).$$
Then~$\big(k,\widetilde{D}_k\big)$ is a parameterized
differential field over $(k,D_k)$.
\end{enumerate}
\end{example}

Actually, Example~\ref{exmp-commparamdf}\eqref{ex:836} is quite general as
the following statement shows.

\begin{proposition}\label{prop-commbasis}
Let $(K,D_K)$ be a parameterized differential field over a
differential field $(k,D_k)$ with $\Char k= 0$ and injective~$\theta_K$ and~$\theta_k$. Then we are in the case of
Example~\ref{exmp-commparamdf}\eqref{ex:836}, that is, there exists
a commuting basis
$$\big\{\partial_{x,1},\ldots\partial_{x,m},\tilde{\partial}_{t,1},\ldots,
\tilde{\partial}_{t,n}\big\}$$ of $D_K$ over $K$ such that
$$
D_k=k\cdot\partial_{t,1}+\ldots+k\cdot\partial_{t,n},\quad
D_{K/k}=K\cdot\partial_{x,1}+\ldots+K\cdot\partial_{x,n}\quad\text{where}\quad
\partial_{t,i}:=\tilde\partial_{t,i}|_k.
$$
\end{proposition}
\begin{proof}
We follow the idea of the proof of \cite[p.~12,
Proposition~6]{KolDAG}. First, there are sets of formal
variables~$\{x_{\alpha}\}$ and $\{t_{\beta}\}$ such that $K$ is an
algebraic extension of the field $k(\{x_{\alpha}\})$ and $k$ is an
algebraic extension of the field~$\Q(\{t_{\beta}\})$. Since $\Char
k= 0$, these algebraic extensions are separable, whence there are
uniquely defined commuting derivations
$\left\{\partial_{x_{\alpha}}\right\}$ and
$\left\{\partial_{t_{\beta}}\right\}$ from $K$ to itself. Note that
we have
$$
\Der_k(K,K)=\prod _{\alpha}K\cdot \partial_{x_{\alpha}},\quad
\Der(K,K)=\prod _{\alpha}K\cdot\partial_{x_{\alpha}}\oplus
\prod _{\beta}K\cdot\partial_{t_{\beta}}.
$$
In what follows, by a {\it coordinate subspace} in $\Der(K,K)$, we mean a product of
some of (possibly, infinitely many)
$\left(K\cdot\partial_{x_{\alpha}}\right)$'s
and~$\big(K\cdot\partial_{t_{\beta}}\big)$'s.

Since $\theta_K$ is injective, we can consider $D_{K/k}$ and $D_K$
as $K$-subspaces in $\Der_k(K,K)$ and $\Der(K,K)$, respectively. Let
$U\subset \Der_k(K,K)$ be a maximal coordinate subspace such that
$$U\cap D_{K/k}=0.$$ Explicitly,~$U$ is spanned by some of
$\partial_{x_{\alpha}}$'s in the sense of infinite products. Since
$D_{K/k}$ is a finite-dimensional $K$-vector space, the composition
$$
D_{K/k}\to\nolinebreak \Der_k(K,K)\to \Der_k(K,K)/U
$$
is an isomorphism of $K$-vector spaces (finite-dimensionality of
$D_{K/k}$ is important here, because we allow~$U$ to be only a
coordinate subspace in $\Der_k(K,K)$, not an arbitrary one).

Further, let $V\subset \Der(K,K)$ be a maximal coordinate subspace
such that $V\cap D_K=0$ and $V\supset U$. Explicitly, the basis of
$V$ in the sense of infinite products, as above, is obtained by
adding some $\partial_{t_{\beta}}$'s to the basis of~$U$. Since~$\theta_k$ is injective, we have
$$
D_{K/k}=\Der_k(K,K)\cap D_K\subset\Der(K,K).
$$
Together with the finite-dimensionality of $D_K$ over $K$, this
implies that the composition
$$
D_{K}\to \Der(K,K)\to \Der(K,K)/V
$$
is an isomorphism of $K$-vector spaces. Denote this isomorphism by
$\alpha$.

Let $\pi$ be the composition
$$
\begin{CD}
\Der(K,K)\to \Der(K,K)/V@>{\alpha^{-1}}>> D_K.
\end{CD}
$$
Since $V\cap \Der_k(K,K)=U$, we have the following commutative diagram with
injective vertical maps
$$
\begin{CD}
D_{K/k}@>>>\Der_k(K,K)@>>>\Der_k(K,K)/U\\
@VVV@VVV@VVV\\
D_K@>>>\Der(K,K)@>>>\Der(K,K)/V.
\end{CD}
$$
Therefore, $$\pi(\Der_k(K,K))\subset D_{K/k}.$$
Finally, consider the finite sets of all indices $\{\alpha_i\}$ and
$\{\beta_j\}$ such that the corresponding
derivations~$\partial_{x_{\alpha_i}}$ and~$\partial_{t_{\beta_j}}$
do not belong to $V$. Then the elements
$$\partial_{x,i}:=\pi{\big(\partial_{\alpha_i}\big)},
\tilde\partial_{t,j}:=\pi{\big(\partial_{\beta_j}\big)}$$ form a
basis in $D_K$. Since the subspaces~$U$ and $V$ are coordinate, this
is a commuting basis, as required.
\end{proof}

\subsection{Differential modules}\label{subsection-diffmod}

We define differential modules as follows.

\begin{definition}\label{defin-diffmod}
\hspace{0cm}
\begin{itemize}
\item
A {\it $D_R$-module} over a differential ring $(R,D_R)$ (or simply over $R$) is a
pair $(M,\nabla_M)$, where $M$ is an $R$-module and
$$\nabla_M:M\to \Omega_R\otimes_R M$$
is an additive map such that, for all $a\in R$ and $m\in M$, we
have $$\nabla_M(am)=\dd a\otimes m+a\cdot\nabla_M(m)$$ and the composition
$$
\begin{CD}
M@>\nabla_M>> \Omega_R\otimes_R M@>\nabla_M>> {\wedge}^2_R
\Omega_R\otimes_R M
\end{CD}
$$
is zero, where $\nabla_M:\Omega_R\otimes_R M\to
{\wedge}^2_R \Omega_R\otimes_R M$ is defined by
$$\nabla_M(\omega\otimes m):=\dd\omega\otimes m-\omega\wedge\nabla_M(m)$$
for all $m\in M$ and $\omega\in\Omega_R$.
\item The condition
$\nabla_M\circ\nabla_M=0$ is called the {\it integrability condition}.
\item We put
$M^{D_R}:=\Ker\nabla_M$.
\item A {\it morphism between $D_R$-modules}
\mbox{$\Psi:(M,\nabla_M)\to(N,\nabla_N)$} is a morphism of $R$-modules
\mbox{$\Psi:M\to N$} that commutes with~$\nabla$. For short, we sometimes
omit $\nabla_M$ in the notation. Denote the category of
$D_R$-modules over $R$ by $\DMod(R,D_R)$. Denote the full
subcategory of $D_R$-modules over $R$ that are finitely generated as
$R$-modules by~$\DModf(R,D_R)$.
\end{itemize}
\end{definition}

Equivalently, a $D_R$-module over a differential ring $(R,D_R)$ is a
pair $(M,\rho_M)$, where $M$ is an $R$-module and $\rho_M:D_R\to
\End_{\Z}(M)$ is an $R$-linear morphism of Lie rings such that for
all $\partial\in D_R$, $a\in R$, and $m\in M$, we have
$$\rho_M(\partial)(am)=a\cdot\rho_M(\partial)(m)+\partial(a)\cdot m.$$
Further, an $R$-linear map $\Psi:M\to N$ is a morphism of
differential modules if and only if, for all $m\in M$ and
$\partial\in D_R$, we have
$$\Psi\left(\rho_M(\partial)(m)\right)=\rho_N(\partial)\left(\Psi(m)\right).$$
We sometimes omit
$\rho_M$ and use just $\partial(m)$ to denote $\rho_M(\partial)(m)$.

\begin{remark}
If $M$ and $D_R$ are finitely generated free $R$-modules, then a choice of
bases in $D_R$ and $M$ over $R$ gives an equivalent
definition of the $D_R$-module structure on $M$ in terms of connection matrices.
\end{remark}

\begin{definition}\label{defin-tenspr}
Given $D_R$-modules $(M,\nabla_M)$ and $(N,\nabla_N)$ over $R$, the $D_R$-module structures on the
{\it tensor product} \mbox{$M\otimes_R N$} and on the {\it internal Hom module}
$\Hom_R(M,N)$ are defined by
$$
\nabla_{M\otimes N}(m\otimes
n):=m\otimes\nabla_N(n)+\nabla_M(m)\otimes n\in \Omega_R\otimes_RM\otimes_R N,
$$
$$
\nabla_{\Hom(M,N)}(\Psi)(m):= \nabla_N(\Psi(m))-\Psi(\nabla_M(m))\in
 \Omega_R\otimes_RN
$$
for all $m\in M$, $n\in N$, and $\Psi\in\Hom_R(M,N)$ (we
use the canonical isomorphism $\Omega_R\otimes_RM\otimes_R N\cong M\otimes_R\Omega_R\otimes_R N$   and write $\Psi$ instead of $\id_{\Omega_R}\otimes_R\Psi$ to
be short).
\end{definition}

Note that $$(M\otimes_R N,\nabla_{M\otimes N})\quad \text{and}\quad
(\Hom_R(M,N),\nabla_{\Hom(M,N)})$$ are well-defined as $D_R$-modules
over~$R$, namely, the integrability condition holds for them. The
tensor product on $D_R$-modules defines a tensor category structure
on $\DMod(R,D_R)$ with the internal Hom object being defined as
above (Section~\ref{subsection-prelTann}).

\begin{remark}\label{remark-diffalg}
A $D_R$-algebra $S$ over $R$ is the same as an $R$-algebra $S$ with a $D_R$-module structure $\nabla_S$ on $S$ over~$R$ such that the unit and multiplication maps are morphisms of $D_R$-modules over $R$. Given $D_R$-algebras $S$ and $T$, we obtain a $D_R$-algebra structure on $S\otimes_R T$ following Definition~\ref{defin-tenspr}.
\end{remark}

The extension of scalars for differential modules is defined as follows.

\begin{definition}\label{defin-extscaldiffmod}
Let
$\varphi:(R,D_R)\to(S,D_S)$ be a morphism of differential rings
and  $(M,\nabla_M)$ be a $D_R$-module over $R$. Then the
{\it extension of scalars of $(M,\nabla_M)$ from $(R,D_R)$ to
$(S,D_S)$} is the $D_S$-module
$$
\left(M_S:=S\otimes_R M,\nabla_{M_S}\right),$$
where, for all $m\in M$ and $a\in S$, we
have:
$$
\nabla_{M_S}(a\otimes m):=a\cdot (\varphi_*\otimes\id_M)(\nabla_M(m))+\dd a\otimes m\in \Omega_S\otimes_R M=\Omega_S\otimes_S M_S.
$$
\end{definition}

Equivalently, for all $\partial\in D_S$, $m\in M$, and $a\in S$, we have
$$
\rho_{M_S}(\partial)(a\otimes m):=\sum _i (a
b_i)\otimes\rho_M(\partial_i)(m)+\partial(a)\otimes m\in M_S,
$$
where $$D_\varphi(\partial)=\sum _i{b_i\otimes\partial_i},\quad b_i\in
S,\ \partial_i\in D_R.$$
Note that $(M_S,\nabla_{M_S})$ is well-defined as a $D_S$-module over $S$, namely, the integrability condition holds for it.

\begin{example}
In the notation of Example~\ref{exampmorphdifffrings}, consider
the rank one $R$-module $M=R\cdot e$ with the $D_R$-module structure over~$R$ defined by
$$
\nabla_M(e):=(1/z)\,\dd z\otimes e.
$$
Then the pair $(M_S,\nabla_{M_S})$, with
$$
\nabla_{M_S}(e)=\dd x\otimes fe+\dd y\otimes ge,
$$
satisfies the integrability condition if and only if $\partial_y f=\partial_x g$,
that is, if and only if $\varphi_*$ satisfies the integrability condition.
\end{example}

\subsection{Parameterized Picard--Vessiot extensions}\label{subsection-PPV}

First, let us give the definition of a non-parameterized Picard--Vessiot extension in
terms of differential fields as defined above.

\begin{definition}\label{defin-PV}
Let $(K,D_K)$ be a differential field and $M$ be a
finite-dimensional $D_K$-module over $K$. A {\it Picard--Vessiot
extension for $M$}, or, shortly, {\it a PV extension for $M$}, is a $D_K$-field $(L,D_L)$ over $K$ (in
particular, we have a field extension $K\subset L$ and $D_L\cong
L\otimes_K D_K$) such that the following conditions are satisfied:
\begin{enumerate}
\item We have
$
K^{D_{K}}=L^{D_{L}}$.
\item
There is a basis $\{m_1,\ldots,m_n\}$ of $M_L$ over $L$ such that all
the $m_i$'s belong to $(M_L)^{D_L}$.
\item
There is no smaller $D_K$-subfield over $K$ in $L$ containing
the coordinates of the $m_i$'s in a basis of $M$ over $K$.
\end{enumerate}
\end{definition}
In particular, the canonical morphism $$L\otimes_k M_L^{D_L}\to
M_L$$ is an isomorphism, where $k:=K^{D_K}=L^{D_L}$.

\begin{example}
Consider the differential field $(K,\mathfrak g)$ with zero $\theta_K$, where $\mathfrak g$ is a finite-dimensional Lie algebra over $K$.
Let $V$ be a finite-dimensional $\mathfrak g$-module over $K$, that is, $V$ is a finite-dimensional representation of $\mathfrak g$ over $K$. Let $G$ be the smallest algebraic subgroup in $\GL(V)$ such that its Lie algebra contains the image of the representation map
$\rho_V:{\mathfrak g}\to\nolinebreak {\mathfrak{gl}}(V)$. The field $L$ of
rational functions on $G$ is a $\mathfrak g$-field over $K$:
$\mathfrak g$ acts on $L$ by translation invariant vector fields on $G$ through $\rho_V$. The $\mathfrak g$-field $L$ is a Picard--Vessiot extension for $V$.
\end{example}

Let $(K,D_K)$ be a parameterized differential field over a
differential field $(k,D_k)$ and let $(L,D_L)$ be a
$D_K$-field over $K$. Then we obtain a morphism of
differential fields $$(k,D_k)\to (L,D_L)$$ as the composition of the
morphisms $(k,D_k)\to (K,D_K)$ and $(K,D_K)\to (L,D_L)$. The
isomorphism $D_L\cong L\otimes_K D_K$ induces an isomorphism
$$D_{L/k}\cong L\otimes_K D_{K/k},$$ where, as in Definition~\ref{defin-paramdifffield}, $$D_{L/k}:= \Ker (D_L\to L\otimes_k D_k).$$ Thus,
$(L,D_{L/k})$ is a $D_{K/k}$-field over $K$.

The following definition of a parameterized Picard--Vessiot
extension essentially repeats the corresponding definition
from~\cite{PhyllisMichael}.

\begin{definition}\label{defin-PPV}
Let $(K,D_K)$ be a parameterized differential field over a
differential field $(k,D_k)$ and~$M$ be a finite-dimensional
$D_{K/k}$-module over $K$.
\begin{itemize}
\item
A {\it parameterized Picard--Vessiot extension for $M$}, or,
shortly, {\it a PPV extension for $M$}, is a $D_K$-field~$(L,D_L)$
over $K$ such that the following conditions are satisfied:
\begin{enumerate}
\item
We have
$
K^{D_{K/k}}=L^{D_{L/k}}$.
\item
There is a basis $\{m_1,\ldots,m_n\}$ of $M_L$ over $L$ such that all
the $m_i$'s belong to $(M_L)^{D_{L/k}}$, where $M_L$ is a $D_{L/k}$-module over the $D_{L/k}$-field $L$ (see the discussion preceding the definition).
\item
There is no smaller $D_K$-subfield over $K$ in $L$ containing
the coordinates of the $m_i$'s in a basis of $M$ over $K$.
\end{enumerate}
\item
A {\it morphism between PPV extensions} is an isomorphism between
the corresponding $D_K$-fields over~$K$. Let $\PPV(M)$ denote the
category of all PPV extensions for $M$.
\end{itemize}
\end{definition}

Note that, in the notation of Definition~\ref{defin-PPV}, we have $L^{D_{L/k}}=k$, that is, $(L,D_L)$ is a
parameterized differential field over~$(k,D_k)$. If $\Char k= 0$ and
$(k,D_k)$  is differentially closed, then all PPV extensions for a
given $D_{K/k}$-module are isomorphic,~\cite[Theorem~3.5]{PhyllisMichael} (see examples of PPV
extensions therein).

In the case of Example~\ref{exmp-commparamdf}\eqref{ex:836},
Definition~\ref{defin-PPV} becomes equivalent to the definition of a
PPV extension as given in \cite{PhyllisMichael}. It makes sense to
consider PPV extensions, because they lead to a reasonable Galois
theory for integrable systems of differential equations with
parameters. Namely, as shown in~\cite{PhyllisMichael}, a PPV
extension defines a parameterized differential Galois group, which
is a differential algebraic group over~$(k,D_k)$.

 In addition, there
is a Galois correspondence between differential algebraic subgroups
and PPV subextensions, see Section~\ref{section-Galois} for the case
when $(k,D_k)$ is not necessarily differentially closed. To investigate the parameterized differential Galois theory, one also needs the notion of a PPV ring.
\begin{definition}\label{defin-PPVring}
Let $(K,D_K)$ be a parameterized differential field over a
differential field $(k,D_k)$, $M$ be a finite-dimensional
$D_{K/k}$-module over $K$, and $L$ be a PPV extension for $M$. Let $m_i\in M_L$ be as in  Definition~\ref{defin-PPV}. A {\it parameterized Picard--Vessiot ring associated with $L$} is a $D_K$-subalgebra in $L$ generated by the coordinates $a_{ij}$ of the $m_i$'s in a basis of $M$ over $K$ and the inverse of the determinant $1/\det(a_{ij})$.
\end{definition}

\subsection{Jet rings}\label{subsection-defAtiyah}

The proof of the main result, Theorem~\ref{theor-main}, requires an
appropriate notion of a differential Tannakian category over a
differential field that goes along with the notion of a differential
module. Since it seems not possible to give a direct analogue of
Definition~\ref{defin-diffmod} in a more general setting, one needs
another approach to differential modules. Actually, $D_R$-modules
over a differential ring $(R,D_R)$ turn out to be comodules over an
object similar to a Hopf algebroid, namely, the~$2$-jet ring $\Qr_R$
(Definition~\ref{defin-secondjet}). This approach has a natural
version with $R$-modules replaced by other objects over $R$, e.g.,
Hopf algebras over~$R$ or abelian $R$-linear tensor categories. The
latter leads to the notion of a differential object
(Definition~\ref{defin-diffobj}).

Let $(R,D_R)$ be a differential ring.

\begin{definition}\label{defin-firstjet}
A {\it $1$-jet ring} is the abelian group
$\P^{1}_R:=R\oplus\Omega_R$ with the following commutative ring
structure: $$a\cdot b=ab,\quad a\cdot
\omega=a\omega,\quad\text{and}\quad \omega\cdot\eta=0,\quad a\in R,\
\omega,\eta\in\Omega_R$$ (recall that $\Omega_R=D_R^\vee$\,).
\end{definition}

Consider two ring homomorphisms $l,r:R\to\P^{1}_R$ given by $$l(a):=a\quad \text{and}\quad r(a):=a+\dd a,\quad a\in\nolinebreak R.$$ Thus, $\P_R^{1}$ is an algebra over $R\otimes R$. Explicitly, for all
$a,b\in R$ and $\omega\in \Omega_R$, we have
$$
l(a)\cdot (b+\omega):=ab+a\omega\quad\text{and}\quad (b+\omega)\cdot
r(a):=ab+a\omega+b\,\dd a.
$$
It follows that $\Omega_R$ is an $(R\otimes R)$-ideal in $\P_R^1$.
The homomorphism $r:R\to \P^1_R$ provides a canonical right $R$-linear splitting $$\P^1_R\cong R\oplus\Omega_R,$$ which differs from the left $R$-linear splitting. It follows that $\P^{1}_R$ is a finitely generated projective
$R$-module with respect to both $R$-module structures.

\begin{definition}\label{defin-secondjet}
A {\it $2$-jet ring} $\Qr_R$ is the subset in
\mbox{$\P^{1}_R\otimes_R\P^{1}_R$} that consists of all elements
$$a\otimes 1+1\otimes\omega+\omega\otimes 1-\eta,$$ where $a\in R$,
$\omega\in\Omega_R$, and $\eta\in \Omega_R\otimes_R\Omega_R$ are
such that $\dd \omega$ equals the image of $\eta$ under the natural
map $$\Omega_R\otimes_R\Omega_R\to\wedge^2_R\Omega_R.$$ Put~$I_R$ to
be the set of elements in $\Qr_R$ with $a=0$.
\end{definition}

\begin{remark}
Note that, according to our notation, the tensor product
$\P_R^{1}\otimes_R\P_R^{1}$ involves both left and right
$R$-module structures on $\P_R^{1}$.
\end{remark}

\begin{example}
Let $(R,D_R)$ be as in
Example~\ref{examp-diffring}~\eqref{en:650}. Then
$$\P_R^1=(R\otimes_k R)/J^{2},$$ where $J$ is the kernel
of the multiplication homomorphism $R\otimes_k R\to R$. If $2$ is invertible in $R$, then $\P_R^2=(R\otimes_k R)/J^{3}$,~\cite{BerthelotOgus}.
\end{example}

Let us list some important properties of the $2$-jet ring. One can
show that $\Qr_R$ is an $(R\otimes R)$-subalgebra in
$\P^{1}_R\otimes_R\P^{1}_R$ with respect to the ``external'' $R$-modules structures. This defines two ring homomorphisms
from $R$ to~$\Qr_R$, which we denote also by~$l$ and~$r$.
Explicitly, we have $$l(a)=a\otimes 1\quad\text{and}\quad
r(a)=a\otimes 1+1\otimes \dd a+\dd a\otimes 1.$$ Denote  the natural
embedding by
\begin{equation}\label{eq:DeltaPR2}
\Delta:\Qr_R\to \P^{1}_R\otimes_R\P^{1}_R,
\end{equation}
and put also
$$
e:\P_R^{1}\to R,\quad a+\omega\mapsto a.
$$
Note that $\Delta$ and $e$ are morphisms of algebras over $R\otimes
R$. Both compositions
$$
\begin{xy}(0,15)*+{\Qr_R}="ab";(20,15)*+{\P_R^{1}\otimes_R \P_R^{1}}="a"; (45,15)*+{\P_R^{1}}="b";
{\ar@<0ex>"ab";"a"^<<<<<<{\Delta}};
{\ar@<0.7ex>"a";"b"^>>>>>>>>{e\cdot\id}};
{\ar@<-0.7ex>"a";"b"_>>>>>>>>{\id\cdot e}};
\end{xy}
$$
coincide with the surjective morphism of $(R\otimes R)$-algebras
\begin{equation}\label{eq:PR2PR1}
\Qr_R\to \P_R^{1},\quad (a\otimes 1+1\otimes\omega+\omega\otimes
1-\eta)\mapsto a+\omega.
\end{equation}
The kernel of this homomorphism equals the kernel of the natural map
$$
\Omega_R\otimes_R\Omega_R\to{\wedge}^2_R\Omega_R,
$$
which is, by definition, the second symmetric power $\Sym^2_R\Omega_R$ of
$\Omega_R$. Since $\Sym^2_R\Omega_R$ is a finitely generated
projective $R$-module,~$\Qr_R$ is a finitely generated projective
$R$-module with respect to both $R$-module structures, being an
extension of~$\P_R^{1}$ by $\Sym^2_R\Omega_R$. We also denote the map $\Qr_R\to R$ defined as the composition
$$\Qr_R\to\P^{1}_R\stackrel{e}\to R$$ by
$e$. Explicitly, we have
$$
e(a\otimes 1+1\otimes\omega+\omega\otimes 1-\eta)=a.
$$
Thus, we have $I_R=\Ker(e)$.

A morphism between differential rings
$\varphi:(R,D_R)\to (S,D_S)$ defines a homomorphism of
$(R\otimes R)$-algebras
$$
\left(\P^{1}_\varphi:=\varphi\oplus\varphi_*\right):\P^{1}_R\to
\P^{1}_S.
$$
The integrability condition for $\varphi$ is equivalent
to the fact that the ring homomorphism
$$
\P^{1}_\varphi\otimes\P^{1}_\varphi:\P^{1}_R\otimes_R\P^{1}_R\to
\P^{1}_S\otimes_S\P^{1}_S
$$
sends $\Qr_R$ to $\Qr_S$. Indeed, for any element $$1\otimes
\omega+\omega\otimes 1-\eta\in\Qr_R,$$ the element
$$
{\left(\P^{1}_\varphi\otimes\P^{1}_\varphi\right)}(1\otimes
\omega+\omega\otimes 1-\eta)=
1\otimes\varphi_*(\omega)+\varphi_*(\omega)\otimes 1-\varphi_*(\eta)
$$
belongs to $\Qr_S$ if and only if $\dd(\varphi_*(\omega))$ equals
the image of $\varphi_*(\eta)$ under the natural map
$$\Omega_S\otimes_S\Omega_S\to\wedge^2_S\Omega_S,$$ while the
latter coincides with $\varphi_*(\dd \omega)$. Thus, we obtain a
morphism of $(R\otimes R)$-algebras
$$
\Qr_{\varphi}:\Qr_R\to\Qr_S.
$$
One can show that $\Qr_{\varphi}$ commutes with the morphisms
$l,r,\Delta$, and $e$.

\begin{remark}
Assume that $2$ is invertible in $R$ (in~particular, $\Char
R\ne 2$). Then there is a section
\begin{equation}\label{eq:wedgetotensor}
{\wedge}^2_R\Omega_R\hookrightarrow \Omega_R\otimes_R\Omega_R,\quad
\omega_1\wedge\omega_2\mapsto\frac{1}{2}(\omega_1\otimes\omega_2-\omega_2\otimes\omega_1)
\end{equation}
of the natural quotient map $$\Omega_R\otimes_R\Omega_R\to \wedge^2_R\Omega_R,$$
and $\Qr_R$ is generated as a subring and a left (or right)
submodule in $\P_R^{1}\otimes_R\P_R^{1}$ by all elements of type
$$\langle\omega\rangle:=1\otimes\omega+\omega\otimes
1-\dd\omega,\quad \omega\in\Omega_R.$$ In addition, $I_R$ is
generated by $\langle\omega\rangle$ for all $\omega\in\Omega_R$ and
$\Sym^2_R\Omega_R=I_R\cdot I_R$.
\end{remark}

\subsection{Differential rings vs. Hopf algebroids}\label{subsection-illusie}

Let us cite  some relations between differential rings and Hopf algebroids from~\cite{Illusie}. The content of this section is not needed for the rest of the text, but we have decided to include it for the convenience of the reader.

Given a differential ring $(R,D_R)$, we have defined  a 2-jet ring $\Qr_R$ in Section~\ref{subsection-defAtiyah}. Actually, the construction depends only on the
{\it 2-truncated de Rham complex}
$$
\begin{CD}
R@>\dd>>\Omega_R@>\dd>>\wedge^2_R\Omega_R
\end{CD}
$$
associated with $(R,D_R)$ (Definition~\ref{defin-omega}).

Conversely, let ${\big(R,A^2\big)}$ satisfy the similar properties as ${\big(R,\Qr_R\big)}$ does. Namely, call a {\it $2$-truncated Hopf algebroid with divided powers} a pair of rings
$\big(R,A^{2}\big)$ ($A^2$ is just a notation for a ring) together with the following data: two ring homomorphisms
$l,r:R\to A^{2}$, a morphism of $(R\otimes R)$-algebras $e:A^{2}\to R$,
a map of sets $\gamma:I\to A^2$, where $I:=\Ker(e)$, and a morphism of $(R\otimes R)$-algebras
$$
\Delta:A^{2}\to A^{1}\otimes_R A^{1},
$$
where $$A^{1}:=A^{2}/I^{[2]},\quad I^{[2]}:=I\cdot I+\gamma(I).$$ We require that, for all $a\in R$, $x,y\in I$, we have
$$
\gamma(ax)=a^2\gamma(x),\quad \gamma(x+y)=\gamma(x)+xy+\gamma(y),\quad I\cdot I^{[2]}=0,$$ and the compositions
$$
\begin{xy}(0,15)*+{A^{2}}="ab";(20,15)*+{A^{1}\otimes_R A^{1}}="a"; (45,15)*+{A^{1}}="b";
{\ar@<0ex>"ab";"a"^<<<<<{\Delta}};
{\ar@<0.7ex>"a";"b"^>>>>>>>>{e\otimes\id}};
{\ar@<-0.7ex>"a";"b"_>>>>>>>>{\id\otimes e}};
\end{xy}
$$
are equal to the canonical surjection $A^{2}\to A^{1}$, where
$e$ also denotes the canonical morphism \mbox{$A^{1}\to R$}.
In particular, for~$\Qr_R$, we put
$$
\gamma(1\otimes\omega+\omega\otimes 1-\eta):=\omega\otimes\omega.
$$
It follows that there is an antipode map $A^2\to
{\big(A^2\big)}^s$ that satisfies the usual properties. An
analogous construction to the one from
Example~\ref{examp-diffring}\eqref{ex:455} provides a 2-truncated de
Rham complex associated with~${\left(R,A^2\right)}$. This implies
that the category of 2-truncated Hopf algebroids with divided powers
is equivalent to the category of 2-truncated de Rham complexes.

Further, as shown in Remark~\ref{rmk-Ill1}, there is a way to
construct a differential ring based on a 2-truncated de Rham complex
with finitely generated projective $\Omega_R$. The Jacobi identity
for the Lie bracket is equivalent to the vanishing of the
composition $$\dd\circ\dd:\Omega_R\to\wedge^3_R\Omega_R.$$ This gives
an auxiliary condition on the corresponding 2-truncated Hopf
algebroids with divided powers, which can be explicitly written in
terms of a certain ring $A^3$ (which is a 3-jet ring in the case of
$\Qr_R$),~\cite[1.3.5]{Illusie}. This condition is similar to
the associativity condition for a Hopf algebroid
(Section~\ref{subsection-prelHopfalg}). 

It follows
from~\cite[Proposition 1.2.8]{Illusie} that the category of
2-truncated Hopf algebroids with divided powers, with $I/I^{[2]}$
being a flat $R$-module, and with the associativity condition is
equivalent to the category of 2-truncated de Rham complexes with
$\Omega_R$ being a flat $R$-module and with vanishing
$$\dd\circ\dd:\Omega_R\to\wedge^3_R\Omega_R.$$ Also, the category of
2-truncated de Rham complexes with $\Omega_R$ being a finitely
generated projective $R$-module and with vanishing
$\dd\circ\dd:\Omega_R\to\wedge^3_R\Omega_R$ is equivalent to the
category of differential rings.

Recall that a {\it formal Hopf algebroid} is defined similarly to a
Hopf algebroid with $A$ being a pro-object in the category of
$(R\otimes R)$-algebras. A~formal Hopf algebroid $\big(R,\widehat A\big)$ with divided powers on $$I=\Ker\big(e:\widehat{A}\to R\big)$$ is
{\it complete} if the natural map $$\widehat A\to
``{\varprojlim}\mbox{''}A/I^{[i]}$$ is an isomorphism. It follows from
\cite[Th\'eor\`eme~1.3.6]{Illusie} that the category of
$2$-truncated Hopf algebroids with divided powers, with $I/I^{[2]}$ being a flat $R$-module, and with the associativity condition is equivalent to the category of
complete formal Hopf algebroids with divided powers and with $I/I^{[2]}$ being a flat $R$-module.

In particular, the category of differential rings over $\Q$ is
equivalent to the category of complete formal Hopf algebroids~$\big(R,\widehat A\big)$ with $R$ being a $\Q$-algebra and
$I/I^2$ being a finitely generated projective $R$-module. For
example, if $D_R=R\cdot\partial$ and~$R$ is a $\Q$-algebra, then,
for the corresponding complete formal Hopf algebroid
$\big(R,\widehat{A}\big)$, the ring~$\widehat{A}$ equals the ring
of formal Taylor series $R[[t]]$, and we have
$$
l(a)=a,\quad r(a)=\sum_{i=0}^{\infty}\partial^i(a)/i!,\quad
\text{and}\quad \Delta(t)=1\otimes t+t\otimes 1.
$$
In other words, the formal Hopf algebroid
$\big(R,\widehat{A}\big)$ is given by the action of the formal
additive group~$\widehat{\mathbb G}_a$ on~$\Spec(R)$.

It seems that, in general, formal Hopf algebroids that are complete
with respect to the usual powers $I^i$ correspond to iterative
Hasse--Schmidt derivations on the differential side.

Assume that all $D_R$-modules over $R$ that are finitely generated
over $R$ are projective $R$-mo\-dules (for example, this holds if
$R$ is a field). Then the category $\DModf(R,D_R)$ is a Tannakian
category with the forgetful fiber functor $$\DModf(R,D_R)\to\Mod(R).$$
It seems to be a non-trivial problem to give an explicit description
of the corresponding Hopf algebroid $(R,A)$ in terms of $D_R$, whose formal completion is the complete formal Hopf algebroid
associated with~$(R,D_R)$.

\subsection{Differential objects}\label{subsection-diffobj}

The pair $\big(R,\Qr_R\big)$
resembles a Hopf algebroid (Section~\ref{subsection-prelHopfalg}).
The main difference with a Hopf algebroid is that~$\Delta$ does not
send $\Qr_R$ to the tensor square of itself. However, one can define
a comodule over~$\big(R,\Qr_R\big)$ in the same way as one
defines a comodule over a Hopf algebroid. In the present paper, we
use a generalization of this notion.

Let $\Mc$ be a category {\it cofibred} over commutative rings, that
is, for each commutative ring $R$, there is a category~$\Mc(R)$ and,
given a ring homomorphism $R\to S$, there is a functor
$$S\otimes_R -:\Mc(R)\to\Mc(S),$$ called an {\it extension of
scalars}, compatible with taking composition of ring homomorphisms
(for more details, see~\cite{Grothendieck1959}).

\begin{example}
$\Mc(R)$ can be the category of $R$-modules, $R$-algebras, Hopf
algebras over $R$, Hopf algebroids over $R$, etc.
\end{example}

We will now define differential objects, generalizing
stratifications on sheaves from~\cite{BerthelotOgus}.

\begin{definition}\label{defin-diffobj}
\hspace{0cm}
\begin{itemize}
\item
A {\it $D_R$-object in $\Mc$} over $(R,D_R)$ (or simply over $R$) is
a pair $\big(X,\epsilon^{2}_X\big)$, where $X$ is an object
in~$\Mc(R)$ and
$$
\epsilon^{2}_X:X\otimes_R\Qr_R\to \Qr_R\otimes_R X
$$
is a morphism in the category $\Mc\big(\Qr_R\big)$ such that the following two conditions are satisfied. First, we have
$$
R\otimes_{\Qr_R}\epsilon^{2}_X=\id_X,
$$
where the $\Qr_R$-module structure on $R$ is defined by the ring homomorphism $e:\Qr_R\to R$.
Put
$$
{\left(\epsilon_X^{1}:=\P_R^{1}\otimes_{\P_R^{2}}\epsilon_X^{2}\right)}\colon X\otimes_R\P_R^1\to\P_R^1\otimes_R X,
$$
where the $\P_R^{2}$-mo\-dule structure on $\P_R^{1}$ is given by the
canonical surjection $\P_R^{2}\to \P_R^{1}$. The second condition says that
the composition of morphisms in $\Mc{\big(\P_R^1\otimes_R\P_R^1\big)}$
$$
\begin{CD}
X\otimes_R P_R^{1}\otimes_R P_R^{1}@>\epsilon_X^{1}\otimes_{P^1_R}{\left(P_R^1\otimes_R P_R^1\right)}>> P_R^{1}\otimes_R
X\otimes_R P_R^{1}@>{\left(\P_R^1\otimes_R\P_R^1\right)}\otimes_{P^1_R}\epsilon_X^{1}>>
P_R^{1}\otimes_R P_R^{1}\otimes_R X
\end{CD}
$$
is equal to the extension of scalars
$$
\left(P_R^{1}\otimes_R
P_R^{1}\right)\otimes_{\Qr_R}\epsilon^{2}_X,
$$
where the $\P_R^{2}$-mo\-dule structure on $\P_R^{1}\otimes_R \P_R^{1}$
is given by the ring homomorphism $\Delta$.
\item
The morphism
$\epsilon_X^{2}$ is called a {\it $D_R$-structure} on~$X$.
\item A {\it
morphism between $D_R$-objects} in~$\Mc$ over $(R,D_R)$ is a
morphism between objects in $\Mc(R)$ that commutes with~$\epsilon^{2}$.
\end{itemize}
\end{definition}

\begin{remark}
Perhaps, a more conceptually proper way to define a differential object would also involve the 3-jet ring to encode the associativity condition (Section~\ref{subsection-illusie}), but the present definition will be enough for our purposes. However, all examples that arise in the paper satisfy the associativity condition.
\end{remark}

\begin{definition}\label{defin-extrestr}
We say that a cofibred category $\Mc$ over rings has {\it
restrictions of scalars} if, for any ring homomorphism $R\to S$,
there is a functor $\Mc(S)\to\Mc(R)$, called a {\it restriction of
scalars}, that is right adjoint to the extension of scalars. We
usually denote the value of the restriction of scalars functor in
the same way as its argument.
\end{definition}

Thus, for all objects $X$ in $\Mc(R)$
and $Y$ in $\Mc(S)$, there is a functorial isomorphism
$$
\Hom_{\Mc(R)}(X,Y)\cong\Hom_{\Mc(S)}(S\otimes_R X,Y).
$$
Also, the restriction of scalars defines an object $S\otimes_R X$
in $\Mc(R)$, which is functorial in $S$ and $X$: given a
homomorphism of $R$-algebras $\varphi:S\to T$ and a morphism $f:X\to
Y$ in $\Mc(R)$, we have the morphism in~$\Mc(R)$
$$
\varphi\otimes
f:S\otimes_R X\to T\otimes_R Y.
$$
In particular, we have a canonical
morphism $X\to S\otimes_R X$ in $\Mc(R)$ given by the ring
homomorphism $R\to S$.

\begin{example}
The cofibred
categories of modules and algebras have restrictions of scalars,
while the cofibred categories of Hopf algebras and Hopf algebroids
do not have restrictions of scalars.
\end{example}

Given an object $X$ in $\Mc(R)$, by $_R\big( \P_R^{2}\otimes_R
X\big)$, denote  the object in $\Mc(R)$ defined as follows: first
one takes the extension of scalars $\P_R^2\otimes_R X$ with respect
to the right morphism $r:R\to \P_R^{2}$ and then applies the
restriction of scalars with respect to the left morphism $l:R\to
\P_R^{2}$. The proof of the following proposition is a direct
application of the adjointness between the extension and restriction
of scalars.

\begin{proposition}\label{prop-equivcomod2tr}
Suppose that a cofibred category $\Mc$ over rings has restrictions
of scalars. Then a $D_R$-object in $\Mc$ over~$(R,D_R)$ is the same
as a pair $\big(X,\phi_X^{2}\big)$, where $X$ is an object in $\Mc(R)$ and
$$\phi^{2}_X:X\to{}_R{\left(\P_R^{2}\otimes_R X\right)}$$ is a morphism
in $\Mc(R)$ such that
$$
(e\otimes\id_X)\circ\phi^{2}_X=\id_X
$$
and the following diagram
commutes in $\Mc(R)$:
$$
\begin{CD}
X@>\phi^{2}_X>>_R{\left(\P_R^{2}\otimes_R X\right)}\\
@V\phi^{1}_X VV@V\Delta\otimes\,\id_X VV\\
_R{\left(\P_R^{1}\otimes_R X\right)}@>\id_{\P}\otimes\phi^{1}_X>>_R{\left(\P_R^{1}\otimes_R \P_R^{1}\otimes_R X\right)},
\end{CD}
$$
where $\phi^{1}_X$ is the composition of $\phi_X^{2}$ with the morphism $$_R{\left(\P_R^{2}\otimes_R X\right)}\to
{}_R{\left(\P_R^{1}\otimes_R X\right)}.$$
\end{proposition}

In Section~\ref{subsection-explicitdefin}, we use the following
statement.

\begin{proposition}\label{prop-isomcomod2tr}
Suppose that a cofibred category $\Mc$ over rings has restrictions
of scalars. Then, for any $D_R$-object $X$ in~$\Mc$ over $R$, the
morphism $\epsilon^{1}_X$ in $\Mc{\big(\P_R^{1}\big)}$ is an
isomorphism.
\end{proposition}
\begin{proof}
The proof is similar to that for a Hopf algebra or a Hopf algebroid. The idea is that $\big(R,\P_R^1\big)$ corepresents a groupoid in the category of $R$-algebras with a two-step filtration, where the filtration on~$\P_R^1$ is given by $\P_R^1\supset\Omega_R$. More precisely, put
$$
\P_R^1\otimes_R^{1}\P_R^1:={\left(\P_R^1\otimes_R\P^1_R\right)}\big/
{\left(\Omega_R\otimes_R\Omega_R\right)},\quad
\imath:\P^1_R\to{\left(\P^1_R\right)}^s,\quad a+\omega\mapsto a-\omega,
$$
where, as in Section~\ref{subsection-prelHopfalg}, the superscript
$s$ denotes the interchange of the left and right $R$-module
structures. The homomorphism $$\Delta:\Qr_R\to\P_R^1\otimes_R\P_R^1$$
induces a homomorphism $$\P_R^1\to \P_R^1\otimes_R^{1}\P_R^1,$$  also  denoted by $\Delta$.

Let us construct an inverse to $\epsilon^1_X$ explicitly. Denote the composition in $\Mc(R)$
$$
\begin{CD}
X@>\phi_X^1>>_R{\left(\P_R^1\otimes_R X\right)}@>\imath\otimes\id_X>>{\left(X\otimes_R\P_R^1\right)}_R
\end{CD}
$$
by
$\psi$.
We will prove that $\epsilon^1_X\circ\psi$ equals the morphism $$X\to {\left(\P_R^1\otimes_R X\right)}_R$$ given by the ring homomorphism $r:R\to \P^1_R$. This would imply that $\epsilon^1_X$ is inverse to the morphism in $\Mc{\big(\P_R^1\big)}$ from $\P_R^1\otimes_R X$ to $X\otimes_R\P_R^1$ that corresponds by adjunction to $\psi$, thus, $\epsilon_X^1$ is an isomorphism.

By the adjunction relation between $\epsilon$ and $\phi$, the composition
$$
\begin{CD}
_R{\left(\P_R^1\otimes_R X\right)}@>\imath\otimes\id_X>>{\left(X\otimes_R\P_R^1\right)}_R
@>\epsilon_X^1>>{\left(\P_R^1\otimes_R X\right)}_R
\end{CD}
$$
is equal to the composition
$$
\begin{CD}
_R{\left(\P_R^1\otimes_R X\right)}@>{\id_{\P}\otimes\phi^1_X}>>
_R{\left(\P_R^1\otimes_R^{1}\P_R^1\otimes_R X\right)}@>\imath\cdot\id_{\P}\otimes\id_X>>{\left(\P_R^1\otimes_R X\right)}_R.
\end{CD}
$$
Since $X$ is a $D_R$-object, we have
$$
(\Delta\otimes\id_X)\circ \phi^2_X={\left(\id_P\otimes\phi^1_X\right)}\circ\phi_X^1:X\to{}_R{\left(\P_R^1\otimes_R\P_R^1\otimes_R X\right)}.
$$
Applying the ring homomorphism $$\P_R^1\otimes_R\P_R^1\to \P_R^1\otimes^1_R\P_R^1,$$ we obtain that both compositions
$$
\begin{xy}(0,15)*+{X}="ab";(20,15)*+{_R{\left(\P_R^1\otimes_R X\right)}}="a"; (60,15)*+{_R{\left(\P_R^1\otimes^1_R\P_R^1\otimes_R X\right)}}="b";
{\ar@<0ex>"ab";"a"^<<<<<{\phi^1_X}};
{\ar@<0.7ex>"a";"b"^>>>>>>>>>{\id_{\P}\otimes\phi_X^1}};
{\ar@<-0.7ex>"a";"b"_>>>>>>>>>{\Delta\otimes \id_X}};
\end{xy}
$$
are the same. Further, as for Hopf algebroids, we have
$$
(\imath\cdot\id_{\P})\circ\Delta=r\circ e:\P_R^1\to \P_R^1,
$$
where we consider $$\Delta:\P_R^1\to\P_R^1\otimes^1_R\P^1_R.$$
Finally, the composition $e\circ \phi_X^1:X\to X$ is the identity.
All together, this implies the needed statement about
$\epsilon_X^1\circ\psi$.
\end{proof}

\begin{remark}
It is not clear whether the morphism $\phi^2_X$ must be an isomorphism in the general case.
\end{remark}

\subsection{Examples of differential objects}\label{subsection-diffobjexamps}

Definition~\ref{defin-diffobj} is motivated by the following
statement.

\begin{proposition}\label{prop-diffmodjets}
Given an $R$-module $M$, a $D_R$-module structure on $M$ over $R$ is
the same as a $D_R$-structure on $M$ as an object in the cofibred
category of modules.
\end{proposition}
\begin{proof}
The cofibred category of modules has restrictions of scalars. Hence,
by Proposition~\ref{prop-equivcomod2tr}, a $D_R$-structure on~$M$ as
an object in the cofibred category of modules is given by an $R$-linear morphism
$$\phi^{2}_M:M\to{}_R{\left(\Qr_R\otimes_R M\right)}$$ that satisfies the
conditions therein.
Assume that $\nabla_M$ is a $D_R$-module structure on $M$. Consider the map
\begin{equation}\label{eq:phi1M}
\phi^{1}_M:M\to \P^{1}_R\otimes_R M,\quad m\mapsto 1\otimes m-\nabla_M(m).
\end{equation}
The Leibniz rule for $\nabla_M$ is equivalent to the left
$R$-linearity of $\phi^{1}_M$. Also, we have $$(e\otimes\id_M)\circ\phi_M^{1}=\id_M.$$
Note that the cokernel of the injective map
$$\Delta:\Qr_R\to\P_R^{1}\otimes_R\P_R^{1}$$ is a projective
$R$-module, being an extension of $\Omega_R$ by~$\wedge^2_R\Omega_R$. Therefore, the map
$$
\Delta\otimes\id_M:\Qr_R\otimes_R M\to \P^{1}_R\otimes_R\P^{1}_R\otimes_R M
$$
is injective. The integrability condition for $\nabla_M$ is
equivalent to the fact that the image of the composition
$$
\begin{CD}
M@>\phi^{1}_M>>\P^{1}_R\otimes_R M@>\id\otimes\phi^{1}_M>>
\P^{1}_R\otimes_R\P^{1}_R\otimes_R M
\end{CD}
$$
is contained in $\Qr_R\otimes_R M$. To see this, take any $m\in M$
and set $$\nabla_M(m)=\sum _i \omega_i\otimes m_i,\quad
\omega_i\in\Omega_R,\ m_i\in M.$$ Then the element
\begin{align*}
&{\left(\id\otimes\phi^{1}_M\right)}{\left(\phi^{1}_M(m)\right)}=
{\left(\id\otimes\phi^{1}_M\right)}{\left(1\otimes m-\sum _i \omega_i\otimes m_i\right)}=\\
&
=1\otimes 1\otimes m-\sum _i 1\otimes\omega_i\otimes m_i-\sum _i
\omega_i\otimes1\otimes m_i+\sum _i\omega_i\otimes\nabla_M(m_i)
\end{align*}
belongs to $\Qr_R\otimes_R M$ if and only if
$$\sum _i\dd\omega_i\otimes m_i=\sum _i\omega_i\wedge\nabla_M(m_i)\in
\wedge^2_R\Omega_R\otimes_R M.$$ Finally, put
$$
\phi_M^{2}:={\left(\id\otimes\phi^{1}_M\right)}\circ\phi^{1}_M
$$
to be the obtained map from $M$ to $\Qr_R\otimes_R M$.

Conversely, assume that $\phi_M^2$ is a $D_R$-structure on $M$.
Then the formula
$$
\nabla_M(m):=1\otimes m-\phi^{1}_M(m)
$$
defines a $D_R$-module structure on $M$ over $R$.
\end{proof}

\begin{example}\label{example-diffobjects}
\hspace{0cm}
\begin{enumerate}
\item\label{en:1646}
A $D_R$-object over $R$ in the cofibred category of algebras is the same as a $D_R$-algebra over~$R$.
\item\label{examp-diffHopfalgebra}
A $D_R$-Hopf algebra over $R$ is a Hopf
algebra $A$ over $R$ such that $A$ is a $D_R$-algebra over $R$ and
the coproduct, the counit, and the antipode maps are
morphisms of $D_R$-algebras.
\item\label{en:6104}
Given a differential ring $(\kappa,D_\kappa)$, a $D_\kappa$-Hopf
algebroid over $\kappa$ is a Hopf algebroid $(R,A)$ over~$\kappa$ such that $R$ and $A$ are $D_\kappa$-algebras over $\kappa$
and  $(l,r,\Delta,e,\imath)$ are morphisms of $D_{\kappa}$-algebras over
$\kappa$.
\end{enumerate}
\end{example}

Here is an application of this approach to differential structures.

\begin{proposition}\label{prop-Hopf}
Let $A$ be a $D_R$-algebra over $R$ such that $A$ is also a Hopf algebra over $R$.
Suppose that the coproduct map is a morphism of $D_R$-algebras over $R$.
Then the counit and antipode maps are also morphisms of $D_R$-algebras over $R$, that is, $A$ is a $D_R$-Hopf algebra over $R$.
\end{proposition}
\begin{proof}
Since the coproduct map is differential, the morphism
$$\epsilon^{2}_A:\Qr_R\otimes_R A\to A\otimes_R\Qr_R$$ commutes with
the coproduct maps in the corresponding Hopf algebras
$\Qr_R\otimes_R A$ and $A\otimes_R\Qr_R$ over $\P_R$. Therefore, it
commutes with the counit and the antipode maps (for example, see
\cite[Section~2.1]{Water}).
\end{proof}

In Section~\ref{subsection-defDiffTann} we use the following statement.

\begin{lemma}\label{lemma-diffHopfalg}
Let $(R,A)$ be a $D_\kappa$-Hopf algebroid over a differential ring $(\kappa,D_\kappa)$.
Then the composition of the isomorphisms of abelian groups
$$
A\otimes_R\Qr_R\stackrel{\sim}\longrightarrow \Qr_A\stackrel{\sim}\longrightarrow \Qr_R\otimes_R A
$$
is an isomorphism of $R$-bimodules.
\end{lemma}
\begin{proof}
Let $\varphi:R\to A$ denote the left homomorphism.
The left $R$-module structure on $A\otimes_R\Qr_R$ corresponds to the $R$-module structure on $\Qr_A$ given by the composition $$R\stackrel{\varphi}\longrightarrow A \stackrel{l}\longrightarrow\Qr_A.$$ The left $R$-module structure on $\Qr_R\otimes_R A$ corresponds to the $R$-module structure on $\Qr_A$ given by the composition $$\begin{CD}R@>l>>\Qr_R@>\Qr_{\varphi}>>\Qr_A.\end{CD}$$ Since $\varphi$ is a morphism of differential rings, $\Qr_{\varphi}$ is a morphism of $R$-bimodules. In particular, the compositions above coincide. Therefore, the left $R$-modules structures on $A\otimes_R\Qr_R$ and $\Qr_R\otimes_R A$ are the same. The proof for the right $R$-module structures in analogous.
\end{proof}

\subsection{Lie derivative}\label{subsection-Lieder}

In Section~\ref{sec:paramAtiyah}, we use the Lie derivatives defined on jet rings. Let $(R,D_R)$ be a differential ring.

\begin{definition}
A {\it weak $D_R$-module} is an $R$-module $M$ together with a
morphism of Lie rings $$\rho_M:D_R\to\End_{\Z}(M)$$ that satisfies the
Leibniz rule with respect to the multiplication by scalars from~$R$
(thus, a $D_R$-module is a weak $D_R$-module such that $\rho_M$ is
$R$-linear). Morphisms between weak $D_R$-modules are defined
similarly to morphisms between $D_R$-modules. As with differential
modules, we sometimes omit $\rho_M$ and use just $\partial(m)$ to
denote $\rho_M(\partial)(m)$.
\end{definition}

\begin{remark}
As in Definition~\ref{defin-tenspr}, given two weak $D_R$-modules,
one can show that the Leibniz rule defines a weak $D_R$-structure on
their tensor product.
\end{remark}

\begin{definition}\label{def:Liederivative}
Given $\partial\in D_R$ and $\omega\in\Omega_R$, define the {\it Lie derivative} as follows:
$$
L_{\partial}(\omega):=\dd(\omega(\partial))+(\dd\omega)(\partial\wedge-)\in\Omega_R,
$$
where $(\dd\omega)(\partial\wedge-)$ denotes the element in $\Omega_R=D_R^{\vee}$ that sends any $\delta\in D_R$ to $(\dd\omega)(\partial\wedge\delta)$.
\end{definition}

The notation $L_{\partial}(\omega)$ instead of $\partial(\omega)$ avoids confusing the Lie derivative with the result of the pairing between~$D_R$ and~$\Omega_R$. It follows from the definition of $\dd\omega$ that, for any $\xi\in D_R$, we have
\begin{equation}\label{eq:Lie}
L_{\partial}(\omega)(\xi)=\partial(\omega(\xi))-\omega([\partial,\xi]).
\end{equation}
Also, one can show that, for any $a\in R$, we have
\begin{equation}\label{eq:weakLie}
L_{a\partial}(\omega)=aL_{\partial}(\omega)+\omega(\partial)\dd a.
\end{equation}
The Lie derivative defines a weak $D_R$-structure on $\Omega_R$. By linearity, we obtain a weak $D_R$-structure on \mbox{$\P_R^1\cong R\oplus\Omega_R$}:
$$
\partial(a+\omega):=\partial(a)+L_{\partial}(\omega),\quad a\in R,\,\omega\in\Omega_R\,,\partial\in D_R.
$$
It follows that $r:R\to \P_R^1$ is a morphism of weak $D_R$-modules. Since $$\dd(\partial(a))=L_{\partial}(\dd a)\quad \text{for all } a\in R,\ \partial\in D_R,$$ we have that $l:R\to\P_R^1$ is a morphism of weak $D_R$-modules. The Leibniz rule for $L_{\partial}$ on $\Omega_R$ implies that the multiplication morphism $$\P_R^1\otimes_R\P_R^1\to\P_R^1$$ is also a morphism of weak $D_R$-modules.

\begin{remark}\label{rem:Lie}
By the Leibniz rule, the Lie derivative extends to a weak
$D_R$-structure on $\wedge^2\Omega_R$, which we also denote by
$L_{\partial}$. One can show that $L_{\partial}$ commutes with the
map $\dd:\Omega_R\to\wedge_R^2\Omega_R$. This implies that the
subring $$\Qr_R\subset\P_R^1\otimes_R\P_R^1$$ is preserved under the
action of $D_R$ via the weak $D_R$-module structure
on~$\P_R^1\otimes_R\P_R^1$.
\end{remark}

\section{Differential categories}\label{sec:differentialcategories}

\subsection{Extension of scalars for abelian tensor categories}\label{subsection-prelextscal}

Our aim is to apply Definition~\ref{defin-diffobj} of a differential
object with $X$ being an abelian $R$-linear tensor category. For
this, we need to use extension of scalars for such categories
associated with homomorphisms of rings. Let us briefly describe
this. One can find more details, for example, 
in \cite[p.~155]{Deligne}, \cite[p.~407]{MilneMotives}, and more recent
papers~\cite{Gaits} and~\cite{Stalder}.

We use the terminology from Section~\ref{subsection-prelTann}. We
fix a commutative ring $R$, a commutative $R$-algebra $S$, and an
abelian $R$-linear tensor category $\Cat$. According to our
definitions, this means that, in particular, the tensor product in
$\Cat$ is right-exact and $R$-linear in both arguments.

\begin{definition}\label{defin-extcat}
The {\it extension of scalars of $\Cat$ from $R$ to $S$} is an abelian $S$-linear tensor category $S\otimes_R\Cat$ together with a right-exact $R$-linear tensor functor
$$
S\otimes_R-:\Cat\to S\otimes_R\Cat
$$
that is universal from the left among all such data, that is, for any abelian $S$-linear tensor category $\Dop$, taking the composition with $S\otimes_R-$ defines an equivalence of categories:
$$
\Fun_S^{r,\otimes}(S\otimes_R\Cat,\Dop)\stackrel{\sim}\longrightarrow
\Fun_R^{r,\otimes}(\Cat,\Dop),\quad F\mapsto F\circ(S\otimes_R-),
$$
where $\Fun_S^{r,\otimes}$ denotes the category of right-exact $S$-linear tensor functors (similarly, for $\Fun_R^{r,\otimes}$).
\end{definition}

We usually denote the extension of scalars just by $S\otimes_R \Cat$
(keeping in mind that we also fix the functor \mbox{$S\otimes_R-$}).
Let us describe some general properties of the extension of
scalars for categories. First, consider a homomorphism of
\mbox{$R$-algebras} $S\to T$ and assume that the extensions of
scalars $S\otimes_R\Cat$ and~$T\otimes_S(S\otimes_R\Cat)$ exist.
Then $T\otimes_S(S\otimes_R\Cat)$ is equivalent to the extension of
scalars $T\otimes_R\Cat$.

Further, the category $S\otimes_R\Cat$ is functorial in $S$ and
$\Cat$ in the following way. Let $\varphi:S\to T$ be a
homomorphism of $R$-algebras, $\Dop$ be an abelian $R$-linear tensor category, and  $F:\Cat\to\Dop$ be a right-exact $R$-linear tensor functor. Assume that both $S\otimes_R\Cat$ and $T\otimes_R\Dop$ exist. Then we obtain a right-exact
$S$-linear tensor functor
$$
\varphi\otimes F:S\otimes_R\Cat\to T\otimes_R\Dop
$$
defined by the universal property of $S\otimes_R\Cat$ applied to the right-exact $R$-linear tensor functor
$$
\begin{CD}
\Cat\stackrel{F}\longrightarrow \Dop@>T\otimes_R->>T\otimes_R\Dop.
\end{CD}
$$
The assignment $F\mapsto\varphi\otimes F$ is functorial in $F$.
If $\psi:T\to U$ is a homomorphism of $R$-algebras, $\Ec$ is an abelian $R$-linear tensor category, $G:\Dop\to \Ec$ is a right-exact $R$-linear tensor functor, and $U\otimes_R\Ec$ exists, then there is a canonical
isomorphism between tensor functors:
$$
(\psi\otimes G)\circ(\varphi\otimes F)\cong (\psi\circ \varphi)\otimes(G\circ F).
$$

Sometimes, we also denote $\id_S\otimes F$ by $S\otimes_R F$. Also, we have that $\varphi\otimes\id_{\Cat}=S\otimes_R-$ for $\varphi:R\to S$. We hope that this coincidence will not make any confusion.

In Definition~\ref{defin-Dkfunctor}, we will need a slight generalization of the previous functor $\varphi\otimes F$. Namely, let
$$
\begin{CD}
R@>>>S\\
@VVV @VV\varphi V\\
U@>>>T\\
\end{CD}
$$
be a commutative diagram of rings, $\Dop$ be an abelian $U$-linear tensor category, and  $F:\Cat\to\Dop$ be a right-exact $R$-linear tensor functor. Assume that both $S\otimes_R\Cat$ and $T\otimes_U\Dop$ exist. Then, similarly to the above, we obtain a right-exact $S$-linear tensor functor
$$
\varphi\otimes F:S\otimes_R\Cat\to T\otimes_{U}\Dop.
$$
If $(S\otimes_R U)\otimes_U\Dop$ exists, then we have
\begin{equation}\label{eq:explextscal}
(\varphi\otimes F)(X)=T\otimes_{(S\otimes_R U)}F(X).
\end{equation}

The following important result is proved in~\cite[Theorem~1.4.1]{Stalder} (see also~\cite[p.~155]{Deligne} and \cite[p.~407]{MilneMotives}).

\begin{theorem}\label{theor-extscal}
Let $\Cat$ be a Tannakian category over a field $k$ and $k\subset K$
be a field extension. Then there exists the extension of scalars
$K\otimes_k\Cat$.
\end{theorem}

Further, recall that an {\it $S$-module in $\Cat$} is a pair
$(X,\alpha_X)$, where $X$ is an object in $\Cat$ and
$$
\alpha_X:S\to \End_{\Cat}(X)
$$
is a
homomorphism of $R$-algebras. Morphisms between $S$-modules in $\Cat$ are naturally defined. Given an $R$-module~$M$ and an object $X$ in $\Cat$, define
$$
M\otimes_R X
$$
to be an object in $\Cat$ such that there is a functorial isomorphism of $R$-modules
\begin{equation}\label{eq:homtensor}
\Hom_{\Cat}(M\otimes_R X,Y)\cong \Hom_R(M,\Hom_{\Cat}(X,Y)).
\end{equation}
The object $M\otimes_R X$ is well-defined up to a unique
isomorphism if it exists. If an $R$-module $M$ is of finite presentation, that is, there is
a right-exact sequence of $R$-modules
$$
\begin{CD}
R^{\oplus m}@>\varphi>> R^{\oplus n}@>>> M@>>> 0,
\end{CD}
$$
then, $M\otimes_R X$ exists for any $X$. By~\eqref{eq:homtensor}, for an $S$-module $(X,\alpha_X)$ in $\Cat$, the homomorphism $\alpha_X$ defines a morphism
$$
a_X:S\otimes_R X\to X.
$$
The following result is extensively used in what follows. Its
proof can be found 
in~\cite[5.11]{DeligneFS},
where an equivalent approach to the extension of scalars for
categories is used (see also~\cite{Stalder} and~\cite{Gaits}).

\begin{proposition}\label{prop-extscalfinite}
Let $\Cat$ be an abelian $R$-linear tensor category.
Suppose that $S$ is of finite presentation as an $R$-module. Then the abelian $S$-linear tensor category of $S$-modules in $\Cat$ is equivalent to the extension of scalars $S\otimes_R\Cat$ and the functor~$S\otimes_R-$ sends $X$ to $S\otimes_R X$.
\end{proposition}

\begin{example}\label{examp-extscal}
If $S$ is of finite presentation as an $R$-module, then the extension of scalars category $S\otimes_R\Mod(R)$ is equivalent to the category $\Mod(S)$ and the functor $S\otimes_R-$ coincides with the usual tensor product functor.
\end{example}

\begin{example}\label{examp-tensBaer}
Let $M$ be an $R$-module of finite presentation. Put $S:=R\oplus M$,
where an $R$-algebra structure on~$S$ is uniquely defined by the
condition $M\cdot M=0$. An $S$-module in $\Cat$ is the same as an
exact sequence $$0\to X'\to X\to X''\to 0$$ together with a morphism
$$
f_X:M\otimes_R X''\to X'.
$$
Namely, with an $S$-module $(X,\alpha_X)$, we associate
$X':=M\cdot X$ and $X'':=X/(M\cdot X)$, where $M\cdot X$ is the image of the morphism
$$
a_X:M\otimes_R X\to X.
$$
If $S$-modules $(X,\alpha_X)$ and $(Y,\alpha_Y)$ correspond to the data
\begin{align*}
&0\to X'\to X\to X''\to 0,\quad f_X:M\otimes_R X''\to X';\\
&0\to Y'\to Y\to Y''\to 0,\quad f_Y:M\otimes_R Y''\to Y',
\end{align*}
then their tensor product $(X,\alpha_X)\otimes (Y,\alpha_Y)$ in
$S\otimes_R\Cat$ is defined as the cokernel of the morphism
$$
b_X\otimes \id_Y-\id_X\otimes b_Y:M\otimes_R(X\otimes Y)\to X\otimes
Y,
$$
where $b_X$ is defined as the composition
$$
\begin{CD}
M\otimes_R X@>>> M\otimes_R X''@>f_X>> X'@>>> X,
\end{CD}
$$
and similarly for $b_Y$. In particular, if the tensor product in
$\Cat$ is exact in both arguments and the morphisms~$f_X$, $f_Y$ are
isomorphisms, then (see also ~\cite[5.1.3]{Moshe}) the tensor product
$(X,\alpha_X)\otimes(Y,\alpha_Y)$ corresponds to the Baer sum of the exact
sequences
\begin{align*}
&0\to M\otimes_R (X''\otimes Y'')\to X\otimes Y''\to X''\otimes Y''
\to 0,\\
&0\to X''\otimes (M\otimes_R Y'')\to X''\otimes Y\to X''\otimes
Y''\to 0.
\end{align*}
\end{example}

\subsection{Differential abelian tensor categories}\label{subsection-defDiffTann}

Throughout this subsection, we fix a differential ring $(R,D_R)$. We use constructions from Sections~\ref{subsection-diffobj} and~\ref{subsection-prelextscal}.
Recall that the jet rings $\P_R^1$, $\Qr_R$, and $\P_R^1\otimes_R\P_R^1$
(Definition~\ref{defin-secondjet}) are finitely generated
projective $R$-modules with respect to both left and right $R$-module structures.
Hence, they are of finite presentation as $R$-modules and there is an extension of
scalars from $R$ to $\Qr_R$ for abelian tensor categories (Definition~\ref{defin-extcat} and Proposition~\ref{prop-extscalfinite}).

Consequently, Definition~\ref{defin-diffobj} gives the notion of a
$D_R$-object over $R$ in the cofibred $2$-category of abelian tensor
categories, or a {\it $D_R$-category} over $R$ for short. Here
``morphisms'' between tensor categories are tensor functors. The
main difference with the case of a usual cofibred category as in
Definition~\ref{defin-diffobj} is that, instead of considering
equalities between morphisms, one should fix isomorphisms between
tensor functors.

Further, there are also restrictions of scalars between $\Qr_R$ and
$R$ for abelian tensor categories (this follows from the definition
of the extension of scalars for categories,
Definition~\ref{defin-extcat}). Proposition~\ref{prop-equivcomod2tr}
remains valid in the case of a cofibred $2$-category instead of a
cofibred \mbox{($1$-)category}. Thus, one has an equivalent
definition of a $D_R$-category over $R$ in terms of $\phi$'s instead
of~$\epsilon$'s. We prefer to use the definition in terms of
$\phi$'s. Note that Definitions~\ref{defin-intdifftens} and~\ref{defin-Dkfunctor} below have a more explicit equivalent form,
see Section~\ref{subsection-explicitdefin} (also, compare with
\cite[Example 12]{Gaits}, where the case of the coaction of a Hopf
algebra on a category is considered).

Similarly to
Section~\ref{subsection-diffobj},
$$_R{\left(\Qr_R\otimes_R\Cat\right)}$$ denotes the abelian tensor category
$\Qr_R\otimes_R\Cat$ considered with the $R$-linear structure
obtained by the left ring homomorphism $l:R\to \P_R^{2}$.

\begin{definition}\label{defin-intdifftens}
A {\it $D_R$-category} over $(R,D_R)$ (or simply over $R$) is a collection
${\left(\Cat,\phi^{2}_{\Cat},\Phi_{\Cat},\Psi_{\Cat}\right)}$, where~$\Cat$
is an abelian $R$-linear tensor category,
$$
\phi^{2}_{\Cat}:\Cat\to{}_R{\left(\Qr_R\otimes_R\Cat\right)}
$$
is a right-exact $R$-linear tensor functor,
$$
\Phi_{\Cat}:(e\otimes\id_{\Cat})\circ\phi^{2}_{\Cat}\stackrel{\sim}\longrightarrow
\id_{\Cat}
$$
is an isomorphism of tensor functors from $\Cat$ to itself
(recall that $e:\Qr_R\to R$ is a ring homomorphism), and
$$
\Psi_{\Cat}:(\Delta\otimes\,\id_{\Cat})\circ\phi^{2}_{\Cat}\stackrel{\sim}\longrightarrow
{\big(\id_{P^1_R}\otimes\phi^{1}_{\Cat}\big)}\circ\phi_{\Cat}^1
$$
is an isomorphism between tensor functors from $\Cat$ to
$\P_R^1\otimes_R\P_R^1\otimes_R\Cat$, where~$\phi^{1}_{\Cat}$ is the
composition of~$\phi_{\Cat}^{2}$ with the functor
$$\P_R^{2}\otimes_R{\Cat}\to \P_R^{1}\otimes_R {\Cat}.$$
For short, we usually denote a
$D_R$-category
${\left(\Cat,\phi^{2}_{\Cat},\Phi_{\Cat},\Psi_{\Cat}\right)}$
just by $\Cat$. We call the collection~$\left(\phi^2_{\Cat},\Phi_{\Cat},\Psi_{\Cat}\right)$ a
{\it $D_R$-structure on $\Cat$}.
\end{definition}

In other words, $\Phi_{\Cat}$ is an isomorphism between the composition
$$
\begin{CD}
\Cat@>\phi_{\Cat}^{2}>>\Qr_R\otimes_R\Cat@>e\otimes\id_{\Cat}>>\Cat
\end{CD}
$$
and the identity functor from $\Cat$ to itself, while the isomorphism $\Psi_{\Cat}$ makes the following diagram of categories to commute:
$$
\begin{CD}
\Cat@>\phi^{2}_{\Cat}>>\P_R^{2}\otimes_R {\Cat}\\
@V\phi^{1}_{\Cat} VV@V\Delta\otimes\,\id_{\Cat} VV\\
\P_R^{1}\otimes_R
{\Cat}@>\id_{P}\otimes\,\phi^{1}_{\Cat}>>\P_R^{1}\otimes_R
\P_R^{1}\otimes_R {\Cat}.
\end{CD}
$$

\begin{example}\label{examp-moddiffcat}
The category $\Mod(R)$ of $R$-modules has a canonical
$D_R$-structure given by the composition of $R$-linear tensor
functors (see also Example~\ref{examp-extscal})
$$
\begin{CD}\Mod(R)@>-\otimes_R\Qr_R>>_R\Mod{\left(\Qr_R\right)}
\cong{}_R{\left(\Qr_R\otimes_R\Mod(R)\right)}.\end{CD}
$$
Explicitly, for an \mbox{$R$-module~$M$}, we put
$$
\phi^2_R(M):={\left({\left(M\otimes_R\Qr_R\right)}_R,\alpha\right)}
$$
in $\Qr_R\otimes_R\Mod(R)$, where
$$
\alpha:\Qr_R\to\End_R{\left(M\otimes_R\Qr_R\right)}
$$
is the natural homomorphism. In other words, $\phi^2_R(M)$ is the
{\it Atiyah extension} of $M$ (see also
Proposition~\ref{prop-phi1expl} and
Remark~\ref{rem-expldefin}~\eqref{en:2845}).
\end{example}

\begin{example}\label{examp-comoddiffcat}
The $\kappa$-linear category $\Comod(R,A)$ of comodules over a
$D_{\kappa}$-Hopf algebroid $(R,A)$ over a differential ring~$(\kappa,D_\kappa)$
(Example~\ref{example-diffobjects}\eqref{en:6104}) has a canonical
$D_\kappa$-structure given by the composition of $\kappa$-linear tensor functors
\begin{align*}
&\begin{CD}
\Comod(R,A)@>-\otimes_\kappa
\Qr_{\kappa}>>_\kappa\Comod{\left(R\otimes_\kappa
\Qr_{\kappa},A\otimes_\kappa\Qr_{\kappa}\right)}\cong
{_\kappa\Comod{\left(\Qr_\kappa\otimes_\kappa
R,\Qr_\kappa\otimes_\kappa A\right)}}
\end{CD}\\
&\qquad\qquad\qquad\qquad\qquad
\cong{}_\kappa{\left(\Qr_\kappa\otimes_\kappa\Comod(R,A)\right)}.
\end{align*}
Explicitly, given a comodule $M$ over $A$, we define an $A$-comodule
structure on $\phi^2_R(M)$ as the composition
$$
\phi^2_R(M)\to \phi_R^2(M\otimes_R A)=\left(M\otimes_R
A\otimes_R\Qr_R\right)_R\cong \left(M\otimes_R\Qr_R\otimes_R
A\right)_R=\left(\phi^2_R(M)\otimes_R A\right)_R,
$$
where the non-trivial isomorphism in the middle is defined as in
Lemma~\ref{lemma-diffHopfalg}. Thus, the functor $\phi^2_R$ extends
to a $D_{\kappa}$-structure on the category $\Comod(R,A)$. If one
does an explicit calculation in the case when~$(\kappa,D_\kappa)$ is
a differential field with one derivation, $R=\kappa$, and $A$ is a
$D_{\kappa}$-Hopf algebra over $\kappa$, then one recovers the
formula from \cite[Theorem 1]{OvchRecoverGroup}.
\end{example}
We will define differential functors now.
\begin{definition}\label{defin-Dkfunctor}
\hspace{0cm}
\begin{itemize}
\item
Let $\varphi:(R,D_R)\to (S,D_S)$ be a morphism of differential
rings, $\Cat$ be a $D_R$-category over $R$, and $\Dop$ be a
$D_S$-category over $S$. A {\it differential functor} from $\Cat$ to
$\Dop$ is a pair $(F,\Pi_F)$, where $F:\Cat\to\Dop $ is a
right-exact $R$-linear tensor functor and
$$
\Pi_F:{\left(\Qr_{\varphi}\otimes F\right)}\circ \phi^2_{\Cat}\stackrel{\sim}\longrightarrow
\phi_{\Dop}^2\circ F
$$
is a isomorphism between tensor functors from $\Cat$ to
$\P_S^2\otimes_S\Dop$ such that $\Phi_{\Cat}$ commutes with
$\Phi_{\Dop}$ via~$\Pi_F$ and $\Psi_{\Cat}$ commutes with
$\Psi_{\Dop}$ via $\Pi_F$. For short, we usually denote a
differential functor~$(F,\Pi_F)$ just by~$F$. We call~$\Pi_F$ a {\it
differential structure} on $F$.
\item
A {\it morphism between
differential functors} is a morphism between tensor functors
$\Phi:F\to G$ that commutes with the~$\Pi$'s.
\end{itemize}
Denote the category of differential functors from $\Cat$ to $\Dop$
by $\Fun_R^{D}(\Cat,\Dop)$.
\end{definition}

In other words, the isomorphism $\Pi_F$ makes the following diagram
of categories commutative:
$$
\begin{CD}
\Cat@>F>>\Dop\\
@V\phi_{\Cat}^2VV@V\phi_{\Dop}^2VV\\
\Qr_R\otimes_R\Cat@>\Qr_{\varphi}\otimes F>>\Qr_S\otimes_S\Dop\,.
\end{CD}
$$

\begin{example}\label{examp-difffunctor}
\hspace{0cm}
\begin{enumerate}
\item\label{en:2371}
Given a morphism of differential rings $(R,D_R)\to(S,D_S)$,
the extension of
scalars functor
$$
S\otimes_R -:\Mod(R)\to\Mod(S)
$$
is canonically a differential functor.
\item\label{en:2370}
Given a $D_\kappa$-Hopf algebroid $(R,A)$ over a differential ring $(\kappa,D_\kappa)$, the forgetful functor
$$
\Comod(R,A)\to \Mod(R)
$$
is canonically a differential functor, where we consider the
$D_R$-structure on $\Mod(R)$ with \mbox{$D_R:=R\otimes_{\kappa}
D_{\kappa}$}.
\item\label{en:2377}
Given a $D_\kappa$-Hopf algebroid $(R,A)$ over a differential ring $(\kappa,D_\kappa)$ and a morphism of $D_\kappa$-algebras $R\to S$, the extension of scalars functor
$$
S\otimes_R-:\Comod(R,A)\to\Comod(S,{}_SA_S)
$$
is canonically a differential functor.
\end{enumerate}
\end{example}

The following statement is needed in the proof of Theorem~\ref{theor-PPVdff}.

\begin{lemma}\label{corol-composdifffunc}
Let $\Cat$, $\Dop$, and $\Ec$ be $D_R$-categories, $F:\Cat\to \Dop$
be a functor, and $G:\Dop\to \Ec$ be a fully faithful differential
functor. Then there is a bijection between differential structures
on $F$ and $G\circ F$.
\end{lemma}
\begin{proof}
If $F$ is a differential functor, then $G\circ F$ is also a
differential functor, being a composition of differential functors.
Conversely, suppose that $G\circ F$ is a differential functor. Consider the
diagram of categories:
$$
\begin{CD}
\Cat@>F>>\Dop@>G>>\Ec\\
@V\phi_{\Cat}^2VV@V\phi_{\Dop}^2VV@V\phi_{\Ec}^2VV\\
\Qr_R\otimes_R\Cat@>\id_P\otimes F>>\Qr_R\otimes_R\Dop@>\id_P\otimes G
>>\Qr_R\otimes_R \Ec.
\end{CD}
$$
Since $G\circ F$ is a differential functor, we obtain an isomorphism between tensor functors
$$
{\big(\id_{\P_R^2}\otimes G\big)}\circ{\big(\id_{\P_R^2}\otimes
F\big)}\circ\phi_{\Cat}^2\stackrel{\sim}\longrightarrow
\phi^2_{\Ec}\circ G\circ F.
$$
Further, since $G$ is a differential functor, we obtain an isomorphism between tensor functors
$$
\phi^2_{\Ec}\circ G\circ F\stackrel{\sim}\longrightarrow
{\big(\id_{\P_R^2}\otimes G\big)}\circ\phi_{\Dop}^2\circ F.
$$
Taking the composition, we obtain an isomorphism between tensor functors
$$
{\big(\id_{\P_R^2}\otimes G\big)}\circ{\big(\id_{\P_R^2}\otimes
F\big)}\circ\phi_{\Cat}^2\stackrel{\sim}\longrightarrow
{\big(\id_{\P_R^2}\otimes G\big)}\circ\phi_{\Dop}^2\circ F.
$$
Since $G$ is fully faithful, the functor $\id_{\P_R^2}\otimes G$ is
also fully faithful. Therefore, we obtain an isomorphism between
tensor functors
$$
\Pi_F:{\big(\id_{\P_R^2}\otimes
F\big)}\circ\phi_{\Cat}^2\stackrel{\sim}\longrightarrow
\phi_{\Dop}^2\circ F.
$$
It follows that this indeed defines a differential structure on $F$.
\end{proof}

We also use extensions of scalars for differential categories in the proof of Theorem~\ref{theor-PPVdff}.

\begin{proposition}\label{prop-diffextscal}
Let $\varphi:(R,D_R)\to (S,D_S)$ be a morphism between differential rings, $\Cat$ be a $D_R$-category over $R$, and suppose that the extension of scalars category $S\otimes_R\Cat$ exists (Definition~\ref{defin-extcat}).
\begin{enumerate}
\item\label{en:2368}
There is a canonical $D_S$-structure on $S\otimes_R \Cat$ such that the functor
$$
S\otimes_R-:\Cat\to S\otimes_R\Cat
$$
is canonically a differential functor.
\item\label{en:2375}
Let $\Dop$ be a $D_S$-category over $S$. Then taking the composition with $S\otimes_R-$ defines an equivalence of categories:
$$
\Fun_S^{D}(S\otimes_R\Cat,\Dop)\stackrel{\sim}\longrightarrow
\Fun_R^{D}(\Cat,\Dop),\quad F\mapsto F\circ(S\otimes_R-).
$$
\end{enumerate}
\end{proposition}
\begin{proof}
To prove statement~\ref{en:2368}, define the functor
$$
\phi^2_{S\otimes_R\Cat}:S\otimes_R\Cat\to \Qr_S\otimes_S(S\otimes_R\Cat)\cong \Qr_S\otimes_R\Cat
$$
by the universal property of $S\otimes_R\Cat$ applied to the right-exact $R$-linear tensor functor
$$
\begin{CD}\Cat\stackrel{\phi_{\Cat}^2}\longrightarrow\Qr_R\otimes_R\Cat
@>\Qr_{\varphi}\otimes\id_{\Cat}>>\Qr_S\otimes_R\Cat.
\end{CD}
$$
This also defines a differential structure on the functor
$S\otimes_R-$. To prove statement~\ref{en:2375}, one applies the
universal property of~$S\otimes_R\Cat$ directly.
\end{proof}

\begin{remark}\label{remark-extescalcatdiff}
\hspace{0cm}
\begin{enumerate}
\item
Applying Proposition~\ref{prop-diffextscal} to $\Cat=\Mod(R)$, we
obtain that the canonical functor $$S\otimes_R\Mod(R)\to \Mod(S)$$ is
a differential functor between $D_S$-categories
(Example~\ref{examp-moddiffcat}) provided that $S\otimes_R\Mod(R)$
exists.
\item\label{en:2422}
Given $D_R$-categories $\Cat$ and $\Dop$ over $R$ and a differential
functor $F:\Cat\to\Dop$, the functor $$S\otimes_R F:S\otimes_R\Cat\to
S\otimes_R\Dop$$ is canonically a differential functor between
$D_S$-categories over $S$ provided that~$S\otimes_R\Cat$ and
$S\otimes_R\Dop$ exist.
\item
Both Definition~\ref{defin-extscaldiffmod} and
the construction from Proposition~\ref{prop-diffextscal}\eqref{en:2368} are particular cases of extensions of scalars for differential objects.
\end{enumerate}
\end{remark}

\subsection{Definitions in the explicit form}\label{subsection-explicitdefin}

The following technical result provides an explicit information
about objects of type $\phi^2_{\Cat}(X)$, where $X$ is an object in
a $D_R$-category $\Cat$. Recall that we have a decreasing filtration
by ideals  in $\Qr_R$ (see Definition~\ref{defin-secondjet}
for~$I_R$)
$$
\Qr_R\supset I_R\supset \Sym_R^2\Omega_R\supset 0
$$
and canonical isomorphisms
$$
\Qr_R/I_R\cong R,\quad I_R/\Sym^2_R\Omega_R\cong \Omega_R.
$$

\begin{lemma}\label{lemma-isom}
Let $\Cat$ be a $D_R$-category over $R$. Then, for any object $X$ in
$\Cat$, there are functorial isomorphisms (see
Section~\ref{subsection-prelextscal} for $\Omega_R\otimes_R X$):
$$
\phi_{\Cat}^2(X)/I_R\cdot\phi_{\Cat}^2(X)\cong X,\ \ \
I_R\cdot\phi_{\Cat}^2(X)/\Sym^2_R\Omega_R\cdot\phi_{\Cat}^2(X)\cong
\Omega_R\otimes_R X,\ \ \
\Sym^2_R\Omega_R\cdot\phi_{\Cat}^2(X)\cong
\Sym^2_R\Omega_R\otimes_R X.
$$
\end{lemma}
\begin{proof}
The isomorphism $\Phi_{\Cat}$ yields the first isomorphism, because
$$
(e\otimes\id_{\Cat})\left(\phi^2_{\Cat}(X)\right)\cong
\phi_{\Cat}^2(X)\big/I_R\cdot\phi_{\Cat}^2(X).
$$
Hence, the $\Qr_R$-module structure on $\phi_{\Cat}^2(X)$ defines
surjective morphisms
$$
\Omega_R\otimes_R X\stackrel{\alpha}\longrightarrow
I_R\cdot\phi_{\Cat}^2(X)\big/\Sym^2_R\Omega_R\cdot\phi_{\Cat}^2(X),\quad
\Sym^2_R\Omega_R\otimes_R
X\stackrel{\beta}\longrightarrow\Sym^2_R\Omega_R\cdot\phi_{\Cat}^2(X),
$$
where we use that
$$
I_R\cdot I_R\subset \Sym^2_R\Omega_R\quad\text{and}\quad
I_R\cdot\Sym^2_R\Omega_R=0.
$$
Let us prove that $\alpha$ is injective and, thus, is an
isomorphism. By the definition of $\phi^1_{\Cat}$
(Definition~\ref{defin-intdifftens}), we have
$$
\phi^1_{\Cat}(X)\cong
\phi^2_{\Cat}(X)\big/\Sym^2_R\Omega_R\cdot\phi^2_{\Cat}(X).
$$
Thus, we need to show that the corresponding morphism
$$
\Omega_R\otimes_R X\stackrel{\gamma}\longrightarrow\phi_{\Cat}^1(X)
$$
is injective (see also Example~\ref{examp-tensBaer}).
By Proposition~\ref{prop-extscalfinite}, the right-exact $R$-linear
tensor functor
$$\phi_{\Cat}^1:\Cat\to{}_R{\left(\P_R^1\otimes_R\Cat\right)}$$
defines a right-exact $\P_R^1$-linear tensor functor
$$
\epsilon^1_{\Cat}:\Cat\otimes_R\P_R^1\to \P_R^1\otimes_R\Cat
$$
such that, for any object $X$ in $\Cat$, we have a functorial
isomorphism
$$
\phi^1_{\Cat}(X)\cong
\epsilon_{\Cat}^1{\left(X\otimes_R\P_R^1\right)}.
$$
Further, the proof of Proposition~\ref{prop-isomcomod2tr} remains
valid in the case of a cofibred 2-category instead of a cofibred
\mbox{(1-)category}. Thus, $\epsilon^1_{\Cat}$ is an equivalence of
categories and, in particular, is exact. On the other hand,
since~$\Phi_{\Cat}$ is an isomorphism, we have an isomorphism of
tensor functors $$R\otimes_{\P_R^1}\epsilon_{\Cat}^1\cong \id_{\Cat}$$
from~$\Cat$ to itself, where we consider the ring homomorphism
$e:\P_R^1\to R$. Explicitly, this means that, for a
$\P_R^1$-module~$Y$ in $\Cat$ such that $\Omega_R$ acts trivially on
$Y$, there is a functorial isomorphism $\epsilon_{\Cat}^1(Y)\cong Y$
(Proposition~\ref{prop-extscalfinite}). Therefore, applying
$\epsilon_{\Cat}^1$ to the injective morphism in
$\Cat\otimes_R\P_R^1$
$$
X\otimes_R\Omega_R\longrightarrow X\otimes_R\P_R^1
$$
given by the split embedding $\Omega_R\subset\P_R^1$ (and using that
$X\otimes_R\Omega_R=\Omega_R\otimes_R X$), we show the injectivity
of $\gamma$.
Now let us prove that $\beta$ is injective and, thus, it is an
isomorphism. Consider the object
$$
Z:={\big(\id_{\P_R^1}\otimes\phi_{\Cat}^1\big)}{\left(\phi^1_{\Cat}(X)\right)}
$$
in $\P_R^1\otimes_R\P_R^1\otimes_R\Cat$. We have a commutative
diagram in $\Cat$
$$
\begin{CD}
\Sym_R^2\Omega_R\otimes_R X@>\beta>>\Sym_R^2\Omega_R\cdot
\phi_{\Cat}^2(X)\\
@V f VV@V g VV\\
\Omega^{\otimes 2}_R\otimes_R X@>h>>\Omega_R^{\otimes 2}\cdot Z,
\end{CD}
$$
where $h$ is given by the action of $\P^1_R\otimes_R\P^1_R$ on $Z$
(we use that $X\cong Z/(I_R\cdot Z)$ and $\Omega^{\otimes 2}_R\cdot I_R=0$), the morphism~$f$ is defined
by the embedding $$\Sym_R^2\Omega_R\to \Omega_R^{\otimes 2},$$ and the
morphism $g$ is induced by the isomorphism~$\Psi_{\Cat}$.
Using the injectivity of $\gamma$ for $X$ and for
$\phi^1_{\Cat}(X)$, we obtain that $h$ is injective. Since
$$\Omega_R^{\otimes 2}\big/\Sym_R^2\Omega_R\cong\wedge_R^2\Omega_R$$ is
a projective $R$-module, $f$ is also injective, which implies the
injectivity of $\beta$.
\end{proof}

Now let us give a more explicit (though, a longer) definition of a
$D_R$-category. First, consider only the functor~$\phi^1_{\Cat}$. In
this case, the situation is similar to the previously known
differential Tannakian category over a field with one derivation \cite[Definition~3]{difftann}, and~\cite[Definition~5.2.1]{Moshe}.
For simplicity, we assume that the tensor product is {\it exact} in~$\Cat$.

\begin{proposition}\label{prop-phi1expl}
Let $\Cat$ be an abelian $R$-linear tensor category such that the
tensor product is exact. Then to define a right-exact $R$-linear
tensor functor
$$
\phi^{1}_{\Cat}:\Cat\to{}_R{\left(\P^1_R\otimes_R\Cat\right)}
$$
together with an isomorphism between tensor functors
$$
\Phi_{\Cat}:(e\otimes\id_{\Cat})\circ\phi^{1}_{\Cat}\stackrel{\sim}\longrightarrow
\id_{\Cat}
$$
is the same as to define the following data:
\begin{enumerate}
\item\label{en:2704}
A functor $\At^1_{\Cat}:\Cat\to\Cat$ together with a functorial
exact sequence
\begin{equation}\label{eq:Atiyah}
\begin{CD}
0@>>>\Omega_R\otimes_R X@>>> \At^1_{\Cat}(X)@>>>X@>>> 0
\end{CD}
\end{equation}
for any object $X$ in $\Cat$;
\item\label{en:2662}
An isomorphism
$$
\At^1_{\Cat}(\uno)\stackrel{\sim}\longrightarrow\P_R^1\otimes_R\uno,
$$
where we consider the right $R$-module structure on $P^1_R$, such
that exact sequence~\eqref{eq:Atiyah} coincides with the natural
exact sequence for $X=\uno$:
$$
\begin{CD}
0@>>>\Omega_R\otimes_R\uno@>>>\P_R^1\otimes_R\uno@>>>\uno@>>> 0,
\end{CD}
$$
and, for any $a\in R\to\End_{\Cat}(\uno)$, we have
$$\At^1_{\Cat}(a)=l(a),$$ where we denote elements of $R$
(respectively, in~$\P_R^1$) and their images under the morphisms to
$\End(\uno_{\Cat})$ (respectively, to $\End(\P_R^1\otimes_R\uno)$)
in the same way;
\item\label{en:2679}
A functorial isomorphism with the Baer sum
\begin{equation}\label{eq:AtiyahBaer}
\At^1_{\Cat}(X\otimes Y)\stackrel{\sim}\longrightarrow
{\left(\At^1_{\Cat}(X)\otimes Y)+_\Bm
(X\otimes\At^1_{\Cat}(Y)\right)}
\end{equation}
for all objects $X$ and $Y$ in $\Cat$ that respects commutativity
and associativity constraints in $\Cat$ and the splitting of
$$\At^1_{\Cat}(\uno)\cong\P_R^1\otimes_R\uno$$ given by the canonical
right $R$-linear splitting $\P^1_R\cong R\oplus\Omega_R$.
\end{enumerate}
\end{proposition}
\begin{proof}
Given $\phi_{\Cat}^1$, let $\At^1_{\Cat}$ be the composition of
$\phi^1_{\Cat}$ with the forgetful functor
$\P^1_R\otimes_R\Cat\to\Cat$ (Proposition~\ref{prop-extscalfinite}).
Then Example~\ref{examp-tensBaer} and Lemma~\ref{lemma-isom}
(namely, its part that concerns the first two adjoint quotients)
imply the needed statement.
\end{proof}

\begin{remark}\label{rem-expldefin}
\hspace{0cm}
\begin{enumerate}
\item\label{en:2845}
The notation $\At$ is explained by an analogy with the case
$\Cat=\Mod(R)$ (Example~\ref{examp-moddiffcat}), when the
corresponding functor coincides with the Atiyah extension:
$$\At_{\Cat}^1(M)={\left(M\otimes_R\P_R^1\right)}_R$$ for an
$R$-module $M$. In particular, for a $D_k$-Hopf algebra $A$ or a
$D_k$-Hopf algebroid $(R,A)$ over $k$, we have that
$$\At^1_{\Cat}(M)={\left(M\otimes_R\P_R^1\right)}_R,$$ where
$\Cat=\Comod(A)$ or $\Cat=\Comod(R,A)$
(Example~\ref{examp-comoddiffcat}) and~$M$ is an $A$-comodule.
\item
To give the functor $\At_{\Cat}$ is the same as to give an object of
type $(\Omega_R[1],\alpha)$ in the category of K\"ahler
differentials for the derived category of $\Cat$ as defined
in~\cite{Markarian}.
\item\label{it:2806}
To be strict, we distinguish between a $\P_R^1$-module $(Y,\alpha_Y)$
in $\Cat$ and the corresponding object $Y$ in $\Cat$, which makes
the difference between $\phi^1_{\Cat}(X)$ and $\At^1_{\Cat}(X)$.
\end{enumerate}
\end{remark}

To define a $D_R$-category in these terms, let us first
discuss several properties of the functor
$$\At^1_{\Cat}:\Cat\to\Cat.$$ It is not tensor and is not $R$-linear.
For any object $X$ in $\Cat$, $\At^1_{\Cat}(X)$ is canonically a
$P^1_R$-module in $\Cat$ with respect to the right $R$-module
structure on $\P_R^1$ (Example~\ref{examp-tensBaer}). For any $a\in
R$, we have
$$
\At^1_{\Cat}(a)=a-\dd a,
$$ where~$a$ acts on objects in $\Cat$, being a scalar from $R$, and
$$
\dd a\in\Omega_R\subset\P_R^1
$$
acts on $\At^1_{\Cat}(X)$ as the composition
$$
\begin{CD}
\At^1_{\Cat}(X)\to X@>{\dd a\otimes \id_X}>>\Omega_R\otimes_R
X\to\At^1_{\Cat}(X).
\end{CD}
$$
Further, for any $X$ in~$\Cat$, the object
$\At^1_{\Cat}{\big(\At^1_{\Cat}(X)\big)}$ is a
${\big(\P_R^1\otimes_R\P_R^1\big)}$-module in $\Cat$. Consider
the filtration by ideals:
\begin{equation}\label{eq:filtration}
\P^1_R\otimes_R\P^1_R\supset
{\left(\Omega_R\otimes_R\P^1_R+\P^1_R\otimes_R\Omega_R\right)}\supset
\Omega_R\otimes_R\Omega_R\supset 0.
\end{equation}
This defines a decreasing filtration on
$\At^1_{\Cat}{\big(\At^1_{\Cat}(X)\big)}$. By exact
sequence~\eqref{eq:Atiyah}, the corresponding adjoint quotients are
as follows:
$$
X,\quad \left(\Omega_R\otimes_R
X\right)\oplus\left(\Omega_R\otimes_R X\right),\quad
\Omega_R\otimes_R\Omega_R\otimes_R X.
$$
In addition, the Baer sum isomorphism~\eqref{eq:AtiyahBaer} (or,
equivalently, the tensor property of $\phi^1_{\Cat}$) implies that
there is a product map
$$
m:\At^1_{\Cat}{\left(\At^1_{\Cat}(X)\right)}\otimes
\At^1_{\Cat}{\left(\At^1_{\Cat}(Y)\right)}\to
\At^1_{\Cat}{\left(\At^1_{\Cat}(X\otimes Y)\right)}.
$$

We will use the following technical result. By a filtered ring, we
mean a ring $A$ together with a decreasing filtration $$A=F^0A\supset
F^1A\supset\ldots\quad \text{such that}\quad F^iA\cdot F^jA\subset F^{i+j}A.$$

\begin{lemma}\label{lemma-filtr}
Let $A$ be a finitely filtered ring, $f:M\to N$ be a morphism
between $A$-modules (possibly, between $A$-modules in an appropriate
abelian tensor category). Suppose that
$$
\gr^0f:\gr^0M\to \gr^0 N
$$
is an isomorphism and, for any $i$, the canonical morphism
$$
\gr^iA\otimes_{\gr^0A}\gr^0N\to\gr^iN
$$
is an isomorphism. Then $f$ is an isomorphism.
\end{lemma}
\begin{proof}
We have surjective morphisms $$\gr^i A\otimes_{\gr^0 A}\gr^0M\to
\gr^iM.$$ By the hypotheses of the lemma, their compositions with~$\gr^if$ is an isomorphism. Thus, $\gr^if$ is an isomorphism. Since
$A$ is finitely filtered, we conclude that $f$ is an isomorphism.
\end{proof}

\begin{proposition}\label{prop-expl}
Let $\big(\Cat,\At^1_{\Cat}\big)$ be as in
Proposition~\ref{prop-phi1expl}. Then to define a $D_R$-structure on
$\Cat$ with $\phi^1_{\Cat}$ being given by $\At_{\Cat}^1$ is the
same as to define a functorial $\Qr_R$-submodule
$$
\At^2_{\Cat}(X)\subset\At^1_{\Cat}{\left(\At^1_{\Cat}(X)\right)}
$$
such that, for all $X$ and $Y$ in $\Cat$, the following is
satisfied:
\begin{enumerate}
\item\label{en:2800}
The map $m$ sends $\At_{\Cat}^2(X)\otimes \At_{\Cat}^2(Y)$ to
$\At_{\Cat}^2(X\otimes Y)$;
\item\label{en:2802}
The adjoint quotients of the intersection of $\At_{\Cat}^2(X)$ with
the above filtration on $\At^1_{\Cat}{\big(\At^1_{\Cat}(X)\big)}$
are contained in
$$
X,\quad \Omega_R\otimes_R X,\quad \Sym^2_R\Omega_R\otimes_R X,
$$
where we consider the diagonal embedding
$$
\Omega_R\otimes_R X\hookrightarrow \left(\Omega_R\otimes_R
X\right)\oplus\left(\Omega_R\otimes_R X\right)
$$
and the natural embedding
$$
\Sym^2_R\Omega_R\otimes_R X\hookrightarrow
\Omega_R\otimes_R\Omega_R\otimes_R X;
$$
\item\label{en:2808}
The induced map from $\At_{\Cat}^2(X)$ to
$X=\gr^0\At^1_{\Cat}\big(\At^1_{\Cat}(X)\big)$ is surjective.
\end{enumerate}
\end{proposition}
\begin{proof}
Given a $D_R$-structure $\phi^2_{\Cat}$, let $\At_{\Cat}^2$ be the
composition of $\phi^2_{\Cat}$ with the forgetful functor
$$\Qr_R\otimes_R\Cat\to\Cat.$$ Since~$\phi^2_{\Cat}$ is a tensor
functor and we have an isomorphism of tensor functors $\Psi_{\Cat}$,
$\At_{\Cat}^2$ satisfies statement~\ref{en:2800}. Also, by
Lemma~\ref{lemma-isom}, we have statements~\ref{en:2802} and~\ref{en:2808}.

Conversely, let $\At_{\Cat}^2$
satisfy statements~\ref{en:2800},~\ref{en:2802}, and~\ref{en:2808}. To
construct the isomorphism $\Psi_{\Cat}$, we need to show that the
natural morphism
$$
\mu:{\left(\P_R^1\otimes_R\P_R^1\right)}\otimes_{\Qr_R}\At^2_{\Cat}(X)\to
\At^1_{\Cat}\left(\At^1_{\Cat}(X)\right)
$$
is an isomorphism. Note that $\mu$ is a morphism between
$\big(\P_R^1\otimes_R\P_R^1\big)$-modules in $\Cat$. Consider the
filtration on the source and target of $\mu$ given by
filtration~\eqref{eq:filtration}. By statements~\ref{en:2802}
and~\ref{en:2808}, the natural morphism
$$
\At^2_{\Cat}(X)\big/I_R\cdot\At^2_{\Cat}(X)\to X
$$
is an isomorphism. Therefore, the first adjoint quotient of the
source of $\mu$ is isomorphic to $X$ and $\gr^0\mu$ is an
isomorphism (being the identity from $X$ to itself). By
Lemma~\ref{lemma-filtr} applied to the finitely filtered
ring~$\big(\P_R^1\otimes_R\P_R^1\big)$,~$\mu$ is an isomorphism.

The tensor structure on the functor $\phi^2_{\Cat}$ is given by the
product map $m$. The fact that we obtain an isomorphism follows from
Lemma~\ref{lemma-filtr} applied to the finitely filtered ring
$\P_R^2$. Finally, $$\phi^2_{\Cat}:\Cat\to \Qr_R\otimes_R\Cat$$ is
$R$-linear with respect to the left homomorphism $l:R\to \Qr_R$,
because so is the functor $\phi^1_{\Cat}$, and, hence,
$\big(\id_{\P_R^1}\otimes\phi^1_{\Cat}\big)\circ\phi^1_{\Cat}$.
\end{proof}

\begin{definition}\label{defin-diffgenercat}
Given an object $X$ in a rigid $D_k$-category $\Cat$, let $\langle
X\rangle_{\otimes,D}$ denote the minimal full rigid
$D_k$-subcategory in $\Cat$ that contains $X$ and is an closed under
taking subquotients. We say that the category~$\langle
X\rangle_{\otimes,D}$ is {\it $D_k$-tensor generated} by the
object~$X$.
\end{definition}

\begin{remark}\label{remark-difftensgen}
In the notation of Definition~\ref{defin-diffgenercat}, $\Cat$ is
$D_k$-tensor generated by $X$ if and only if there is no smaller
full subcategory in $\Cat$ containing $X$ and closed under taking
direct sums, tensor products, duals, subquotients, and applying the
functor $\At_{\Cat}^1$ (Section~\ref{subsection-explicitdefin}),
because $\At^{2}_{\Cat}$ is a subobject in
$\At_{\Cat}^{1}{\big(\At_{\Cat}^{1}(X)\big)}$. In addition, the
category $\langle X\rangle_{\otimes,D}$ is the union of
all~$\Cat_i$'s, where $\Cat_i$ is the subcategory in $\Cat$ tensor
generated by~${\big(\At^{1}_{\Cat}\big)^{\circ\,i}}(X)$.
\end{remark}

\begin{remark}
\hspace{0cm}
\begin{enumerate}
\item\label{en:28022}
Definition~\ref{defin-intdifftens} is analogous to the definition of a group action on a category (for example, see \cite{Deligne1997}) so that the isomorphisms $\Phi$ and $\Psi$ correspond to the unit and associativity constraints, respectively. We do not require  the pentagon condition for $\Psi$ in Definition~\ref{defin-intdifftens} as we are not considering $\P_R^3$ (Section~\ref{subsection-illusie}). 

On the contrary, the compatibility condition between $\Phi$ and $\Psi$ makes sense in our set-up and means that, for any object $X$ in $\Cat$, the following compositions coincide:
$$
\begin{CD}
\At^2_{\Cat}(X)\to \At^1_{\Cat}\left(\At^1_{\Cat}(X)\right)@>\At^1(\pi_X)>> \At^1_{\Cat}(X),\quad
\end{CD}
\begin{CD}
\At^2_{\Cat}(X)\to \At^1_{\Cat}\left(\At^1_{\Cat}(X)\right)@>\pi_{\At^1(X)}>>\At^1_{\Cat}(X),
\end{CD}
$$
where $\pi_X:\At^1_{\Cat}(X)\to X$ is the morphism given by exact sequence~\eqref{eq:Atiyah}. We do not require  this condition in Definition~\ref{defin-intdifftens} as well. However, it holds for Examples~\ref{examp-moddiffcat},~\ref{examp-comoddiffcat} and for the differential category constructed in Theorem~\ref{teor-paramAtiayh}.
\item\label{en:3210}
Suppose that $D_R$ is of rank one over $R$ and the compatibility condition from part~\ref{en:28022} holds for a $D_R$-category~$\Cat$. Then we have $$\Sym^2_R \Omega_R=\Omega_R\otimes_R\Omega_R$$ and, by a dimension argument, the embedding
$$
\At^2_{\Cat}(X)\to \At^1_{\Cat}\left(\At^1_{\Cat}(X)\right)
$$
identifies $\At^2_{\Cat}(X)$ with the kernel of the morphism
$$
\At^1(\pi_X)-\pi_{\At^1(X)}:\At^1_{\Cat}\left(\At^1_{\Cat}(X)\right)\to \At^1_{\Cat}(X).
$$
Therefore, $\At^2_{\Cat}$ is uniquely defined by $\At^1_{\Cat}$, or, equivalently, $\phi_{\Cat}^1$ is uniquely defined up to a canonical isomorphism by $\phi_{\Cat}^2$.
\item
Suppose that $F:\Cat\to \Dop$ is a faithful differential functor between $D_R$-categories and the compatibility condition from part~\ref{en:28022} holds for $\Dop$. Then this condition also holds for $\Cat$. In particular, if $\Cat$ is a $D_k$-Tannakian category (Definition~\ref{defin-diffTancat}) over a differential field $(k,D_k)$, then the compatibility condition holds for $\Cat$ by the end of part~\ref{en:28022}. If, in addition, $\dim_k(D_k)=1$, then, by part~\ref{en:3210}, we see that Definition~\ref{defin-diffTancat} is equivalent to the definitions of a differential Tannakian category over a field with one derivation from \cite[Definition~3]{difftann} and~\cite[Definition~5.2.1]{Moshe}.
\end{enumerate}
\end{remark}

Let us discuss the relation between
Definition~\ref{defin-intdifftens} and the definition of a neutral
differential Tannakian category with several commuting derivations
given in~\cite[Definition~3.1]{diffreductive}. Suppose that
$D_R$ is a free $R$-module generated by commuting derivations
$\partial_1,\ldots,\partial_d$. Let $\omega_1,\ldots,\omega_d$ be
the dual basis in $\Omega_R=D_R^{\vee}$. There is an involution
$\sigma$ of the $(R\otimes R)$-algebra $\P^1_R\otimes_R\P_R^1$
uniquely defined by the condition $\sigma(\omega_i\otimes
1)=1\otimes\omega_i$ for all $i$. For example, for any
$$\omega=\sum _i a_i\omega_i\in\Omega_R,\ \ \ a_i\in R,$$ we have
$$\sigma(1\otimes\omega)=\omega\otimes 1+\sum _i\omega_i\otimes\dd
a_i.$$ The subring of invariants under the involution $\sigma$
coincides with $\Qr_R$, because $\dd\omega_i=0$ for all $i$.
Further, for any $i$, the morphism of differential rings $(R,D_R)\to
(R,R\cdot\partial_i)$ induces the ring homomorphism
$\P_R^1\to\P_i^1$, where $\P_i^1$ denotes the 1-jet ring associated
with the differential ring $(R,R\cdot\partial_i)$. It follows that
$\sigma$ induces a collection of ring isomorphisms
$$\P_i^1\otimes_R\P_j^1\cong \P_j^1\otimes_R\P_i^1$$ that commute with
$\sigma$ via the homomorphisms
$$
\P_R^1\otimes_R\P_R^1\to \P_i^1\otimes_R\P_j^1.
$$
Next, let $\Cat$ be a $D_R$-category over $R$. Then, for any object
$X$ in $\Cat$, the isomorphism
$$
\mu:{\left(\P_R^1\otimes_R\P_R^1\right)}\otimes_{\Qr_R}\At^2_{\Cat}(X)
\stackrel{\sim}\longrightarrow
\At^1_{\Cat}\left(\At^1_{\Cat}(X)\right)
$$
induces an involution $\sigma_X$ on
$\At^1_{\Cat}\big(\At^1_{\Cat}(X)\big)$ such that the invariants
of $\sigma_X$ coincide with $\At_{\Cat}^2(X)$. For any~$i$, the ring
homomorphism $\P_R^1\to\P_i^1$ induces a morphism
$$
\At^1_{\Cat}(X)\to\At^1_i(X),
$$
where we have a functorial exact sequence
$$
\begin{CD}
0@>>>X@>>> \At^1_i(X)@>>>X@>>> 0.
\end{CD}
$$
Since the ring homomorphism
$$
\P_R^1\otimes_R\P_R^1\to\bigoplus _{i,j}\left(\P^1_i\otimes_R\P^1_j\right)
$$
is injective, the natural morphism
$$
\At^1_{\Cat}\left(\At^1_{\Cat}(X)\right)\to\bigoplus _{i,j}\At^1_i\left(\At^1_j(X)\right)
$$
is also injective. It follows that to define $\sigma_X$ it is enough
to specify a collection of isomorphisms
$$
S_{i,j}:\At^1_i\left(\At^1_j(X)\right)\stackrel{\sim}\longrightarrow
\At^1_j\left(\At^1_i(X)\right)
$$
that should satisfy certain compatibility conditions. If, in
addition, $R=k$ is a field, $\Cat$ is a neutral Tannakian category,
and the fiber functor commutes with $\At^1$ and sends the
isomorphisms $S_{i,j}$ to the corresponding isomorphisms in
$\Vect(k)$, then~$S_{i,j}$ satisfy the compatibility conditions and
define correctly $\At^2_{\Cat}$ as the equalizer in
$\At^1_{\Cat}\big(\At^1_{\Cat}(X)\big)$ of all the isomorphisms~$S_{i,j}$. Also, one needs to require the Baer sum isomorphisms for
$\At^1_i$ to obtain the Baer sum isomorphism for $\At^1_{\Cat}$,
which would preserve~$\At^2_{\Cat}$. The latter coincides with
the definition of a neutral differential Tannakian category as given
in \cite[Definition~3.1]{diffreductive}.

Finally, let us perform a calculation that we use in Section~\ref{subsection-isomfibfunct}. Let $(R,D_R)$ be a differential ring with free~$D_R$. Choose a basis $\partial_1,\ldots,\partial_d$ in $D_R$ over $R$ and let $\omega_1,\ldots,\omega_d$ be the dual basis in $\Omega_R$. Consider free $R$-modules $$M=R\cdot e_1\oplus\ldots \oplus R\cdot e_m\quad \text{and}\quad N=R\cdot f_1\oplus\ldots \oplus R\cdot f_n$$ and a morphism of $R$-modules $\phi:M\to N$ given by a matrix $T$. Then the morphism $$\At^1_R(\phi):\At^1_R(M)\to\At^1_R(N)$$ is given by the matrix
$$
\begin{pmatrix}
T&0&\ldots&0&0\\
-\partial_1(T)&T&\ldots&0&0\\
\vdots&&\ddots&&\vdots\\
-\partial_{d-1}(T)&0&\ldots&T&0\\
-\partial_d(T)&0&\ldots&0&T
\end{pmatrix},
$$
where we consider the basis
$$
{\left\{e_1\otimes 1,\ldots,e_m\otimes 1,e_i\otimes\omega_j\right\}},\quad 1\leqslant i\leqslant m,\ 1\leqslant j\leqslant d,\quad\text{in}\ \At^1_R(M)={\left(M\otimes_R\P_R^1\right)}_R$$ with respect to the right $R$-module structure (Remark~\ref{rem-expldefin}\eqref{en:2845}) and the analogous basis in $\At^1_R(N)$.

\subsection{Differential Tannakian categories}\label{subsection-isomfibfunct}

Throughout this subsection, we fix a differential field $(k,D_k)$ and use the notions and notation from
Section~\ref{subsection-prelTann}. Let us define differential
Tannakian categories.

\begin{definition}\label{defin-diffTancat}
\hspace{0cm}
\begin{itemize}
\item
A {\it $D_k$-Tannakian category} over $(k,D_k)$ (or simply over $k$)
is a $D_k$-category $\Cat$ over $k$
(Definition~\ref{defin-intdifftens}) such that $\Cat$ is rigid, the
homomorphism $k\to\End_{\Cat}(\uno)$ is an isomorphism, and there
exists a $D_k$-algebra~$R$ over $k$ together with a differential
functor $\omega:\Cat\to \Mod(R)$ (Definition~\ref{defin-Dkfunctor}).
\item
Given two differential functors $\omega,\eta:\Cat\to\Mod(R)$, denote
the set of isomorphisms between~$\omega$ and~$\eta$ as differential
functors by $\fIsom^{\otimes,D}(\omega,\eta)$.
\item A {\it neutral $D_k$-Tannakian
category} over $k$ is a $D_k$-Tannakian category over~$k$ with a
fixed differential functor to~$\Vect(k)$.
\end{itemize}
\end{definition}

\begin{remark}
We use notation from Definition~\ref{defin-diffTancat}. Since the
category $\Cat$ is rigid and any differential functor is right-exact
(Definition~\ref{defin-Dkfunctor}), we see that the functor $\omega$
is exact \cite[2.10(i)]{DeligneFS}. In particular, $\omega$ is a fiber functor from~$\Cat$ to $\Mod(R)$.
\end{remark}

\begin{example}
Let $(R,A)$ be a $D_k$-Hopf algebroid over $k$ (Example~\ref{example-diffobjects}\eqref{en:6104}) such that $A$ is faithfully flat over~\mbox{$R\otimes_k R$}.
Since the forgetful functor $$\Comodf(R,A)\to\nolinebreak\Mod(R)$$ is a fiber functor (Section~\ref{subsection-prelHopfalg}) and differential  (Example~\ref{examp-difffunctor}\eqref{en:2370}), the category $\Comodf(R,A)$ is a $D_k$-Tannakian category over~$k$.
In particular, if $R=k$ and $A$ is a $D_k$-Hopf algebra over $k$ (Example~\ref{example-diffobjects}\eqref{examp-diffHopfalgebra}), then the category $\Comodf(A)$ is a neutral $D_k$-Tannakian category over~$k$.
\end{example}

Given a $D_k$-Tannakian category $\Cat$ over $k$, a differential
functor $\omega:\Cat\to\Mod(R)$, and a morphism of $D_k$-algebras~$R\to S$, we put
$$
\omega_S:\Cat\to\Mod(S),\quad X\mapsto S\otimes_R\omega(X).
$$

\begin{proposition}\label{prop-reprdiffalg}
Let $R$ be a $D_k$-algebra over $k$, $\Cat$ be a $D_k$-Tannakian category over $k$,
$$
\omega,\eta:\Cat\to\Mod(R)
$$
be differential functors, and let $A$ be the $R$-algebra that
corepresents the functor  (Section~\ref{subsection-prelTann}):
$$
\Isom^{\otimes}(\omega,\eta):\Alg(R)\to\Sets,\quad S\mapsto \fIsom^{\otimes}(\omega_S,\eta_S).
$$
Then $A$ has a canonical structure of a $D_R$-algebra over $R$
such that $A$ corepresents the functor
$$
\Isom^{\otimes,D}(\omega,\eta):\DAlg(R,D_R)\to \Sets,\quad
S\mapsto \fIsom^{\otimes,D}(\omega_S,\eta_S).
$$
\end{proposition}
\begin{proof}
First let us construct a $D_R$-structure on $A$. The idea is as
follows. The collection $(\Cat,R,\omega,\eta)$ is a $D_k$-object in
the (2-)category of collections that consist of a Tannakian category
over $k$, a $k$-algebra, and two fiber functors to modules over this
algebra. On the other hand, the (pseudo-) functor that assigns $A$ to
such a collection commutes with extensions and restrictions of
scalars between $k$ and $\Qr_k$. This defines a $D_k$-structure on
$A$. Let us give more details. By the definition of $A$, the
$\Qr_R$-algebra $A\otimes_R\Qr_R$ corepresents the functor
$$
\Isom^{\otimes}{\left(\omega_{\Qr_R},\eta_{\Qr_R}\right)}:\Alg{\left(\Qr_R\right)}\to\Sets,
$$
where, as above,
$$
\omega_{\Qr_R}:\Cat\to \Mod{\left(\Qr_R\right)},\quad X\mapsto\omega(X)\otimes_R\Qr_R,\quad
\mbox{and}\quad
\eta_{\Qr_R}:\Cat\to \Mod{\left(\Qr_R\right)},\quad X\mapsto\eta(X)\otimes_R\Qr_R.
$$
The functors $\omega_{\Qr_R}$ and $\eta_{\Qr_R}$ are exact
$k$-linear tensor functors. Moreover, $\omega_{\P_R^2}$ is the
composition of the functor
$$
-\otimes_k\Qr_k:\Cat\to\Cat\otimes_k\Qr_k
$$
and the functor
$$
\omega\otimes_k\Qr_k:\Cat\otimes_k\Qr_k\to\Mod(R)\otimes_k\Qr_k\cong\Mod{\left(\Qr_R\right)}.
$$
The analogous relations hold for $\eta_{\Qr_R}$ and
$\eta\otimes_k\Qr_k$. Hence, by
Proposition~\ref{prop-extscalfinite}, there is a canonical
isomorphism of functors from $\Alg{\big(\Qr_R\big)}$ to $\Sets$:
\begin{equation}\label{eq:3269}
\Isom^{\otimes}{\left(\omega_{\Qr_R},\eta_{\Qr_R}\right)}\cong
\Isom^{\otimes}{\left(\omega\otimes_k\Qr_k,\eta\otimes_k\Qr_k\right)}.
\end{equation}
Similarly, the $\Qr_R$-algebra $\Qr_R\otimes_R A$
corepresents the functor
$$
\Isom^{\otimes}{\left(_{\Qr_R}\omega,_{\Qr_R}\eta\right)}:\Alg{\left(\Qr_R\right)}\to\Sets
$$
and we have an isomorphism of functors
\begin{equation}\label{eq:3285}
\Isom^{\otimes}{\left(_{\Qr_R}\omega,_{\Qr_R}\eta\right)}\cong \Isom^{\otimes}{\left(\Qr_k\otimes_k\omega,\Qr_k\otimes_k\eta\right)}.
\end{equation}
Again, by Proposition~\ref{prop-extscalfinite}, the right-exact $k$-linear tensor functor $$\phi_{\Cat}^2:\Cat\to\Qr_k\otimes_k\Cat$$ defines a right-exact $\Qr_k$-linear tensor functor
$$
\epsilon_{\Cat}^2:\Cat\otimes_k\Qr_k\to\Qr_k\otimes_k\Cat.
$$
In addition, the isomorphism $\Pi_{\omega}$ defines an
isomorphism of tensor functors
$$
\omega\otimes_k\Qr_k\stackrel{\sim}\longrightarrow  {\left(\Qr_k\otimes_k\omega\right)}\circ\epsilon_{\Cat}^2
$$
from $\Cat\otimes_k\Qr_k$ to $\Mod{\left(\Qr_R\right)}$. Analogously, $\Pi_{\eta}$ defines an isomorphism of tensor functors
$$
\eta\otimes_k\Qr_k\stackrel{\sim}\longrightarrow  {\left(\Qr_k\otimes_k\eta\right)}\circ\epsilon_{\Cat}^2.
$$
This leads to a morphism of functors
$$
\Isom^{\otimes}{\left(\Qr_k\otimes_k\omega,\Qr_k\otimes_k\eta\right)}\to \Isom^{\otimes}{\left(\omega\otimes_k\Qr_k,\eta\otimes_k\Qr_k\right)}.
$$
Hence, by isomorphisms~\eqref{eq:3269} and~\eqref{eq:3285}, we obtain a morphism of functors
$$
\Lambda:\Isom^{\otimes}{\left(_{\Qr_R}\omega,_{\Qr_R}\eta\right)}\to
\Isom^{\otimes}{\left(\omega_{\Qr_R},\eta_{\Qr_R}\right)}.
$$
By the corepresentability properties of $A\otimes_R\Qr_R$ and $\Qr_R\otimes_R A$, the morphism of functors $\Lambda$ corresponds to a morphism of~$\Qr_R$-algebras
$$
\epsilon_A^2:A\otimes_R\Qr_R\to\Qr_R\otimes_R A.
$$
Since the isomorphisms $\Psi_{\Cat}$ and $\Phi_{\Cat}$ commute with $\Psi_R$ and $\Phi_R$ via $\omega$ and $\eta$, the morphism $\epsilon^2_A$ satisfies the required properties (Example~\ref{example-diffobjects}\eqref{en:1646}) to define a $D_R$-structure on $A$.

Now let us prove the corepresentability property of $A$ in the
category of $D_R$-algebras. Let $S$ be a $D_R$-algebra,
\mbox{$\alpha:\omega_S\to \eta_S$} be an isomorphism of tensor functors,
and let $f:A\to S$ be the corresponding morphism of $R$-algebras. We
need to show that~$\alpha$ is differential if only if $f$ is
differential. Note that $\alpha$ is differential if and only if the
map
$$
\Lambda_S:\fIsom^{\otimes}{\left(_{\Qr_S}\omega,_{\Qr_S}\eta\right)}\to
\fIsom^{\otimes}{\left(\omega_{\Qr_S},\eta_{\Qr_S}\right)}
$$
sends $_{\Qr_S}\alpha$ to $\alpha_{\Qr_S}$. This is equivalent to the equality between the morphism
$$
f\otimes\id_{\Qr_R}:A\otimes_R\Qr_R\to S\otimes_R\Qr_R
$$
and the composition
$$
\begin{CD}
A\otimes_R\Qr_R@>\epsilon_A^2>>\Qr_R\otimes_R A@>\id_{\Qr_R}\otimes f>> \Qr_R\otimes_R S@>{\left(\epsilon^2_S\right)}^{-1}>> S\otimes_R\Qr_R.
\end{CD}
$$
The latter is equivalent to $f$ being differential.
\end{proof}

\begin{example}\label{examp-diffstrdmod}
Let $D_k=k\cdot\partial$, where $\partial$ is a formal symbol that
denotes the trivial derivation from~$k$ to itself,~$K$ be a
differential field over $(k,D_k)$ such that $k=K^{\partial}$, let
$\Cat=\DMod(K,D_K)$ with $D_K=K\cdot\partial$, $\omega_0:\Cat\to\Vect(k)$ be a fiber functor, and let
\mbox{$\omega:\Cat\to\Vect(K)$} be the forgetful functor. Since the left and the right $k$-module structures on $\Qr_k$ coincide,
$\Cat$ has the trivial $D_k$-structure with
$$
\phi^{2}_{\Cat}(M):=\Qr_k\otimes_k M\cong M\oplus M
$$
for a $\partial$-module $M$ over $K$. Since
$$
\omega_0{\left(\Qr_k\otimes_k M\right)}\cong \Qr_k\otimes_k \omega_0(M)\cong \omega_0(M)\otimes_k\Qr_k,
$$
we see that $\omega_0$ is a differential functor. By
Proposition~\ref{prop-diffmodjets}, for any $\partial$-module $M$
over $K$, there is a canonical isomorphism of $(K\otimes K)$-modules
$$
M\otimes_K\Qr_K\cong \Qr_K\otimes_K M.
$$
Since $$\left(\Qr_K\otimes_K M\right)_K\cong \left(\Qr_k\otimes_k M\right)_K,$$ we obtain
that $\omega$ is a differential functor. Let $A$ be the $K$-algebra
that corepresents the functor
$$
\Isom^{\otimes}\big((K\otimes_k-)\circ\omega_0,\omega\big).
$$
Proposition~\ref{prop-reprdiffalg} provides a $\partial$-structure
on $A$. This $\partial$-structure coincides with the one defined
in~\cite[9.2]{DeligneFS} (note that the definition of a
$\partial$-structure from \cite[9.2]{DeligneFS} works well
for the whole category $\DMod(K,D_K)$, not just a subcategory tensor
generated by one object).
\end{example}

\begin{theorem}\label{theor-diffTanncorr}
Let $\Cat$ be a $D_k$-Tannakian category over a differential field
$(k,D_k)$, $R$ be a $D_k$-algebra over $k$, and let
$\omega:\Cat\to\Mod(R)$ be a differential functor. Then there exists
a $D_k$-Hopf algebroid $(R,A)$ over $(k,D_k)$ such that $A$ is
faithfully flat over $R\otimes_k R$ and $\omega$ lifts up to an
equivalence of $D_k$-categories over $k$
$$
\Cat\stackrel{\sim}\longrightarrow\Comodf(R,A).
$$
\end{theorem}
\begin{proof}
Apply Proposition~\ref{prop-reprdiffalg} to the differential
functors $_{R\otimes R}\,\omega$ and $\omega_{R\otimes R}$
from~$\Cat$ to $\Mod({R\otimes_k R})$, where, as above, for~$X$
in~$\Cat$, we put
$$
(_{R\otimes R}\,\omega)(X):=(R\otimes_k R)\otimes_R\omega(X)\cong
R\otimes_k\omega(X),\quad(\omega_{R\otimes
R})(X):=\omega(X)\otimes_R(R\otimes_kR)\cong \omega(X)\otimes_k R.
$$
This gives a differential algebra $A$ over $R\otimes_k R$, where the
differential structure on $R\otimes_k R$ is defined as on the tensor
product of $D_k$-algebras (Remark~\ref{remark-diffalg}).
From the properties of the functor from $\DAlg(R\otimes_k R)$ to~$\Sets$
corepresented by $A$, it follows that $(R,A)$ is a $D_k$-Hopf algebroid
over $k$ and $\omega$ lifts to a differential functor between
$D_k$-categories $$\Cat\to\Comodf(R,A)$$
(Example~\ref{examp-comoddiffcat}). Finally,
by~\cite[1.12]{DeligneFS} (Theorem~\ref{theor-HopfalgDeligne}),
the latter functor is an equivalence of categories and $A$ is
faithfully flat over $R\otimes_k R$.
\end{proof}

In particular, when $R=k$, Theorem~\ref{theor-diffTanncorr}
recovers \cite[Theorem~2]{OvchTannakian}.

Now let us discuss finiteness properties of the algebra $A$ from
Proposition~\ref{prop-reprdiffalg}.

\begin{proposition}\label{prop-difffingencat}
In the notation of Proposition~\ref{prop-reprdiffalg}, suppose that
$\Cat$ is $D_k$-tensor generated by an object $X$
(Definition~\ref{defin-diffgenercat}). Then $A$ is $D_k$-generated
over $R$ by the matrix entries of the canonical isomorphism
$$
\omega(X)_A\stackrel{\sim}\longrightarrow \eta(X)_A
$$
and the matrix entries of its inverse with respect to any choice of
systems of generators of $\omega(X)_A$ and~$\eta(X)_A$ over~$A$.
\end{proposition}
\begin{proof}
This follows from Proposition~\ref{prop-Tanngen}, Remark~\ref{remark-difftensgen}, and the calculation of $\At^1_R(\phi)$ at the end of Section~\ref{subsection-explicitdefin}.
\end{proof}

\begin{corollary}
Suppose that $(k,D_k)$ is differentially closed, $\Char k=0$, and
the category $\Cat$ is $D_k$-tensor generated by one object. Then
all differential functors from $\Cat$ to $\Vect(k)$ are isomorphic.
\end{corollary}
\begin{proof}
Let $\omega,\eta:\Cat\to\Vect(k)$ be differential functors. By
Proposition~\ref{prop-reprdiffalg}, isomorphisms between~$\omega$
and~$\eta$ as differential functors are in bijection with morphisms
of $D_k$-algebras $A\to k$. By
Proposition~\ref{prop-difffingencat},~$A$ is $D_k$-finitely
generated over $(k,D_k)$. By~\cite[1.12]{DeligneFS}, $A$ is
non-zero, being faithfully flat over $k$. Since $\Char k= 0$, there
is a morphism from $A$ to $k$ (for example, see
\cite[Definition~4]{TrushinSplitting} and the references given
there), which finishes the proof.
\end{proof}

Finally, let us describe the differential structure on the ring $A$
from Proposition~\ref{prop-reprdiffalg} explicitly. We use its notation. First,
recall an explicit construction of $A$. Consider the $R$-module
$$
F:=\bigoplus_{X \in \Ob(\Cat)}\Hom_R(\omega(X),\eta(X))
$$
and the $R$-submodule $T$ of $F$ generated by all elements of type
$$
(\psi\circ\omega(\phi))\oplus(-\eta(\phi)\circ\psi)\in\Hom_R(\omega(X),\eta(X))
\oplus\Hom_R(\omega(Y),\eta(Y)),
$$
where $\phi\in\Hom_{\Cat}(X,Y)$, $\psi\in\Hom_R(\omega(Y),\eta(X))$,
and $X$, $Y$ are objects in $\Cat$. Then we have $A= F/T$
(\cite{Deligne}). For each object~$X$ in~$\Cat$, choose an
$R$-linear section
$$
s_X:\eta(X)\to\At^1_R{\left(\eta(X)\right)}
$$
of the morphism $\At^1_R\left(\eta(X)\right)\to \eta(X)$. By
Remark~\ref{rem-expldefin}\eqref{en:2845} and
Proposition~\ref{prop-diffmodjets}, $s_X$ corresponds to a,
possibly, non-integrable $D_R$-structure on $\eta(X)$. This uniquely defines
an $R$-linear morphism
$$
t_X:\At^1_R\left(\eta(X)\right)\to \Omega_R\otimes_R\eta(X)
$$
such that the canonical morphism
$$\Omega_R\otimes_R\eta(X)\to\At^1_R\left(\eta(X)\right)$$ is a
section of $t_X$ and $t_X\circ s_X=0$. Next, for any $\partial\in
D_R$, consider the additive map
$$
\partial:\Hom_R(\omega(X),\eta(X))\to \Hom_R\left(\omega\left(\At^1_{\Cat}(X)\right),\eta\left(\At^1_{\Cat}(X)\right)\right),\ \ 
\partial(\psi):=s_X\circ{\left(\partial\otimes\id_{\eta(X)}\right)}\circ t_X\circ \At^1_R(\psi),
$$
where $\psi\in\Hom_R(\omega(X),\eta(X))$ and we use the functorial
isomorphism
$$
\omega\left(\At^1_{\Cat}(X)\right)\stackrel{\sim}\longrightarrow\At^1_R\left(\omega(X)\right).
$$
Taking the direct sum over all objects $X$ in $\Cat$, we get the
additive map $\partial:F\to F$. One can show that~$\partial$
preserves the submodule~$T$ and defines a derivation on the
$R$-algebra $A$. All together, this defines a $D_R$-structure on
$A$.

\section{Parameterized Atiyah extensions}\label{sec:paramAtiyah}

\subsection{Construction}

Throughout this section, we fix a differential field $(k,D_k)$ and a
parameterized differential algebra~$(R,D_R)$ over $(k,D_k)$
(Definition~\ref{defin-paramdifffield}). Recall that we have a
differential ring $\left(R,D_{R/k}\right)$, where $D_{R/k}$ is the
kernel of the structure map \mbox{$D_R\to R\otimes_k D_k$}
associated with the morphism of differential rings $(k,D_k)\to
(R,D_R)$. Put $\Omega_{R/k}:=D_{R/k}^\vee$.

\begin{theorem}\label{teor-paramAtiayh}
There is a canonical $D_k$-structure on the category $\DMod{\left(R,D_{R/k}\right)}$ such that the forgetful functor from the $D_k$-category $\DMod{\left(R,D_{R/k}\right)}$ over $(k,D_k)$ to the $D_R$-category $\Mod(R)$ over $(R,D_R)$ is a differential functor.
\end{theorem}
\begin{proof}
We follow the explicit approach from Section~\ref{subsection-explicitdefin}. First, we need to construct a right-exact $k$-linear tensor functor
$$
\phi^1:\DMod{\left(R,D_{R/k}\right)}\to\,_k{\left(\P_k^1\otimes_k\DMod{\left(R,D_{R/k}\right)}\right)}
$$
together with certain isomorphisms between tensor functors. Then we need to functorially construct  a $\Qr_k$-submodule $\At^2(M)$ in $\At^1{\big(\At^1(M)\big)}$ satisfying several properties. 

Recall that we distinguish between a $\P_k^1$-module in $\DMod(R,D_{R/k})$ and the corresponding object in $\DMod(R,D_{R/k})$, which makes the difference between $\phi^1$ and $\At^1$ (Remark~\ref{rem-expldefin}\eqref{it:2806}). In particular, $\phi^1\big(\phi^1(M)\big)$ is not well-defined, while $\At^1\big(\At^1(M)\big)$ is well-defined.
We call $\At^2(M)$ a {\it parameterized Atiyah extension}. The proof is divided into several steps.

\subsection*{Step 1. Construction of $\phi^1(M)$}

Let $M$ be a $D_{R/k}$-module. Put
\begin{equation}\label{eq:paramAt}
\At^1(M):={\left\{m\otimes
1+\sum _im_i\otimes\omega_i\:\big|\:\forall \xi\in D_{R/k},\:
\xi(m)=\sum _i\omega_i(\xi)m_i\right\}}\subset M\otimes_R\P_R^1,
\end{equation}
where $m,m_i\in M$, $\omega_i\in\Omega_R$. Here we use that
$D_{R/k}$ is an $R$-submodule in $D_R$, whence, $\omega_i(\xi)$ is
well-defined. Equivalently,~$\At^1(M)$ is the kernel of the map
\begin{equation}\label{eq:lambda}
\lambda:M\otimes_R\P_R^1\to \Omega_{R/k}\otimes_R M,\quad
m\otimes a+\sum _i m_i\otimes\omega_i\mapsto a\nabla_M(m)+\dd a\otimes m-\sum _i[{\omega}_i]\otimes m_i,
\end{equation}
where the brackets mean the application of the natural quotient map
$\Omega_R\to\Omega_{R/k}$. The Leibniz rule for~$\nabla_M$ implies
that $\lambda$ is well-defined. Also, $\lambda$ is $R$-linear with
respect to the right $R$-module structure on~$M\otimes_R\P_R^1$
defined by the homomorphism $r:R\to\P_R^1$. Hence, $\At^1(M)$ is an
$R$-submodule in $M\otimes_R\P_R^1$ with respect to $r$. Explicitly,
we have
\begin{equation}\label{eq:3273}
a\cdot{\left(m\otimes 1+\sum _i m_i\otimes\omega_i\right)}=am\otimes
1+m\otimes \dd a+\sum _im_i\otimes a\omega_i.
\end{equation}
Let us define a weak $D_{R/k}$-module structure on $\At^1(M)$
(Section~\ref{subsection-Lieder}). Recall that we have a weak
$D_{R/k}$-module structure on $\P_R^1$. Hence, we obtain a weak
$D_{R/k}$-module structure on the tensor product~$M\otimes_R\P_R^1$.
We claim that the corresponding action of an arbitrary element
$\partial\in D_{R/k}$ on $M\otimes_R\P_R^1$ preserves $\At^1(M)$.
Indeed, for any $$m\otimes 1+\sum _i m_i\otimes\omega_i\in
\At^1(M),$$ we have
$$
\partial{\left(m\otimes 1+\sum _i m_i\otimes\omega_i\right)}=\partial(m)\otimes 1+
\sum _i{\left(\partial(m_i)\otimes \omega_i+m_i\otimes
L_{\partial}(\omega_i)\right)}
$$
(see Definition~\ref{def:Liederivative} for $L_{\partial}$). Hence,
we need to show that, for any $\xi\in D_{R/k}$, we have
$$
\xi(\partial(m))=\sum _i\big(\omega_i(\xi)\cdot\partial(m_i)+
L_{\partial}(\omega_i)(\xi)\cdot m_i\big).
$$
By~\eqref{eq:Lie}, the right-hand side is equal to
$$
\sum _i\big(\omega_i(\xi)\cdot\partial(m_i)+
\partial(\omega_i(\xi))\cdot m_i-\omega_i([\partial,\xi])\cdot m_i\big).
$$
Further, by~\eqref{eq:paramAt}, the latter equals
$$
\partial(\xi(m))-[\partial,\xi](m).
$$
Thus, we conclude by the integrability condition for the $D_{R/k}$-module structure on $M$.
Let us check that the above weak $D_{R/k}$-module structure actually defines a $D_{R/k}$-module structure. For all $$a\in R,\quad \partial\in D_{R/k},\quad m\otimes 1+\sum _i m_i\otimes\omega_i\in \At^1(M),$$ we have
\begin{align*}
a\cdot\partial&{\left(m\otimes 1+\sum _i m_i\otimes\omega_i\right)}=a\cdot{\left(\partial(m)\otimes 1+\sum _i {\left(
\partial(m_i)\otimes \omega_i+m_i\otimes L_{\partial}(\omega_i)\right)}\right)}=\\
&=a\partial(m)\otimes 1+\partial(m)\otimes \dd a+\sum _i {\left(
\partial(m_i)\otimes a\omega_i+m_i\otimes aL_{\partial}(\omega_i)\right)}=\\
&=a\partial(m)\otimes 1+\sum _i {\left(\omega_i(\partial)m_i\otimes \dd a+
\partial(m_i)\otimes a\omega_i+m_i\otimes aL_{\partial}(\omega_i)\right)}=\\
&=a\partial(m)\otimes 1+\sum _i {\left(a\partial(m_i)\otimes\omega_i+m_i\otimes L_{a\partial}(\omega_i)\right)}=(a\partial){\left(m\otimes 1+\sum _i m_i\otimes\omega_i\right)},
\end{align*}
where we have used~\eqref{eq:weakLie} and~\eqref{eq:3273}. Thus, we
have shown that $\At^1(M)$ is an object in
$\DMod{\left(R,D_{R/k}\right)}$.

Now let us extend $\At^1(M)$ to an object $\phi^1(M)$ in
$\P_k^1\otimes_k \DMod{\left(R,D_{R/k}\right)}$, that is, let us
define a $\P_k^1$-module structure on~$\At^1(M)$ with respect to the
right homomorphism $r:k\to\P_k^1$. For this, note
that~$M\otimes_R\P_R^1$ is a $\P_R^1$-module.
In addition, the multiplication by $\P_k^1\subset \P_R^1$ preserves $\At^1(M)$: for $k\subset\P^1_k$ this follows from the existence of the $R$-linear structure on $\At^1(M)$, while, for any $$\eta\in \Omega_k\quad
\text{and}\quad m\otimes 1+\sum _i m_i\otimes\omega_i\in \At^1(M),$$ we have
$$
{\left(m\otimes 1+\sum _i m_i\otimes\omega_i\right)}\cdot \eta=
m\otimes\eta
$$
and $\eta(\xi)=0$ for any $\xi\in D_{R/k}$. Moreover, the
multiplication by $\P_k^1$ commutes with the $D_{R/k}$-structure on~$\At^1(M)$,
because the product on $\P_R^1$ respects the weak
$D_{R/k}$-structure via the Leibniz rule
(Section~\ref{subsection-Lieder}) and $$\xi(a+\eta)=0$$ in the above
notation. All together, this defines an object $\phi^1(M)$ in
$\P_k^1\otimes_k\DMod{\left(R,D_{R/k}\right)}$.

\subsection*{Step 2. The functor $M\mapsto \phi^1(M)$}

It follows that $\phi^1(M)$ depends functorially on $M$. Moreover, the explicit description of $\phi^1(M)$ from~\eqref{eq:paramAt} implies a functorial exact sequence in $\DMod{\left(R,D_{R/k}\right)}$:
$$
0\to \Omega_k\otimes_k M\to \phi^1(M)\stackrel{\pi}\longrightarrow M\to 0,\quad \pi{\left(m\otimes 1+\sum _im_i\otimes\omega_i\right)}=m.$$ It follows that the functor $\phi^1$ is exact. By construction, it is also $k$-linear with respect to the left homomorphism $l:k\to\P_k^1$, because the left $R$-linear structure on $\P_R^1$ is involved in the tensor product $M\otimes_R\P_R^1$.

Let us show that the functor $\phi^1$ is tensor. Let $M$ and $N$ be $D_{R/k}$-modules. We have a natural isomorphism
$$
{\left(M\otimes\P_R^1\right)}\otimes_{\P_R^1} {\left(N\otimes\P_R^1\right)}\stackrel{\sim}\longrightarrow (M\otimes_R N)\otimes_R\P^1_R\,.
$$
This induces a map
$$
\phi^1(M)\otimes_{\P^1_k}\phi^1(N)\to (M\otimes_R N)\otimes_R\P_R^1.
$$
The Leibniz rule for the action of $D_{R/k}$ on $M\otimes_R N$ implies that the image of this map lies in the subset
$$
\phi^1(M\otimes_R N)\subset (M\otimes_R N)\otimes_R\P_R^1,
$$
which defines a morphism of $\P_k^1$-modules
$$
m:\phi^1(M)\otimes_{\P^1_k}\phi^1(N)\to \phi^1(M\otimes_R N).
$$
Our aim is to show that $m$ is an isomorphism.
Note that the morphism $\pi$ from above coincides with taking modulo the ideal $\Omega_k\subset \P^1_k$. Also denote taking modulo the ideal in any $\P_k^1$-module by $\pi$. Then the morphism $m$ commutes with the identity map from $M\otimes_R N$ to itself via the corresponding morphisms~$\pi$. By Example~\ref{examp-tensBaer}, the kernel of $\pi$ on $$\phi^1(M)\otimes_{\P^1_k}\phi^1(N)$$
is equal to $$\Omega_k\otimes_k(M\otimes_R N).$$ It follows that the morphism $m$ induces the identity map from $\Omega_k\otimes_k(M\otimes_R N)$ to itself on the kernels of $\pi$. Therefore, $m$ is an isomorphism, which fixes a tensor structure for the functor $\phi^1$. Also, we obtain an isomorphism of tensor functors
$$
(e\otimes\id)\circ\phi^1\cong \id,
$$
where, as above, $e:\P_k^1\to k$ is taking modulo $\Omega_k$.

\subsection*{Step 3. Construction of $\At^2(M)$}

Put
$$
\At^2(M):=\At^1{\left(\At^1(M)\right)}\cap M\otimes_R\Qr_R\subset M\otimes_R\P_R^1\otimes_R\P_R^1.
$$
By Remark~\ref{rem:Lie}, the subring
$$\Qr_R\subset\P_R^1\otimes_R\P_R^1$$ is preserved under the action
of $D_{R/k}$, whence $\At^2(M)$ is a weak $D_{R/k}$-module. Besides,
as shown above, $\At^1{\big(\At^1(M)\big)}$ is a
$D_{R/k}$-module, whence $\At^2(M)$ is also a $D_{R/k}$-module.
Since $\At^1{\big(\At^1(M)\big)}$ is preserved under the right
multiplication by~$\P_k^1\otimes_k\P_k^1$, we obtain that $\At^2(M)$
is preserved under the right multiplication by
$$
\Qr_k\subset {\left(\P_k^1\otimes_k\P_k^1\right)}\cap \Qr_R.
$$
Since multiplication by $\P_k^1\otimes_k\P_k^1$ on $\At^1{\big(\At^1(M)\big)}$ commutes with the $D_{R/k}$-structure, multiplication by $\Qr_k$ commutes with the $D_{R/k}$-structure on $\At^2(M)$. Thus, we see that $\At^2(M)$ is a $\Qr_k$-submodule in $\At^1{\big(\At^1(M)\big)}$ in the category $\DMod{\left(R,D_{R/k}\right)}$.

It follows that $\At^2(M)$ depends functorially on $M$. By Step 2, the tensor structure on $\At^1\circ\At^1$ is induced by the isomorphism
$$
{\left(M\otimes_R\P_R^1\otimes_R\P_R^1\right)}\otimes_{\left(\P_R^1\otimes_R\P_R^1\right)}
{\left(N\otimes_R\P_R^1\otimes_R\P_R^1\right)}\stackrel{\sim}\longrightarrow
(M\otimes_R N)\otimes_R{\left(\P_R^1\otimes_R\P_R^1\right)}.
$$
Since $\P^2_R$ is a subring in $\P^1_R\otimes_R\P^1_R$, we see that the product map
$$
\At^1\left(\At^1(M)\right)\otimes\At^1\left(\At^1(N)\right)\to\At^1\left(\At^1(M\otimes_R N)\right)
$$
preserves $\At^2$.

Consider the filtration by ideals:
$$
\P^1_k\otimes_k\P^1_k\supset {\left(\Omega_k\otimes_k\P^1_k+\P^1_k\otimes_k\Omega_k\right)}\supset \Omega_k\otimes_k\Omega_k\supset 0.
$$
This defines a decreasing filtration on $\At^1{\big(\At^1(M)\big)}$ with the following adjoint quotients (see Section~\ref{subsection-explicitdefin} for more computational details):
$$
M,\quad \left(\Omega_k\otimes_kM\right)\oplus\left(\Omega_k\otimes_k M\right),\quad \Omega_k\otimes_k\Omega_k\otimes_k M.
$$
Consider the intersection of this filtration with $\At^2(M)$.
Since $\At^2(M)$ is contained in $M\otimes_R\Qr_R$, the corresponding adjoint quotients are contained in
$$
M,\quad \Omega_k\otimes_k M,\quad \Sym^2_k\Omega_k\otimes_k M.
$$
Hence, by Proposition~\ref{prop-expl}, $\DMod(R,D_{R/k})$ with
the functor $\At^2$ is a $D_k$-category, provided that the induced
map $$\At^2(M)\to M=\gr^0\At^1\left(\At^1(M)\right)$$ is surjective.

\subsection*{Step 4. Surjectivity of $\At^2(M)\to M$}

Take any
$$
m\in M\quad \text{and}\quad m\otimes 1+\sum _i m_i\otimes\omega_i\in\At^1(M)
.$$
First, let us prove that there exists $x\in M\otimes_R\Omega_R\otimes_R\Omega_R$ such that the image of $x$ under the map
$$
M\otimes_R\Omega_R\otimes_R\Omega_R\to M\otimes_R\wedge^2_R\Omega_R
$$
is equal to $$y:=\sum _i m_i\otimes\dd\omega_i$$ and the image of $x$ under the map
$$
M\otimes_R\Omega_R\otimes_R\Omega_R\to M\otimes_R\Omega_{R/k}\otimes_R\Omega_R
$$
is equal to $$z:=-\sum _i \nabla(m_i)\otimes \omega_i,$$ where we
apply the isomorphism $$\Omega_{R/k}\otimes_R M\cong
M\otimes_R\Omega_{R/k}.$$ For short, 
$$ A:=\Omega_R\otimes_R\Omega_R,\quad
B:=\Ker{\left(\Omega_R\otimes_R\Omega_R\to\wedge^2_R\Omega_R\right)},\quad
C:=\Ker{\left(\Omega_R\otimes_R\Omega_R\to \Omega_{R/k}\otimes_R\Omega_R\right)}.
$$
We have the following exact sequence
$$
A\to (A/B)\oplus(A/C)\to A\big/(B+C)\to 0,
$$
where the first map is given by the diagonal embedding and the
second arrow is induced by taking the difference. Since the
$R$-modules $$A=\Omega_R\otimes_R\Omega_R,\quad
A/B\cong\wedge^2_R\Omega_R,\quad A/C\cong
\Omega_{R/k}\otimes_R\Omega_R,\quad \text{and}\quad A/(B+C)\cong
\wedge^2_R\Omega_{R/k}$$ are projective and, henceforth, flat, we
obtain the exact sequence
$$
M\otimes_R \Omega_R\otimes\Omega_R\to
{\left(M\otimes_R\wedge^2_R\Omega_R\right)}\oplus
{\left(M\otimes_R\Omega_{R/k}\otimes_R\Omega_R\right)}\to
M\otimes_R\wedge^2_R\Omega_{R/k}\to 0.
$$
The integrability condition on $M$ implies that $y\oplus z$ is in
the kernel of the rightmost non-zero map (note that we have switched the tensor
factors $M$ and $\Omega_{R/k}$ unlike in
Definition~\ref{defin-diffmod}, whence there is a sign change).
Hence, by the exactness in the middle, there exists $x$ with the
required properties.

Now let us show that the element
$$
n:=m\otimes 1\otimes 1+\sum _i m_i\otimes\omega_i\otimes 1+\sum _i m_i\otimes 1\otimes\omega_i-x\in M\otimes_R\P_R^1\otimes_R\P_R^1
$$
belongs to $\At^2(M)$. Since $x$ is sent to $y$, we see that $n$
belongs to $M\otimes_R\Qr_R$. By the hypotheses,
$$
m\otimes 1\otimes 1+\sum _i m_i\otimes\omega_i\otimes 1 \in \At^1(M)\otimes 1\subset \At^1(M)\otimes_R\P_R^1.
$$
Since $x$ is sent to $z$, we see that the map
$$
\lambda\otimes\id_{\P^1_R}:M\otimes_R\P_R^1\otimes_R P_R^1\to \Omega_{R/k}\otimes_R M\otimes_R\P_R^1
$$
sends $\sum_i m_i\otimes 1\otimes\omega_i-x$ to zero (recall that
$\lambda$ is defined in~\eqref{eq:lambda}). Since $\P_R^1$ is a
projective and, therefore, flat $R$-module, we conclude
that
$$
\sum _i m_i\otimes 1\otimes\omega_i-x \in \At^1(M)\otimes_R\Omega_R.
$$
Therefore, $$n \in \At^1(M)\otimes_R\P_R^1.$$
It remains to check that
$$
n \in \At^1{\left(\At^1(M)\right)}.$$
For this, we need to show that, for any $\xi\in D_{R/k}$, we have
$$
\xi{\left(m\otimes 1+\sum _i
m_i\otimes\omega_i\right)}=\sum _i\omega_i(\xi)\cdot(m_i\otimes
1)-x(-\otimes\xi)\in \At^1(M),
$$
where $$x(-\otimes\xi)\in M\otimes_R\Omega_R\cong\Hom_R(D_R,M)$$
sends any $\partial\in D_R$ to $x(\partial\otimes\xi)\in M$.
By the explicit formula for the $D_{R/k}$-module structure on
$\At^1(M)$ given in Step 1, the left-hand side is equal to
$$
\xi(m)\otimes 1+\sum _i\xi(m_i)\otimes\omega_i+\sum _im_i\otimes L_{\xi}(\omega_i).
$$
By the explicit formula~\eqref{eq:3273} for the $R$-module structure on
$\At^1(M)$ also given in Step 1, the right-hand side is equal to
$$
\sum _i\omega_i(\xi)m_i\otimes 1+\sum _im_i\otimes\dd(\omega_i(\xi))-x(-\otimes\xi).
$$
Since $$m\otimes 1+\sum _i m_i\otimes\omega_i\in\At^1(M),$$ we have that
$$
\xi(m)\otimes 1=\sum _i\omega_i(\xi)m_i\otimes 1.$$
Further, by the definition of the Lie derivative, we have
$$
\sum _i m_i\otimes L_{\xi}(\omega_i)=\sum _i m_i\otimes\dd(\omega_i(\xi))+\sum _i m_i\otimes(\dd\omega_i)(\xi\wedge-).
$$
Since $x$ is sent to $y$, we have that
$$
\sum _i m_i\otimes(\dd\omega_i)(\xi\wedge-)=x(\xi\otimes-)-x(-\otimes\xi).
$$
Finally, since $x$ is sent to $z$, we have that
$$
x(\xi\otimes-)=-\sum _i\xi(m_i)\otimes\omega_i,$$
which shows the required equality.

\subsection*{Step 5. The forgetful functor $\DMod{\left(R,D_{R/k}\right)}\to\Mod(R)$}

It remains to show that the forgetful functor $\DMod{\left(R,D_{R/k}\right)}\to\Mod(R)$ is differential. By Definition~\ref{defin-Dkfunctor} and~\eqref{eq:explextscal} from Section~\ref{subsection-prelextscal}, it is enough to show that the canonical morphism of $\P_R^2$-modules
$$
\At^2(M)\otimes_{\left(\P^2_k\otimes_k R\right)}\P^2_R\to M\otimes_R\P_R^2
$$
is an isomorphism. This follows directly from
Lemma~\ref{lemma-filtr} applied to the filtered ring $\Qr_R$.
\end{proof}

\begin{remark}\label{rem-parmaconst}
\hspace{0cm}
\begin{enumerate}
\item\label{it:3717}
If $(R,D_R)=(k,D_k)$, then we have
$\DMod{\left(R,D_{R/k}\right)}=\Vect(k)$. It follows from the
construction in Step 1 in the proof of
Theorem~\ref{teor-paramAtiayh} that the $D_k$-structure on
$\DMod{\left(R,D_{R/k}\right)}$ given by
Theorem~\ref{teor-paramAtiayh} coincides with the usual
$D_k$-structure on~$\Vect(k)$.
\item
There is a motivating example for the construction of a
$D_{R/k}$-structure on $\DMod(R,D_{R/k})$. Let $M$ be a $D_R$-module
over~$R$ and put $N:=M^{D_{R/k}}$ (Definition~\ref{defin-diffmod}).
Note that $N$ is a $k$-vector subspace in $M$. Moreover, there is a
$D_k$-module structure on $N$ over $k$ defined as follows. For
$\partial\in D_k$, consider any lift $\tilde\partial\in D_R$ of
$1\otimes\partial$ with respect to the structure map $D_R\to
R\otimes_kD_k$. 
Then, for any $n\in N$, put
$$\partial(n):=\tilde{\partial}(n).$$ In
Theorem~\ref{teor-paramAtiayh}, $M$ is replaced by the category
$\Mod(R)$ and, correspondingly, $N$ is replaced by
$\DMod(R,D_{R/k})$. It seems that both constructions can be
generalized for a wider class of $D_R$-objects or categories instead
of $M$ or $\Mod(R)$.
\item
In~\cite[1.6.3]{Besser}, one finds an alternative definition
of the $D_{R/k}$-module structure on $\At^1(M)$ in terms of lifts of
the $D_{R/k}$-structure on $M$ to, possibly, non-integrable
$D_R$-structures on $M$. The construction from op.cit. is given for
families of varieties but it applies in the setting of
parameterized differential algebras as well. However, the approach to
$\At^1(M)$ from Step~1 of the proof of
Theorem~\ref{teor-paramAtiayh} seems to be more convenient to show
that one, thus, obtains  a $D_k$-structure on $\DMod(R,D_{R/k})$.
\end{enumerate}
\end{remark}

In Section~\ref{subsection-PPVff}, we use the following result.

\begin{lemma}\label{lemma-changeofparam}
Given a morphism $(R,D_R)\to (S,D_S)$ of parameterized differential
algebras over $(k,D_k)$, the extension of scalars functor
(Definition~\ref{defin-extscaldiffmod})
$$
S\otimes_R-:\DMod{\left(R,D_{R/k}\right)}\to
\DMod{\left(S,D_{S/k}\right)}
$$
is canonically a differential functor between $D_k$-categories over
$(k,D_k)$.
\end{lemma}
\begin{proof}
For a $D_{R/k}$-module $M$, consider the morphism
$$
M\otimes_R\P_R^1\to M_S\otimes_S\P_S^1=M\otimes_R\P_S^1.
$$
It follows that this morphism sends $\At^1(M)$ to $\At^1(M_S)$.
Hence, the morphism
$$
M\otimes_R\P_R^2\to M_S\otimes_S\P_S^2=M\otimes_R\P_S^2
$$
sends $\At^2(M)$ to $\At^2(M_S)$. Thus, we obtain a morphism of
$\P_S^2$-modules
$$
\At^2(M)\otimes_{\left(\P_R^2\otimes_R S\right)}\P_S^2\to
\At^2(M_S).
$$
By Lemma~\ref{lemma-filtr} applied to the filtered ring $\Qr_S$,
this is an isomorphism.
\end{proof}

\subsection{Matrix description}
Let us describe the differential structure on $\At^1(M)$ in the case
of a parameterized field explicitly. In the particular case
when~$D_k$ is one-dimensional, this will coincide with the
prolongation functor from~\cite[Section~5]{OvchTannakian}. Let
$(K,D_K)$ be a parameterized differential field over $(k,D_k)$. Let
$\partial_{t,1},\ldots,\partial_{t,q}$ be a basis of~$D_k$ over $k$,
and let
$$
\partial_{x,1},\ldots,\partial_{x,p},\tilde{\partial}_{t,1},\ldots,\tilde{\partial}_{t,q}
$$
be a basis of $D_K$ over $K$ such that $\tilde\partial_{t,i}$ are
sent to $1\otimes\partial_{t,i}$ under the structure map $D_K\to
K\otimes_k D_k$. Let $\omega_{t,1},\ldots,\omega_{t,q}$ be the dual basis in
$\Omega_k$ to $\partial_{t,1},\ldots,\partial_{t,q}$, and let
$$
\widetilde \omega_{x,1},\ldots,\widetilde \omega_{x,p},\omega_{t,1},\ldots,\omega_{t,q}
$$
be the dual basis in $\Omega_K$ to
$\partial_{x,1},\ldots,\partial_{x,p},\tilde
\partial_{t,1},\ldots,\tilde \partial_{t,q}$. Thus, we have $\widetilde
\omega_{x,i}{\big(\tilde\partial_{t,j}\big)}=0$.

Let $M$ be a finite-dimensional $D_{K/k}$-module over $K$ and 
$\{e_1,\ldots,e_m\}$ be a basis of $M$ over $K$. For $\partial\in
D_{K/k}$, let $A_\partial\in \Mat_{m\times m}(K)$ be the connection
matrix on $M$ \cite[Section~1.2]{Michael}, that is, we have
$$
\partial{\left(\overline e\right)}=-\overline e\cdot A_{\partial},
$$
where $\overline e:=(e_1,\ldots,e_m)$. Put
$A_i:=A_{\partial_{x,i}}$, $1\leqslant i\leqslant p$. Then we obtain
the following basis for $\At^1(M)$:
$$
\big\{f_1,\ldots,f_m,e_i\otimes \omega_{t,j}\big\},\quad 1\leqslant i\leqslant m,
\,1\leqslant j\leqslant q,\quad
(f_1,\ldots,f_m)=\overline f,\quad
\overline f:=\overline e\otimes 1-\sum_{i=1}^p \overline e\cdot
A_i\otimes \widetilde \omega_{x,i}.
$$

\begin{proposition}
In the above basis for $\At^1(M)$, the connection matrix for $\partial\in D_{K/k}$ is equal to
$$
\begin{pmatrix}
A_{\partial}&0&\ldots&0&0\\
B_1&A_{\partial}&\ldots&0&0\\
\vdots&&\ddots&&\vdots\\
B_{q-1}&0&\ldots&A_{\partial}&0\\
B_q&0&\ldots&0&A_{\partial}
\end{pmatrix},\quad
B_i:=-\tilde\partial_{t,i}(A_{\partial})-A_{\left[\partial,\tilde\partial_{t,i}\right]},\quad
1\Le i\Le q.
$$
\end{proposition}
\begin{proof}
We use the construction of the differential structure on $\At^1(M)$
as given in Step 1 of the proof of Theorem~\ref{teor-paramAtiayh}.
By definition, we have
$$\partial{\big(e_i\otimes \omega_{t,j}\big)}=\partial(e_i)\otimes \omega_{t,j}$$ and
$$\partial{\left(\overline f\right)}=-\overline e\cdot A_\partial \otimes 1-\sum_{i=1}^p \overline e\cdot\partial(A_i)
\otimes\widetilde \omega_{x,i}+ \sum_{i=1}^p\overline e\cdot  A_iA_\partial
\otimes\widetilde \omega_{x,i}-\sum_{i=1}^p\overline e\cdot A_i\otimes
L_{\partial}{\left(\widetilde \omega_{x,i}\right)}.
$$
On the other hand, by the definition of $K$-linear structure~\eqref{eq:3273} on $\At^1(M)$, we have
$$
\overline f\cdot A_\partial=\overline e\cdot A_\partial\otimes
1+\overline e\otimes\dd A_\partial-\sum_{i=1}^p\overline e\cdot
A_i\otimes A_\partial \widetilde \omega_{x,i}.
$$
Since the action of $\partial$ is well-defined on $\At^1(M)$, the sum
$$
\partial{\left(\overline f\right)}+\overline f\cdot A_\partial=-\sum_{i=1}^p \overline e\cdot\partial(A_i)
\otimes\widetilde \omega_{x,i}-\sum_{i=1}^p\overline e\cdot A_i\otimes
L_{\partial}{\left(\widetilde \omega_{x,i}\right)}+\overline e\otimes\dd
A_\partial
$$
belongs to $M\otimes_k\Omega_k$ and, hence, it is uniquely determined by its values at all $\tilde\partial_{t,j}$. Evaluating this explicitly and using that $\widetilde \omega_{x,i}{\big(\tilde\partial_{t,j}\big)}=0$, we obtain the needed result.
\end{proof}

\subsection{PPV extensions and differential functors}\label{subsection-PPVff}

The following statement is a parameterized version
of~\cite[9.6]{DeligneFS} (see also
Proposition~\ref{prop-Deligne2}).

\begin{theorem}\label{theor-PPVdff}
Let $(K,D_K)$ be a parameterized differential field over a
differential field $(k,D_k)$, \mbox{$\Char k=0$}, $M$ be a
finite-dimensional $D_{K/k}$-module over $K$. Then there is an
equivalence of categories
$$
\Phi:\PPV(M)\stackrel{\sim}\longrightarrow
\Fun^{D}_k(\Cat,\Vect(k)),
$$
where $\Cat:=\langle M\rangle_{\otimes,D}$ is the full subcategory
in $\DMod(K,D_{K/k})$ $D_k$-tensor generated by~$M$
(Definition~\ref{defin-diffgenercat}), where  the
$D_k$-structure on $\DMod(K,D_{K/k})$ is as in
Theorem~\ref{teor-paramAtiayh}.
\end{theorem}
\begin{proof}
First, let us construct the functor $\Phi$. Let $L$ be a PPV
extension for $M$. By construction, the solution space functor
$$
\omega_0:\Cat\to\Vect(k),\quad X\mapsto X_L^{D_{L/k}}
$$
 is $k$-linear. By definition of a PPV extension, there is a canonical isomorphism
\begin{equation}\label{eq-isom}
L\otimes_k\omega_0(X)\stackrel{\sim}\longrightarrow X_L
\end{equation}
in $\DMod(L,D_{L/k})$. Therefore, the functor $\omega_0$ is exact
and tensor. Let us show that $\omega_0$ is a differential functor
between $D_k$-categories over $k$. By
Remark~\ref{rem-parmaconst}\eqref{it:3717} and
Lemma~\ref{lemma-changeofparam} applied to the morphism $(k,D_k)\to
(L,D_L)$, 
$$
L\otimes_k-:\Vect(k)\to \DMod{\left(L,D_{L/k}\right)}
$$
is a differential functor. Since the functor $L\otimes_k-$ is also fully
faithful, by Lemma~\ref{corol-composdifffunc}, it is enough to prove
that the composition
$$
(L\otimes_k-)\circ\omega_0:\Cat\to\DMod{\left(L,D_{L/k}\right)}
$$
is a differential functor. By isomorphism~\eqref{eq-isom}, this
composition is isomorphic to the extension of scalars functor
$$
L\otimes_K-:\Cat\to\DMod{\left(L,D_{L/k}\right)}.
$$
By Lemma~\ref{lemma-changeofparam}, the $L\otimes_K-$ is a
differential functor, which implies that $\omega_0$ is a
differential functor. We put $\Phi(L):=\omega_0$. One checks that
$\Phi$ extends to a functor.

Now let us construct a quasi-inverse functor $\Psi$ to $\Phi$. Let
$\omega_0:\Cat\to\Vect(k)$ be a differential functor. Consider
the forgetful functor $\omega:\Cat\to\Vect(K)$. By
Theorem~\ref{teor-paramAtiayh}, $\omega$ is a differential functor.
By Theorem~\ref{theor-extscal}, there exists the extension of
scalars $K\otimes_k\Cat$. By
Proposition~\ref{prop-diffextscal}\eqref{en:2368}, $K\otimes_k\Cat$
has a canonical $D_K$-structure and, by
Proposition~\ref{prop-diffextscal}\eqref{en:2375}, the functor
$\omega$ corresponds to a differential functor
$$
\eta:K\otimes_k\Cat\to\Vect(K)
$$
between $D_K$-categories over $K$. By
Remark~\ref{remark-extescalcatdiff}\eqref{en:2422}, we also have a
differential functor
$$
K\otimes_k\omega_0:K\otimes_k\Cat\to K\otimes_k\Vect(k)=\Vect(K)
$$
between $D_K$-categories over $K$. By Proposition~\ref{prop-reprdiffalg}, the functor
$$
\Isom^{\otimes,D}(K\otimes_k\omega_0,\eta):\DAlg(K,D_K)\to \Sets
$$
is corepresented by a $D_K$-algebra $A$ over $K$. We will show that
$A$ is a domain and $L:=\Frac(A)$ is a PPV extension for $M$. For
this, we use analogous results from \cite[9]{DeligneFS}. By
Proposition~\ref{prop-reprdiffalg}, $A$ as a $K$-algebra
corepresents the functor
$$
\Isom^{\otimes}(K\otimes_k\omega_0,\eta):\Alg(K)\to\Sets.
$$
By Definition~\ref{defin-extcat}, there is an equivalence of categories
$$
\Fun^{r,\otimes}_K(K\otimes_k\Cat,\Vect(K))
\stackrel{\sim}\longrightarrow\Fun^{r,\otimes}_k(\Cat,\Vect(K)),
$$
which sends $K\otimes_k\omega_0$ to $(K\otimes_k-)\circ\omega_0$ and
sends $\eta$ to $\omega$ by the construction of $\eta$. Therefore,
$A$ corepresents the functor
\begin{equation}\label{eq:isomKk}
\Isom^{\otimes}((K\otimes_k-)\circ\omega_0,\omega):\Alg(K)\to\Sets.
\end{equation}
Let $\Cat_i$ be the full subcategory in $\Cat$ tensor generated by
$\big(\At^{1}\big)^{\circ\,i}(M)$ and let $A_i$ be the
$K$-algebra that corepresents the functor
$$
\Isom^{\otimes}\big((K\otimes_k-)\circ\omega_0|_{\Cat_i},\omega|_{\Cat_i}\big).
$$
By Remark~\ref{remark-difftensgen}, the category $\Cat$ is a union
of all $\Cat_i$'s, whence we have that $$A=\varinjlim_i A_i.$$ Each
ring~$A_i$ is a particular example of a ring considered
in~\cite[9.2]{DeligneFS}, where it is denoted by $\Gamma(P,{\mathcal
O})$. A $D_{K/k}$-structure on~$A_i$ is defined
in~\cite[9.2]{DeligneFS}. Moreover, all morphisms $$A_i\to A_j,\quad
i\Le j,$$ are morphisms of $D_{K/k}$-algebras over $k$, which
defines a $D_{K/k}$-structure on $A$. By
Example~\ref{examp-diffstrdmod}, this $D_{K/k}$-structure coincides
with the one obtained from the $D_K$-structure on $A$. Thus, it
follows from \cite[9.3]{DeligneFS} that $A$ is a domain and the
field $L:=\Frac(A)$, being a $D_K$-field over~$K$, has no new
$D_{K/k}$-constants. Since $A$ corepresents
functor~\eqref{eq:isomKk}, the embedding $A\hookrightarrow L$
induces an isomorphism
$$
L\otimes_k\omega_0(M)\cong M_L.
$$
It follows from~\cite[9.6]{DeligneFS} that this
isomorphism identifies $1\otimes\omega_0(M)$ with $M_L^{D_{L/k}}$.
Thus, we have an isomorphism
$$
L\otimes_k M_L^{D_{L/k}}\to M_L.$$
Hence, by Proposition~\ref{prop-difffingencat}, $L$ is $D_K$-generated by the
coordinates of horizontal vectors in a basis of~$M$ over $K$, whence
$L$ is a PPV extension. We put $\Psi(\omega_0):=L$. One checks that
$\Psi$ extends to a functor. The proof of the fact that $\Phi$ and
$\Psi$ are quasi-inverses of each other is the same as the proof
of~\cite[Proposition~9.5]{DeligneFS}.
\end{proof}

\begin{remark}\label{rem-PPVring}
It follows from the proof of Theorem~\ref{theor-PPVdff} and Proposition~\ref{prop-difffingencat} that $A$ as above is equal to the PPV ring associated with $L$ (Definition~\ref{defin-PPVring}). Moreover, by the construction of $A$, for any $D_k$-algebra $R$, there is a canonical isomorphism
$$
\Aut^{D_K}(R\otimes_k A/R\otimes_k K)\cong \fIsom^{\otimes,D}(\omega_R,\omega_R). 
$$
\end{remark}

\section{Definability of differential Hopf algebroids}\label{sec:diffHopfalgebroids}

\subsection{Reduction to faithful flatness}

The goal of this section is to prove
Theorem~\ref{theor-defindiffHopfalg}. This technical result is
needed for the proof of Theorem~\ref{theor-main}.

\begin{theorem}\label{theor-defindiffHopfalg}
Let $(K,H)$ be a $D_k$-Hopf algebroid (Example~\ref{example-diffobjects}\eqref{en:6104}) over a differential field $(k,D_k)$ with $K$ being a field and $\Char k=0$. Suppose that $H$ is a $D_k$-finitely generated (Definition~\ref{defin-strictalg}) faithfully flat algebra over $K\otimes_k K$.
Then there exist a $D_k$-finitely generated subalgebra $R$ in $K$ over $k$ and a $D_k$-Hopf algebroid $(R,A)$ over $k$ such that $A$ is a $D_k$-finitely generated faithfully flat algebra over $R\otimes_k R$ and there is an isomorphism of $D_k$-Hopf algebroids over $k$
$$
{\left(K,{}_KA_K\right)}\cong (K,H).
$$
\end{theorem}

The following statement is not used in the paper, but we include it for its own interest.

\begin{corollary}
Let $\Cat$ be a $D_k$-Tannakian category over a differentially closed field $(k,D_k)$ with $\Char k=0$. Suppose that $\Cat$ is $D_k$-tensor generated by one object. Then there exists a differential (fiber) functor $\Cat\to \Vect(k)$.
\end{corollary}

\begin{proof}
 There is a $D_k$-morphism from any $D_k$-algebra over $k$ to a $D_k$-field over $k$. Thus, it follows from Definition~\ref{defin-diffTancat} that there is a differential functor $\Cat\to \Vect(K)$ for a $D_k$-field $K$ over $k$. Combining Theorem~\ref{theor-diffTanncorr}, Proposition~\ref{prop-difffingencat},  Theorem~\ref{theor-defindiffHopfalg}, Section~\ref{subsection-prelHopfalg}, and Example~\ref{examp-difffunctor}\eqref{en:2377}, we obtain a differential functor $\Cat\to \Mod(R)$, where $R$ is a $D_k$-finitely generated $D_k$-algebra over $k$. Since $\Char k= 0$, there
is a morphism from $R$ to $k$ (for example, see
\cite[Definition~4]{TrushinSplitting} and the references given
there), which finishes the proof.
\end{proof}

The proof of Theorem~\ref{theor-defindiffHopfalg} uses the following statements.

\begin{lemma}\label{lemma-finiteHopf}
Let $B$ be a $D_k$-finitely generated $D_k$-Hopf algebra over a differential field $(k,D_k)$ with $\Char k=0$. Then $B$ is of $D_k$-finite presentation over $k$ (Definition~\ref{defin-difffinpres}).
\end{lemma}
\begin{proof}
By~\cite[Proposition 12]{Cassidy}, $B$ is a quotient of the Hopf algebra of the differential algebraic group $\GL_n$, that is, we have a surjective morphism of $D_k$-Hopf algebras
$$
C:=k\{T_{ij}\}[1/\det]\stackrel{\varphi}\longrightarrow B.
$$
Since $C$ is of $D_k$-finite presentation, it is enough to prove that the kernel $I$ of $\varphi$ is $D_k$-finitely generated. Let $C_n\subset C$ be the subring generated over $k$ by all  derivatives of $T_{ij}$ of order at most $n$ with respect to $D_k$. Put $J_n:=I\cap C_n$. Then $J_n$ is a finitely generated Hopf ideal \cite[Section~2.1]{Water} in the finitely generated Hopf algebra $C_n$ over $k$, because the comultiplication $\Delta : C\to C\otimes_k C$  is a $D_k$-morphism. 

Let $I_n$ be the $D_k$-ideal in $C$ generated by $J_n$. Again, since $\Delta$ is a $D_k$-morphism, $I_n$ is a $D_k$-finitely generated Hopf ideal in the Hopf algebra $C$ over $k$. Therefore, $I_n$ is radical~\cite[Theorem~11.4]{Water}. Since $I=\bigcup_n I_n$, by~\cite[Theorem~7.1]{Kap}, $I=I_n$ for some $n$, whence $I$ is $D_k$-finitely generated.
\end{proof}

\begin{lemma}\label{lemma-finitepres}
Let $(K,H)$ be a $D_k$-Hopf algebroid over a differential field
$(k,D_k)$ with $K$ being a field and $\Char k=0$. Suppose that $H$
is a $D_k$-finitely generated faithfully flat algebra over
$K\otimes_k K$. Then $H$ is of $D_k$-finite presentation over
$K\otimes_k K$.
\end{lemma}
\begin{proof}
Since $H$ is $D_k$-finitely generated over $K\otimes_k K$, we have that $$B:=K\otimes_{K\otimes K}H$$ is a $D_k$-finitely generated $D_k$-Hopf algebra over $K$. Therefore, $B$ is of $D_k$-finite presentation over $K$ by Lemma~\ref{lemma-finiteHopf}. Since $\Spec(H)$ is a $D_k$-pseudo-torsor under the group scheme $\Spec(B\otimes_k K)$ over $K\otimes_k K$ (Section~\ref{subsection-prelTann}), we have an isomorphism of $D_k$-algebras over $H$:
$$
B\otimes_K H \cong H\otimes_{K\otimes K}H.
$$
Hence, $H\otimes_{K\otimes K}H$ is of $D_k$-finite presentation over
$H$. By the hypotheses of the lemma, $H$ is faithfully flat over
$K\otimes_k K$. The same argument as in the non-differential case
(for example, see~\cite[Proposition 2.7.1(vi)]{EGAIV2}) implies that
$H$ is of $D_k$-finite presentation over $K\otimes_k K$.
\end{proof}

\begin{proposition}\label{prop-localizealgebroid}
Let $(R,A)$ be a $D_k$-Hopf algebroid over a differential field
$(k,D_k)$ with $R$ being a domain and $\Char k=0$. Suppose that $R$
and $A$ are $D_k$-finitely generated over $k$ and $A_F\ne 0$, where
$F$ is the total fraction ring of $R\otimes_k R$. Then there exists
a non-zero element $f\in R$ such that the localization~$_fA_f$ is
faithfully flat over the localization $R_f\otimes_k R_f$.
\end{proposition}

\begin{proof}[Proof of Theorem~\ref{theor-defindiffHopfalg}]
By Lemma~\ref{lemma-finitepres}, $H$ is of $D_k$-finite presentation over $K\otimes_k K$.
A standard argument implies that there is a $D_k$-finitely generated subalgebra $R$ in $K$ over $k$ and a $D_k$-Hopf algebroid $(R,A)$ over $k$ such that $A$ is of $D_k$-finite presentation over $R\otimes_k R$ and there is an isomorphism of $D_k$-Hopf algebroids
$$
(K,{}_KA_K)\cong (K,H)
$$
over $k$.
Since $H$ is faithfully flat over $K\otimes_k K$, we have $A_F\ne 0$. Hence, by Proposition~\ref{prop-localizealgebroid},
localizing~$R$ by a non-zero element, we obtain that $A$ is faithfully
flat over~$R\otimes_k R$.
\end{proof}

\begin{remark}\label{remark-algfiniteflat}
Proposition~\ref{prop-localizealgebroid} is implied by the following
hypothetical statement: given a morphism $S\to A$ between
\mbox{$D_k$-fi\-nite\-ly} generated algebras over $k$, suppose that
there is a multiplicative set $\Sigma\subset S$ such that the
localization $\Sigma^{-1}A$ is faithfully flat over $\Sigma^{-1}S$;
then there is $g\in \Sigma$ such that $A_g$ is faithfully flat
over~$S_g$. The validity of this statement seems to be not clear,
while its non-differential version is well-known (for example, see
\cite[8.10.5(vi), 11.2.6.1(ii)]{EGAIV3}).
Proposition~\ref{prop-localizealgebroid} provides the partial
case of the above statement in the remark in which $A$ comes from the differential Hopf algebroid $(R,A)$ and $S=R\otimes_k R$.
\end{remark}

The rest of the section  is on
the proof of Proposition~\ref{prop-localizealgebroid}, which we
actually prove in Section~\ref{sec:proofof72}.

\subsection{Auxiliary results}\label{subsection-geomresults}

The following is a modification
of~\cite[Proposition~5]{TrushinInheriting}. The authors are grateful
to D.~Trushin for his suggestion to use this result.

\begin{lemma}\label{lemma-Trushin}
Let $A$ be a $D_S$-finitely generated algebra over a differential
ring $(S,D_S)$ and  let $a_1,\dots, a_p$ be its generators. Suppose that~$A$ is
a domain. Consider the following (non-differential) $S$-subalgebras
in~$A$:
$$
A_n:= S[
(\partial_1\cdot\ldots\cdot\partial_m)(a_i)\:|\:\partial_j\in D_S,\
m\le n,\ 1\Le i\Le p]\,,\,\,n\in\N\,.
$$
Then there exist a natural number $N$ and a non-zero element $g\in A_N$ such that, for any
$n\ge N$, there is an isomorphism
$${(A_{n+1})}_g \cong
{(A_n)}_g{\left[T_1,\ldots,T_{l_n}\right]}
$$
of algebras over the localization ${(A_n)}_g$, where the $T_l$'s are
formal variables.
\end{lemma}
\begin{proof}
Replacing $S$ by its image under the homomorphism $S\to A$, we may
assume that this homomorphism is injective and~$S$ is a domain. Let
$\p$ be the kernel of the surjective morphism of
$D_S$-algebras over~$S$
$$
\varphi : B\to A,\quad y_i\mapsto a_i,
$$
where $B:=S\{y_1,\ldots,y_p\}$ (Definition~\ref{defin-strictalg}). Then $\p$ is a
prime $D_S$-ideal. For a natural $n$, put
$$
B_n:=
S[(\partial_1\cdot\ldots\cdot\partial_m)(y_i)\:|\:\partial_j\in
D_S,\ m\le n,\ 1\Le i\Le p]\subset B.
$$
Then we have $$A_n\cong B_n\big/{\left(\p\cap B_n\right)}.$$
Since $D_S$ is a finitely generated projective $S$-module,
localizing $S$ by a non-zero element, assume that $D_S$ is now a
finitely generated free $S$-module. Let
$$D_S=S\cdot\delta_1\oplus\ldots \oplus S\cdot\delta_d.$$ Then we have
$$
[\delta_i,\delta_j]=\sum_{q=1}^d c_{ij}^q \delta_q,\quad c_{ij}^q
\in S,\ 1\Le i,j\Le d,
$$
which is exactly the situation considered in~\cite{hubert03}.

For every $D_S$-polynomial $f\in B\setminus S$, we define its
leader, separant, and initial as in~\cite[Section~3.2]{hubert03}.
More precisely, put
$$
\Theta:=
\{\id\}\cup\left\{\delta_{i_1}\cdot\ldots\cdot\delta_{i_m}\:|\:1\Le
i_j\Le d,\ m\Ge 1\right\},\quad M:= \{\,\theta y_j\:|\:\theta \in \Theta,\
1\Le j\Le p\,\}\subset B,
$$ and let the order of
$$\delta_{i_1}\cdot\ldots\cdot \delta_{i_m} y_j\in M$$ be $m$. Thus,
$B$ is the ring of polynomials in elements of $M$. Consider an
orderly differential ranking on $M$ \cite[Definition~3.3]{hubert03},
for example, the ranking that first compares the orders of two
elements in $M$ and then compares lexicographically the $y_j$'s and
$\delta_i$'s. If $u_f \in M$ is the leader of $f$ with respect to
this ranking on $M$ and $$f=I_r u_f^r +\ldots+I_0,$$ 
where $I_{i}$, $0\Le i\Le r$, do not contain $u_f$,
then the
separant is $S_f:=\partial f/\partial u_f$ and the initial $I_f$ is
$I_r$. Let $\Sigma \subset \p$ be a characteristic set of
$\p$ with respect to our ranking, \cite[Section~6.3]{hubert03}, and put
$$
W:= \{u_f\:|\: f \in \Sigma\}, \ \ Z:=\left\{\theta u_f\:|\: f\in
\Sigma,\ \theta\in \Theta,\ \theta \ne \id\right\},\ \ X:=
M\setminus(Z\cup W),\ \ \text{and}\ \ \tilde g:= \prod_{f\in \Sigma}
I_f S_f.
$$
Note that $\tilde g\notin \p$, because the differential ideal $\p$ is prime.
By \cite[Section~6.1]{hubert03},
for every $f\in \p$, there exists  $q\Ge 0$ such that
\begin{equation}\label{eq-radical}
\tilde g^{\,q}\cdot f=\sum _i h_i\cdot{(\theta_i f_i)}^{n_i},
\end{equation}
for some $h_i\in B$, $\theta_i\in\Theta$, $f_i\in\Sigma$, and
$n_i>0$, where the polynomials $h_i$'s are free of the elements
of~$Z$. Let $N\in\N$ be such that $B_N\supset W$, that is, $N$ is
the maximal order of the elements of $W$. Since the ranking is
orderly, this implies that $\Sigma \subset B_N$. Further,~$I_f$ and
$S_f$ belong to $B_N$ for any $f\in\Sigma$, because they are
differential polynomials  of order not exceeding $N$. Hence, $\tilde
g\in B_N$. Put
$$
g:=\varphi(\tilde g).
$$
We have that $g\ne 0$, because $\tilde g \notin \p$ as shown
above. Since, again, the ranking is orderly, the localization~$A_g$
is generated by $\varphi(W)$, $\varphi(X)$, and $1/g$ over $S$.
Moreover, if $$f\in S[W\cup X]\subset B$$ is such that $\varphi(f)=0$
in~$A_{g}$, then~\eqref{eq-radical} implies that there exists
$q\Ge0$ such that $$\tilde g^q f\in (\Sigma),$$  the
(non-differential) ideal generated by~$\Sigma$. Therefore,
$$
A_{g} \cong S[W\cup X]_g/(\Sigma)
$$
as $S$-algebras. Thus, for every $n\Ge N$, we have that
${(A_{n+1})}_{g}$ is a polynomial ring over ${(A_n)}_{g}$. Precisely, we have $${(A_{n+1})}_{g}={(A_n)}_{g}[T],$$ where
$T:=(\varphi(X)\cap A_{n+1})\setminus A_n$.
\end{proof}

We use the following notation and conventions in our geometric
constructions. Given morphisms of schemes $\varphi:Y\to X$ and
$\pi:Z\to X$,  denote the fibred product $Y\times_X Z$
by $\varphi^*Z$ and
 the projection to $Y$ by $\varphi^*\pi:\varphi^*Z\to Y$. Thus, there is a Cartesian square of
schemes
$$
\begin{CD}
\varphi^*Z@>>>Z\\
@V\varphi^*\pi VV@V\pi VV\\
Y@>\varphi>>X.
\end{CD}
$$
The morphism $\varphi^*\pi$ is usually called a base change of $\pi$ by the morphism
$\varphi$. The notation $\varphi^*Z$ is correct, provided that $\pi$ is the only considered morphism from $Z$ to $X$.

Given a morphism of schemes $\varphi:Y\to X$ and an open or closed
subscheme $U\subset Y$,  denote the restriction of
the morphism~$\varphi$ to $U$ by $\varphi|_U$.
Given an open or closed
subscheme $W\subset X$,  denote the restriction of $Y$ to
$W$, that is, the preimage $\varphi^{-1}(W)$, by~$Y_W$, and denote the morphism $\varphi|_{Y_W}$ by
$$\varphi_W:Y_W\to W.$$ In
particular, if $x$ is a point in $X$, then $Y_x$ denotes the fiber
of $\varphi$ over $x$ considered as a scheme over the residue field $k(x)$ at $x$.

Given a scheme $X$, denote the projection to the $i$-th factor by
$$p_i:X\times X\to X,\quad i=1,\: 2.$$  Denote the projection to the
product of the $i$-th and $j$-th factors by
$$p_{ij}:X\times X\times X\to X\times X,\quad1\Le i<j\Le 3.$$ Given
a scheme $X$ and a field $F$, denote the set of $F$-points of $X$
by~$X(F)$. That is, an element in~$X(F)$ is a morphism of schemes
$\Spec(F)\to X$.

Recall that a
morphism is faithfully flat if and only if it is both flat and surjective.
A base change~$\varphi^*\pi$ of a
(faithfully) flat morphism $\pi$ by any morphism $\varphi$ is
(faithfully) flat. Further, if a composition of morphisms of schemes
$$
W\stackrel{\lambda}\longrightarrow Z\stackrel{\pi}\longrightarrow X,
$$
is (faithfully) flat with $\lambda$ being faithfully flat, then $\pi$ is (faithfully) flat. Also, we will use the following fact.

\begin{lemma}\label{lemma-critflat}
Consider a Cartesian square of schemes
$$
\begin{CD}
\varphi^*Z@>\pi^*\varphi>>Z\\
@V\varphi^*\pi VV@V\pi VV\\
Y@>\varphi>>X.
\end{CD}
$$
Suppose that $\varphi$ is faithfully flat and there is an open subset
$W\subset \varphi^*Z$ such that the morphism
\mbox{$(\varphi^*\pi)|_W:W\to Y$} is (faithfully) flat and the
morphism $(\pi^*\varphi)|_W:W\to Z$ is surjective. Then $\pi$ is
(faithfully) flat. In particular, if $\varphi^*\pi$ is (faithfully)
flat, then~$\pi$ is (faithfully) flat.
\end{lemma}
\begin{proof}
The morphism $\pi^*\varphi:\varphi^*Z\to Z$ is faithfully flat, being the base change of the faithfully flat morphism $\varphi$ by the morphism $\pi$. Therefore, the morphism
$$
(\pi^*\varphi)|_W:W\to Z
$$
is also faithfully flat, being both flat and surjective. On the other hand, the composition $\pi\circ(\pi^*\varphi)|_W$ is (faithfully) flat, because it is equal to the composition of (faithfully) flat morphisms $\varphi\circ(\varphi^*\pi)|_W$. Therefore, the morphism $\pi$ is (faithfully) flat.
\end{proof}

\begin{definition}\label{def:pseudotorsor}
Let $G\to X$ be a group scheme over a scheme $X$. Suppose that we
are given an action of $G$ on a scheme $T\to X$ over $X$, that is, a
morphism
$
a:G\times_X T\to T
$
that satisfies the group action condition. We say that $T$ is a {\it
pseudo-torsor under $G$} if the morphism
$$
{\left(a,{\rm pr}_T\right)}:G\times_X T\to T\times_X T
$$
is an isomorphism, where ${\rm pr}_T$ is the projection to $T$.
\end{definition}

\begin{lemma}\label{lemma-torsoropen}
Let $\rho:G\to X$ be a group scheme over a scheme $X$, and $\pi:T\to
X$ be a pseudo-torsor under $G$ over $X$. Suppose that there exists
an open subset $V\subset T$ such that the restriction $\pi|_V:V\to
X$ is faithfully flat and the fibers of the morphism $\pi|_V:V\to X$
are dense in the fibers of the morphism $\pi:T\to X$. Then the
morphisms $\rho$ and $\pi$ are faithfully flat.
\end{lemma}
\begin{proof}
The morphism $\rho$ is surjective because of the existence of the
unit section and the morphism $\pi$ is surjective, because the
morphism $\pi|_{V}$ is faithfully flat and, in particular,
surjective. Hence, one needs to show the flatness of $\rho$ and $\pi$.

With this aim, we construct a faithfully flat morphism $\varphi:Y\to
X$ that satisfies the following two conditions. The first condition
is that there is an open subset $W\subset\varphi^*G$ such that the morphism \mbox{$(\varphi^*\rho)|_W:W\to Y$} is flat and the morphism $(\rho^*\varphi)|_W:W\to G$ is surjective. By Lemma~\ref{lemma-critflat}, this implies that $\rho$
is flat. In particular, $\varphi^*\rho$ is flat. The second
condition on $\varphi:Y\to X$ is that there is an isomorphism
$$\varphi^*G\cong\varphi^*T$$ of schemes over $Y$, thus,
$\varphi^*\pi$ is also flat. Again by Lemma~\ref{lemma-critflat},
this implies that $\pi$ is flat, which gives the needed result.

Now let us construct the required morphism $\varphi:Y\to X$. We claim that
$$
Y:=V,\quad\varphi:=\pi|_V,$$
satisfies all conditions above. Indeed, by the hypotheses of the lemma, $\varphi$ is faithfully flat. Further, since $T$ is a pseudo-torsor under $G$, there is an
isomorphism
$$
G\times_X T\stackrel{\sim}\longrightarrow T\times_X T
$$
that commutes with the right projection to $T$. After the
restriction to the open subset $V\subset T$, we obtain an
isomorphism
$$
\psi:G\times_X V\stackrel{\sim}\longrightarrow T\times_X V
$$
of schemes over $V$. In the other notation, $\psi$ is an isomorphism
$\varphi^*G\cong \varphi^*T$ of schemes over $Y$.
Further, consider the open subset $$V\times_X V\subset T\times_X V$$
and put
$$
W:=\psi^{-1}{\left(V\times_X V\right)}\subset G\times_X V=\varphi^*G.
$$
The (right) projection $V\times_X V\to V$ is flat, being the base change of
the flat morphism $\pi|_V:V\to X$ by itself. Since $\psi$ is an isomorphism, the
projection $W\to V$ is also flat, that is, we obtain the
flatness of the morphism $$(\varphi^*\rho)|_W:W\to Y.$$
It remains to prove that the morphism
$$(\rho^*\varphi)|_W:W\to G$$ is surjective. Take a point $g\in G$. We
need to show that the fiber~$W_g$ is non-empty. Let $F$ denote the
residue field at $g$ and put $x:=\rho(g)$ to be the corresponding
$F$-point of $X$. The point $g\in G_x(F)$ defines an automorphism of
the scheme $T_x$ over $F$, which we denote by the same letter~$g$.
By the construction of~$W$, we have the equality
$$
W_g=V_x\cap g^{-1}V_x\subset T_x.
$$
By the hypotheses of the lemma, $V_x$ is a dense open subset in
$T_x$, whence the latter intersection is non-empty.
\end{proof}

Recall that an affine groupoid $\Gamma$ acting on an affine scheme
$X$ over $\kappa$ is a pair $(X,\Gamma)$, where $\Gamma=\Spec(A)$, $X=\Spec(R)$,
and the pair $(R,A)$ is a Hopf algebroid. It follows from the
definition of a Hopf algebroid that one has a morphism $
\pi:\Gamma\to X\times X $ and a morphism of schemes over $X\times
X\times X$
$$
m:p_{12}^*\Gamma\times_{{\left(X^{\times 3}\right)}} p_{23}^*\Gamma\to
p_{13}^*\Gamma.
$$
Moreover, the morphism
$$
(m,{\rm pr}):p_{12}^*\Gamma\times_{{\left(X^{\times 3}\right)}} p_{23}^*\Gamma\to
p_{13}^*\Gamma\times_{{\left(X^{\times 3}\right)}} p_{23}^*\Gamma
$$
is an isomorphism, where
$${\rm pr}:p_{12}^*\Gamma\times_{{\left(X^{\times 3}\right)}} p_{23}^*\Gamma\to
p_{23}^*\Gamma$$
is the projection. Consider the restriction $\Gamma_\Delta$ of~$\Gamma$ to the diagonal $\Delta\subset X\times X$, that is, we have
$\Gamma_{\Delta}=\pi^{-1}(\Delta)$. The morphism
$$
\pi_{\Delta}:\Gamma_{\Delta}\to \Delta\cong X
$$
defines a group scheme over $X$. Take the base change of the
latter morphism by the projection $p_1:X\times X\to X$ and obtain
the group scheme over $X\times X$
$$
\rho:G\to X\times X,
\quad
\rho:=p_1^*(\pi_{\Delta}), \quad
G:=p_1^*(\Gamma_{\Delta}).
$$
It follows that $\pi:\Gamma\to X\times X$ is a
pseudo-torsor under $G$ over $X\times X$. We will need only affine groupoids acting on affine schemes, so, one may suppose this in the following.

\begin{lemma}\label{lemma-ffgroupoid}
Let $\pi:\Gamma\to X\times X$ be a groupoid acting on a scheme $X$.
Suppose that there are open subsets $U\subset X\times X$ and
$V\subset \Gamma$ such that for any $i=1,2$, the fibers of the
projection $p_i|_U:U\to X$ are dense in the fibers of the projection
$p_i:X\times X\to X$, the image $\pi(V)$ is contained in $U$, the
morphism $\pi|_V:V\to U$ is faithfully flat, and the fibers of the
morphism $\pi|_V:V\to U$ are dense in the fibers of the morphism
$\pi_U:\Gamma_U\to U$. Then the morphism $\pi:\Gamma\to X\times X$
is faithfully flat.
\end{lemma}
\begin{proof}
The idea of the proof is to construct a faithfully flat morphism
$\varphi:Y\to X\times X$ such that the base change
$$\varphi^*\pi:\varphi^*\Gamma\to Y$$ is faithfully flat and to
conclude by Lemma~\ref{lemma-critflat}. We are going to define the
morphism $\varphi$ as a composition of two faithfully flat
morphisms.
First, consider the open subset
$$
W:=U\times_X U=(U\times X)\cap (X\times U)\subset X\times X\times X.
$$
Since the open embedding $W\hookrightarrow X\times X\times X$ and
the projection $p_{13}:X^{\times 3}\to X^{\times 2}$ are both flat,
their composition $$p_{13}|_W:W\to X\times X$$ is flat as well. Let us
show that the morphism $p_{13}|_W$ is surjective. Take a point
$z$ on $X\times X$. We need to show that the fiber $W_z$ is
non-empty. Let $F$ denote the residue field at $z$ and put
$x_i:=p_i(z)$ to be the corresponding $F$-points in $X$. By the
construction of $W$, we have the equality
$$
W_z=_{x_1}\!\!U\,\cap U_{x_2}\subset X_k,
$$
where $$_{x_1}U:=p_1^{-1}(x)\cap U,\ \ U_{x_2}:=p_2^{-1}(x)\cap U,\ \
\text{and}\ \ X_F:=X\times\Spec(F).$$ By the hypotheses of the lemma, the open
subsets $_{x_1}U$ and $U_{x_2}$ are dense in~$X_F$, whence their
intersection is non-empty. We conclude that the morphism
$$
p_{13}|_W:W\to X\times X
$$
is surjective, whence it is faithfully flat.
Secondly, consider the morphism
$$
p_{23}^*\pi:p_{23}^*\Gamma\to X\times X\times X
$$
and put
$$
Y:={\left(p_{23}^*\Gamma\right)}_W.$$
Let us show that the morphism ${\left(p_{23}^*\pi\right)}_W:Y\to W$ is faithfully flat. There is a Cartesian square
$$
\begin{CD}
p_{23}^*\Gamma@>>>\Gamma\\
@V p_{23}^*\pi VV@V \pi VV\\
X\times X\times X@>p_{23}>>X\times X.
\end{CD}
$$
Changing $X\times X$ by the open subset $U$, we obtain a Cartesian square
$$
\begin{CD}
Y@>>>\Gamma_U\\
@V {\left(p_{23}^*\pi\right)}_WVV@V \pi_U VV\\
W@>p_{23}|_W>>U.
\end{CD}
$$
Hence, it is enough to show that the morphism $\pi_U:\Gamma_U\to U$ is faithfully
flat. With this aim, consider the group scheme
$$
\rho:G=p_1^*{\left(\Gamma_\Delta\right)}\to X\times X
$$
as in the discussion before the lemma. Take the restrictions
$\Gamma_U=\pi^{-1}(U)$ and $G_U=\rho^{-1}(U)$. Note that the group
scheme $\rho_U:G_U\to U$ over $U$, the pseudo-torsor
$\pi_U:\Gamma_U\to U$ under $G_U$, and the open subset
$V\subset\Gamma_U$ satisfy the hypotheses of
Lemma~\ref{lemma-torsoropen}. Therefore, the morphism $
\pi_U:\Gamma_U\to U$ is faithfully flat, whence the morphism $$
{\left(p_{23}^*\pi\right)}_W:Y\to W $$ is faithfully flat as explained above. Put
$$
\varphi:=p_{13}|_W\circ {\left(p_{23}^*\pi\right)}_W:Y\to X\times X.
$$
The morphism $\varphi$ is faithfully flat, being a composition of faithfully flat morphisms.

Now let us prove that the morphism $\varphi^*\pi:\varphi^*\Gamma\to
Y$ is faithfully flat. For this, we use another equivalent
constructions of the morphism~$\varphi^*\pi$. Consider the diagram
of Cartesian squares
$$
\begin{CD}
p_{23}^*\Gamma\times_{{\left(X^{\times 3}\right)}}p_{13}^*\Gamma@>>>p_{13}^*\Gamma@>>>\Gamma\\
@VVV@V p_{13}^*\pi VV@V\pi VV\\
p_{23}^*\Gamma@>p_{23}^*\pi>>X\times X\times X@>p_{13}>>X\times X.
\end{CD}
$$
This gives the diagram of Cartesian squares
$$
\begin{CD}
{\left(p_{23}^*\Gamma\right)}_W\times_W{\left(p_{13}^*\Gamma\right)}_W@>>>{\left(p_{13}^*\Gamma\right)}_W@>>>\Gamma\\
@VVV@V {\left(p_{13}^*\pi\right)}_W VV@V\pi VV\\
Y={\left(p_{23}^*\Gamma\right)}_W@>{\left(p_{23}^*\pi\right)}_W>>W@>p_{13}|_W>>X\times X.
\end{CD}
$$
Since
$$\varphi^*\Gamma={\left(p_{23}^*\pi\right)}_W^*{\left(p_{13}|_W\right)}^*\Gamma,$$
we obtain that
$$
\varphi^*\Gamma={\left(p_{13}^*\Gamma\right)}_W\times_W{\left(p_{23}^*\Gamma\right)}_W
$$
and the morphism in question $\varphi^*\pi:\varphi^*\Gamma\to Y$
coincides with the projection
$$
{\rm pr}:{\left(p_{13}^*\Gamma\right)}_W\times_W{\left(p_{23}^*\Gamma\right)}_W\to {\left(p_{23}^*\Gamma\right)}_W.
$$
So, we are reduced to show the faithful flatness of the morphism ${\rm pr}$.

Since $(X,\Gamma)$ is a groupoid, there is an isomorphism
$$
p_{12}^*\Gamma\times_{{\left(X^{\times 3}\right)}}
p_{23}^*\Gamma\stackrel{\sim}\longrightarrow
p_{13}^*\Gamma\times_{{\left(X^{\times 3}\right)}} p_{23}^*\Gamma
$$
of schemes over $p_{23}^*\Gamma$ (see the discussion before the lemma). Thus, there is an
isomorphism
$$
{\left(p_{12}^*\Gamma\right)}_W\times_W{\left(p_{23}^*\Gamma\right)}_W\stackrel{\sim}\longrightarrow
{\left(p_{13}^*\Gamma\right)}_W\times_W{\left(p_{23}^*\Gamma\right)}_W
$$
of schemes over ${\big(p_{23}^*\Gamma\big)}_W$. This shows that faithful flatness of the morphism $\rm pr$ is equivalent to the faithful flatness of the projection
$$
{\rm pr}':{\left(p_{12}^*\Gamma\right)}_W\times_W{\left(p_{23}^*\Gamma\right)}_W\to {\left(p_{23}^*\Gamma\right)}_W.$$
Finally, the morphism ${\rm pr}'$ is the base change of the faithfully flat morphism
$\pi_U:\Gamma_U\to U$ by the composition
$$
\begin{CD}
{\left(p_{23}^*\Gamma\right)}_W@>{\left(p^*_{23}\pi\right)}_W>>
W@>p_{12}|_W>> U.
\end{CD}
$$
Therefore, the morphism ${\rm pr}'$ is faithfully flat, which finishes the proof.
\end{proof}

\begin{lemma}\label{lemma-twotorsors}
Let $\psi:G'\to G$ be a morphism between group schemes of finite
type over a scheme $X$, let $\pi:T\to X$ be a pseudo-torsor under
$G$, $\pi':T'\to X$ be a pseudo-torsor under $G'$,  and let
$\varphi:T'\to T$ be a morphism compatible with $\psi$ in the
following sense: the diagram
$$
\begin{CD}
G'\times_X T'@>>>T'\\
@V \psi\times\varphi VV@V\varphi VV\\
G\times_X T@>>>T
\end{CD}
$$
commutes. Let $V\subset T$ be an open subset and put $V':=\varphi^{-1}(V)$. Suppose that the fibers of the morphism~$\pi|_V:V\to X$ are dense in the fibers of the morphism $\pi:T\to X$ and the morphism~$\varphi|_{V'}:V'\to V$ is surjective. Then the fibers of the morphism $$\pi'|_{V'}:V'\to X$$ are dense in the fibers of the morphism~$\pi':T'\to X$.
\end{lemma}
\begin{proof}
First, we reduce the lemma to a question about algebraic groups.
Since the needed result is fiber-wise and all data in the lemma are
stable under a base change, we may assume that $X=\Spec(F)$,
where~$F$ is a field. Further, it is enough to show the density
after the extension of scalars to the algebraic closure of~$F$.
Thus, we assume that $F$ is algebraically closed. Taking an
$F$-point $t'$ on $T'$ and the point $t:=\varphi(t')$ on~$T$, we
obtain isomorphisms $G'\stackrel{\sim}\longrightarrow T'$ and
$G\stackrel{\sim}\longrightarrow T$ that send $\psi$ to $\varphi$.

Therefore, we may assume that $T'=G'$ and $T=G$. Finally, we may assume that the schemes $G$ and~$G'$ are reduced. Summarizing, we have a morphism of algebraic groups $\psi:G'\to G$ and an open dense subset $V\subset G$ such that the morphism $\psi|_{V'}:V'\to V$ is surjective, where $V'=\psi^{-1}(V)$. We need to show that $V'$ is dense in $G'$.

The image of the morphism $\psi$ is a closed subgroup in $G$ (for
example, see \cite[Proposition~2.2.5]{Springer}). On the other hand,
this image contains the dense subset $V$, because the morphism
$$\psi|_{V'}:V'\to V$$ is surjective. Consequently, the morphism
$\psi$ is surjective. It follows that all irreducible components of
the fibers of $\psi$ have the same dimension $d:=\dim(G')-\dim(G)$.

Since $V\subset G$ is a dense open subset and all irreducible components of $G$ have the same dimension~$\dim(G)$, we see that all irreducible components of the closed subset $Z:=G\backslash V\subset G$ have dimension strictly less than $\dim(G)$. Therefore, all irreducible components of the closed subset $\psi^{-1}(Z)\subset G'$ have dimension strictly less than $d+\dim(G)=\dim(G')$. Since $$V'=G'\backslash \psi^{-1}(Z)$$ and all irreducible components of $G'$ have the same dimension $\dim(G')$, we conclude that $V'$ is dense in~$G'$, which finishes the proof.
\end{proof}

\subsection{Proof of Proposition~\ref{prop-localizealgebroid}}\label{sec:proofof72}
We are now ready to give a proof of Proposition~\ref{prop-localizealgebroid}. We use the geometric notation from Section~\ref{subsection-geomresults}.

\begin{proof}[Proof of Proposition~\ref{prop-localizealgebroid}]
We will localize the ring $R$ over a finite set of non-zero elements and then prove that the corresponding localization of $A$ is faithfully flat over the obtained localization of~$R\otimes_k R$.

Let $\{a_i\}$ be a finite set of $D_k$-generators of $A$ over
$R\otimes_k R$ and put $A_0$ to be the $(R\otimes_k R)$-subalgebra
in $A$ generated by the set~$\{a_i\}$. Since $L:=\Frac(R)$ is a
field, by~\cite[3.7, 3.8]{DeligneFS} (see also \cite[\S3.3]{Water}),
the images of $a_i$'s in ${_LA_L}$ are contained in a Hopf
subalgebroid of $\big(L,{_LA_L}\big)$ finitely generated over~$L\otimes_k
L$. Therefore, localizing~$R$ by a non-zero element and enlarging
the finite subset $\{a_i\}\subset A$, we obtain that~$(R,A_0)$ is a
Hopf subalgebroid in $(R,A)$.

 For each natural $n$, put $A_n$ to be
the $(R\otimes_k R)$-subalgebra in $A$ generated by all elements of
the form $$(\partial_1\cdot\ldots\cdot\partial_m)(a_i),\quad
\partial_j\in D_k,\ m\Le n.$$ Since $(R,A)$ is a differential
Hopf algebroid, it follows that $(R,A_n)$ is a Hopf subalgebroid
in~$(R,A)$ for all $n$. Put
$$
X:=\Spec(R),\quad \Gamma:=\Spec(A),\quad \Gamma_n:=\Spec(A_n).
$$
Denote the groupoid morphisms by $\pi_n:\Gamma_n\to X\times X$.

Since $A$ is $D_k$-finitely generated over $k$, we see that $\Gamma$ has finitely many irreducible components~\cite[Theorem~7.5]{Kap}. Applying Lemma~\ref{lemma-Trushin} to each irreducible component of $\Gamma$, we see that there exist a natural number $N$ and an affine dense open subset $W_N\subset \Gamma_N$ such that for any $n\ge N$, the morphisms $$\varphi_n|_{W_n}:W_n\to W_N$$ are faithfully flat, where $W_n:=\varphi_n^{-1}(W_N)$ and $\varphi_n:\Gamma_n\to \Gamma_N$ are the morphisms that arise in the projective system formed by $\Gamma_n$.

Since $\Char k=0$ and $R$ is a domain, the ring $R\otimes_k R$ is reduced. Since the morphism $$\pi_N|_{W_N}:W_N\to X\times_k X$$ is of finite type, by the generic flatness (for example, see~\cite[Proposition~7.91.7]{deJong}), there is a dense open subset $U\subset X\times_k X$ such that the morphism $\pi_N|_{V_N}:V_N\to U$ is flat and of finite presentation, where $$V_N:=W_N\cap\pi_N^{-1}(U).$$ As $A_F\ne 0$, we may also assume that $\pi_N|_{V_N}$ is faithfully flat. It follows that the morphisms
$$\pi_n|_{V_n}:V_n\to U$$ are faithfully flat, where $$V_n:=\varphi_n^{-1}(V_N), \quad n\Ge N.$$
By~\cite[9.5.3]{EGAIV3}, replacing $U$ with a dense open subset, we
obtain that the fibers of the morphism $$\pi_N|_{V_N}:V_N\to U$$ are
dense in the fibers of the morphism $$(\pi_N)_U:(\Gamma_N)_U\to U,$$
because $W_N$ is dense in~$\Gamma_n$. Since $R\otimes_k R$ has
finitely many irreducible components~\cite[Theorem~7.5]{Kap}, we may
assume that $U$ is an affine dense open subset in $X\times_k X$.
Localizing $R$ by a non-zero element, we obtain that, for any
$i=1,2$, the fibers of the projections $p_i|_{U}:U\to X$ are dense
in the fibers of the projection $$p_i:X\times_k X\to X$$ (by the
extension of scalars, this follows from the analogous statement
about irreducible varieties over fields).
For each $n$, put $$G_n:=p_1^*((\Gamma_n)_{\Delta}),$$ where
$\Delta\subset X\times_k X$ is the diagonal. Then $\Gamma_n$ is a
pseudo-torsor under the group scheme~$G_n$ over $X\times_k X$. The
morphism of group schemes $\psi_n:G_n\to G_N$ induced by $\varphi_n$
is compatible with the morphism of pseudo-torsors
$\varphi_n:\Gamma_n\to \Gamma_N$ in the sense of
Lemma~\ref{lemma-twotorsors}. Since the fibers of the morphism
$$\pi_N|_{V_N}:V_N\to U$$ are dense in the fibers of the morphism
$(\pi_N)_U:(\Gamma_N)_U\to U$, we see that, by
Lemma~\ref{lemma-twotorsors}, the fibers of the morphism
$\pi_n|_{V_N}:V_n\to U$ are dense in the fibers of the morphism
$$(\pi_n)_U:(\Gamma_n)_U\to U.$$
We obtain that, for every $n\ge N$, the groupoid $\Gamma_n\to X\times_k
X$ and the open subsets $V_n\subset \Gamma$, $U\subset X\times_k X$
satisfy all hypotheses of Lemma~\ref{lemma-ffgroupoid} (which is also true
for schemes over a field $k$ with the product of schemes taken over $k$). Therefore,
the morphism~$\pi_n$ is faithfully flat. In other terms, the
ring~$A_n$ is faithfully flat over $R\otimes_k R$. Since $A=\bigcup_n A_n$,
where $A_{n}\subset A_{n+1}$, we conclude that $A$ is faithfully
flat over $R\otimes_k R$, which finishes the proof.
\end{proof}

\section{Proofs of the main results}\label{sec:proofsofthemainresults}

\subsection{Proof of Theorem~\ref{theor-main}}\label{subsection-prooftheor-main}

We use the notation from Theorem~\ref{theor-main}. Let $M$ be a
finite-dimensional $D_{K/k}$-module over $K$. Consider a
$D_k$-structure on $\DMod(K,D_{K/k})$ as in
Theorem~\ref{teor-paramAtiayh}. Let $\Cat$ be the subcategory
$\langle M\rangle_{\otimes,D}$ in~$\DMod{\left(K,D_{K/k}\right)}$
$D_k$-tensor generated by $M$ (Definition~\ref{defin-diffgenercat}).
By Theorem~\ref{theor-PPVdff}, to prove the theorem, it is enough to
construct a differential functor from $\Cat$ to $\Vect(k)$, which is
our goal in what follows.

First, we would like to apply Theorem~\ref{theor-diffTanncorr} to
the forgetful functor $\Cat\to\Vect(K)$ and, thus, obtain a
$D_k$-Hopf algebroid. The problem here is that, a priori, there is no
$D_k$-structure on~$K$. To overcome this, the splitting
$\widetilde{D}_k$ is introduced in the hypotheses of the theorem (see also Remark~\ref{rem-allow}).
This allows to switch between the $D_k$-structure on $\Cat$ and the
$D_K$-structure on $K$ as follows.

The morphism of differential fields $(k,D_k)\to\big(k,\widetilde{D}_k\big)$ defines a
$\widetilde{D}_k$-structure on $\Cat$ by
Proposition~\ref{prop-diffextscal}~\eqref{en:2368}.
Denote  the category $\Cat$ with this
$\widetilde{D}_k$-structure by $\widetilde{\Cat}$. Thus, the identity functor
$
\Cat\to\widetilde{\Cat}
$
is a differential functor from a $D_k$-category $\Cat$ to a
$\widetilde{D}_k$-category $\widetilde{\Cat}$. By Theorem~\ref{teor-paramAtiayh}, the forgetful functor is a
differential functor from the $D_k$-category $\Cat$ to the
$D_K$-category $\Vect(K)$. Apply the extension of scalars along the vertical morphisms of the diagram from Remark~\ref{rem-split} to the forgetful functor $\Cat\to\Vect(K)$. 

By
Proposition~\ref{prop-diffextscal}\eqref{en:2375}, we obtain a differential functor
$\omega$ from the $\widetilde{D}_k$-category $\widetilde{\Cat}$ to the $\widetilde{D}_K$-category $\Vect(K)$ (the latter category is with the usual $\widetilde{D}_K$-structure as in Example~\ref{examp-moddiffcat}). Recall that $$\widetilde{D}_K=K\otimes_k\widetilde{D}_k$$ and $K$ is a $\widetilde{D}_k$-field over $k$. Thus, we have a differential functor $\omega:\widetilde{\Cat}\to\Vect(K)$ between $\widetilde{D}_k$-categories.

By Theorem~\ref{theor-diffTanncorr}, there exists a
$\widetilde{D}_k$-Hopf algebroid $(K,H)$ over $k$ such that $H$ is
faithfully flat over~$K\otimes_k K$ and~$\omega$ lifts up to an
equivalence of \mbox{$\widetilde{D}_k$-categories}
$$
\widetilde{\Cat}\stackrel{\sim}\longrightarrow\Comodf(K,H).
$$
Since~$\Cat$ is $D_k$-tensor generated by one object, the
$\widetilde{D}_k$-category $\Cat$ is also $\widetilde{D}_k$-tensor
generated by one object. Hence, Proposition~\ref{prop-difffingencat} and the proof of Theorem~\ref{theor-diffTanncorr} imply that $H$ is $\widetilde{D}_k$-finitely generated
over~$K\otimes_k K$. We apply Theorem~\ref{theor-defindiffHopfalg} to~$(K,H)$ and obtain
the corresponding Hopf algebroid $(R,A)$. The extension of scalars
$$
K\otimes_R-:\Comodf(R,A)\to\Comodf(K,H)
$$
is a differential functor between $\widetilde{D}_k$-categories
(Example~\ref{examp-difffunctor}\eqref{en:2377}). The forgetful
functor
$$
\Comodf(R,A)\to\Mod(R)
$$
is a differential functor, where we consider the $\widetilde{D}_R$-category structure on $\Mod(R)$ with $$\widetilde{D}_R:=R\otimes_k\widetilde{D}_k$$ (Example~\ref{examp-difffunctor}\eqref{en:2370}).
We have that $R$ is $\widetilde{D}_k$-finitely generated over $k$, the morphism $\big(R,\widetilde{D}_R\big)\to \big(K,\widetilde{D}_K\big)$ is strict, where $\widetilde{D}_K:=K\otimes_k\widetilde{D}_k$, and~$(k,D_k)$ is relatively differentially closed in $\big(K,\widetilde{D}_K\big)$ by the hypotheses of the theorem. Therefore, there is a morphism of differential rings $\big(R,\widetilde{D}_R\big)\to (k,D_k)$.
This defines a differential functor
$$
\Mod(R)\to\Vect(k),\quad N\mapsto k\otimes_R N
$$
from the $\widetilde{D}_R$-category $\Mod(R)$ to the $D_k$-category
$\Vect(k)$ (Example~\ref{examp-difffunctor}\eqref{en:2371}).
Summarizing, we obtain a collection of differential functors
$$
\begin{CD}
\Cat\to\widetilde{\Cat}\to\Comodf(K,H)@<K\otimes_R-<<\Comodf(R,A)\to
\Mod(R)\to\Vect(k).
\end{CD}
$$
Since $A$ is faithfully flat over $R\otimes_k R$, the extension of
scalars functor $K\otimes_R-$ is an equivalence of categories
(see~\cite[1.8,3.5]{DeligneFS} and also
Section~\ref{subsection-prelHopfalg}). All together, this defines a
differential functor from $\Cat$ to~$\Vect(k)$, which finishes the
proof.

\subsection{Proof of Theorem~\ref{thm-main}}\label{subsection-proofthm-main}

We need the following simple facts.

\begin{lemma}\label{lem-ratpointparam}
Let $Y$ be an irreducible variety over a field $k_0$ with $\Char k_0
= 0$, $k$ be a field extension of $k_0$, and let $K_0:=k_0(Y)$.
Suppose that $k_0$ is existentially closed in $K_0$. Then, for any
non-empty open subset $ U\subset X:=Y\times_{k_0} k, $ there exists
a $k_0$-point $y$ on $Y$ such that the $k$-point  $ x:=y\times_{k_0}
k $ of $X$ belongs to $U$.
\end{lemma}
\begin{proof}
First note that, if the lemma is proven for an extension $k'$ of $k$,
then this implies the lemma for $k$. Thus, replacing~$k$ by its
extension, if needed, we may assume that $k^{\Gamma}=k_0$, where
$\Gamma$ is the group of all field automorphisms of $k$ over $k_0$,
because $\Char k_0 =0$. For a non-empty open subset $U\subset X$,
take its complement $Z:=X\backslash U$ and consider the intersection
$$
Z':=\bigcap _{\sigma\in \Gamma}\sigma(Z).
$$
The closed subvariety $Z'\subset X$ is invariant under $\Gamma$,
therefore there exists a closed subvariety $W\subset Y$ such that
$Z'=W\times_{k_0} k$. Moreover, $W\neq Y$, because $Z'\subset Z\neq
X$. Put $V:=Y\backslash W$, which is a non-empty open subset in $Y$.
Since $k_0$ is existentially closed in $K_0$, there exists $y\in
V(k_0)$. This defines a $k$-point $x:=y\times_{k_0} k$ in $X$. If $x
\in Z(k)$, then $x\in\sigma(Z(k))$ for any $\sigma\in \Gamma$. Thus,
$x \in Z'(k)$, which contradicts to the fact that $y$ is not in $W$.
Hence, we obtain that $x$ belongs to $U$.
\end{proof}

\begin{lemma}\label{lem-linadd}
Let $k\subset K$ be a field extension and let $\varphi_1,\ldots,\varphi_n:k\to k$
be maps that are linearly independent over $k$. Then $\varphi_1,\ldots,\varphi_n$ are linearly independent over $K$ considered as maps from $k$ to $K$.
\end{lemma}
\begin{proof}
Since $\varphi_1,\ldots,\varphi_n$ are linearly independent over $k$, the image of the map
$$
\Phi:k\to k^{\oplus n},\quad f\mapsto (\varphi_1(f),\ldots,\varphi_n(f)),
$$
spans all $k^{\oplus n}$ over $k$. Therefore, the image of the composition of $\Phi$ and the natural embedding $k^{\oplus n}\subset K^{\oplus n}$ spans all $K^{\oplus n}$ over~$K$, so, $\varphi_1,\ldots,\varphi_n$ are linearly independent over $K$.
\end{proof}

\begin{lemma}\label{lemma-finitegener}
Let $K$ be a $D_k$-field over a differential field $(k,D_k)$ with
$\Char k=0$ such that $K$ is of finite transcendence degree over
$k$. Then any finite subset $\Sigma\subset K$ is contained in a
$D_k$-subalgebra $R$ in $K$ over~$k$ that is finitely generated as
an algebra over $k$.
\end{lemma}
\begin{proof}
Let $L$ be the $D_k$-subfield generated by $\Sigma$ in $K$. It follows from \cite[Theorem~5.6.3]{Pan} that $L$ is a finitely generated field over $k$. Hence, there exists a finite set $S \subset L$ such that $\Sigma \subset S$ and $L = k(S)$. It now follows from differentiating fractions that $R := k[S\cup 1/T]\subset K$ satisfies the requirement of the lemma, where $T\subset K$ is the set of the denominators of $D_k(S)$.
\end{proof}

Now, we prove Theorem~\ref{thm-main} using its notation. Suppose that condition~\ref{it:552} of the theorem holds. Then the structure map identifies $\widetilde{D}_k$ and $1\otimes D_k$, where
$\widetilde{D}_k$ is given in condition~\ref{it:552}. Hence, $$(k,D_k)\to \big(k,\widetilde{D}_k\big)$$ is an isomorphism and $K$ is a $D_k$-field. Let $R$ be a $D_k$-finitely
generated subalgebra in~$K$ over $k$. We need to show that there is a morphism of $D_k$-algebras $R\to k$. By Remark~\ref{rem-mainthm}\eqref{rmk:333}, we have
$$K=\Frac{\left(K_0\otimes_{k_0} k\right)}.$$ Let $\{a_i/b_i\}$ be a finite set of $D_k$-generators of $R$ over $k$ with $a_i,b_i\in K_0\otimes_{k_0}k$.
 Let $R_0$ be the subalgebra in $K_0$ generated over $k$ by the $K_0$-components of summands in $a_i$'s and $b_i$'s and put $$f:=\prod _i b_i.$$ Since~$K_0$ is the field of $D_k$-constants, $R_0\otimes_{k_0}k$ is a $D_k$-differential subalgebra in $K$ over $k$. Hence, $R$ is contained in the localization ${\left(R_0\otimes_{k_0} k\right)}_f$. By Lemma~\ref{lem-ratpointparam} applied to
$$
Y:=\Spec(R_0),\quad U:=\Spec{\left({\left(R_0\otimes_{k_0} k\right)}_f\right)},
$$
there exists a $k_0$-point $y$ on $Y(k_0)$ such that the $k$-point $x:=y\times_{k_0}k$ of $X$ belongs
to $U$. The point $x$ defines a morphism of $k$-algebras $f:R\to k$.
The kernel of $f$ is generated by
$D_k$-constants in $R$, because $K_0=K^{D_k}$. Therefore, $f$ is a morphism of
$D_k$-algebras, and $(k,D_k)$ is relatively differentially closed in $$\left(K,K\otimes_k D_k\right)\cong\big(K,K\otimes_k \widetilde{D}_k\big).$$

Now suppose that condition~\ref{it:551} of the theorem holds. Our
first goal is to construct a splitting~$\widetilde{D}_k$
of~$(K,D_K)$ over $(k,D_k)$ such that the natural map $K\otimes_k
\widetilde{D}_k\to D_K$ is an isomorphism. With this aim, we
consider the ``effective'' quotients
$$
D_K^{\rm eff}:={\rm Im}(\theta_K:D_K\to\Der(K,K)),\quad D_k^{\rm eff}:={\rm
Im}(\theta_k:D_k\to\Der(k,k))
$$
of the differential structures $D_K$ and $D_k$, respectively. It follows from Lemma~\ref{lem-linadd} that the natural map $$K\otimes_kD^{\rm eff}_k\to\Der(k,K) $$ is injective. Therefore, the composition
$$
D_K\to K\otimes_k D_k\to K\otimes_k D^{\rm eff}_k
$$
factors through $D^{\rm eff}_K$, that is, we have a commutative diagram
$$
\begin{CD}
D_K @>>> K\otimes_kD_k\\
@VVV @VVV\\
D_K^{\rm eff}@>>> K\otimes_kD_k^{\rm eff}.
\end{CD}
$$
Hence, $\big(K,D^{\rm eff}_K\big)$ is a parameterized differential field over $\big(k,D_k^{\rm eff}\big)$. By condition~\ref{it:551} of the theorem, we have 
$$
D_{K/k}\stackrel{\sim}\longrightarrow \Der_k(K,K).
$$
Consequently, the natural $K$-linear morphism
$$
D_{K/k}=\Ker{\left(D_K\to K\otimes_k D_k\right)}\to\Ker{\Big(D_K^{\rm eff}\to K\otimes_k D^{\rm eff}_k\Big)\subseteq \Der_k(K,K)}
$$
is an isomorphism. It follows that there is an isomorphism
\begin{equation}\label{eq:5352}
D_K\cong D_K^{\rm eff}\times_{{\left(K\otimes D^{\rm eff}_k\right)}}(K\otimes_k D_k).
\end{equation}
Take commuting
bases in~$D^{\rm eff}_K$ and~$D^{\rm eff}_k$ from Proposition~\ref{prop-commbasis}. Put $\widetilde{D}^{\rm eff}_k$ to be the $k$-linear span of the basis in $D^{\rm eff}_K$. Then $\widetilde{D}^{\rm eff}_k$ is a splitting of $\big(K,D^{\rm eff}_K\big)$ over $\big(k,D^{\rm eff}_k\big)$ such that $$D_K^{\rm eff}\cong K\otimes_k\widetilde{D}^{\rm eff}_k.$$ Put
\begin{equation}\label{eq:5360}
\widetilde{D}_k:=\widetilde{D}^{\rm eff}_k\times_{D^{\rm eff}_k} D_k\subset D_K.
\end{equation}
Since taking effective quotients is a morphism of Lie rings and by formula~\eqref{eq-explicintegrmorph} from Section~\ref{subsection-diffrings}, we have that $\widetilde{D}_k$ is closed under the Lie bracket on $D_K$. Thus, $\widetilde{D}_k$ is a splitting of $(K,D_K)$ over $(k,D_k)$. Comparing~\eqref{eq:5352} and~\eqref{eq:5360}, we see that $$K\otimes_k\widetilde{D}_k\cong D_K.$$ Note that, in this case, $$\dim_k{\big(\widetilde{D}_k\big)}=\dim_K(D_K),$$ while, in the previous case (condition~\ref{it:552} of the theorem), $\dim_k{\big(\widetilde{D}_k\big)}$ could be less than $\dim_K(D_K)$. So, in this case, $\widetilde{D}_k$ could be ``much larger''. Put
$$
D:=\Ker\Big(\widetilde{D}_k\to D_k\Big).
$$
Then we have $K\otimes_k D\cong D_{K/k}$.

Let $(R,D_R)$ be a differential subalgebra in $(K,D_K)$ over $(k,D_k)$ such that the morphism $(R,D_R)\to (K,D_K)$ is strict
and $(R,D_R)$ is differentially finitely generated over $k$. Extending $R$ by a finite number
of elements from $K$, we obtain that $$D_R\cong R\otimes_k\widetilde{D}_k$$ and~$R$ is a $\widetilde{D}_k$-finitely generated $\widetilde{D}_k$-subalgebra in $K$ over $k$.
Since $$\dim_K(D_{K/k})=\dim_k(\Der_k(K,K))$$ is finite,~$K$ is of finite transcendence degree over $k$. Hence, by Lemma~\ref{lemma-finitegener}, we may assume that $R$ is finitely generated as an algebra over~$k$. By the hypotheses of the theorem, we have
$$D_{K/k}\cong \Der_k(K,K).$$ Since $\Char k=0$ and $R$ is finitely generated,
localizing~$R$ by a non-zero element, we may assume that $R$ is
smooth over $k$ and
$$
R\otimes_k
D\cong\Der_k(R,R).
$$
Since $k$ is existentially closed in $K$, there is a homomorphism
$f:R\to k$ of $k$-algebras. We claim that~$f$ extends to a morphism
of differential algebras
$\left(R,D_R\right)\to (k,D_k)$ over $(k,D_k)$.
By definition, to prove this, we have to construct a morphism of Lie rings $$s:D_k\to\widetilde{D}_k =k\otimes_R D_R$$ such that, for all
$\partial\in D_k$ and $a\in R$, we have
\begin{equation}\label{eq-defs}
\partial(f(a))=f(s(\partial)(a)).
\end{equation}
We claim that, for any
$\partial\in D_k$, there is a unique $s(\partial)\in
\widetilde{D}_k$ that satisfies~\eqref{eq-defs}. Indeed, consider the derivation $\delta$ from $R$ to itself defined as the composition
$$
\begin{CD}
R@>{f}>>k@>{\partial}>>k@>>> R
\end{CD}
$$
and consider any $\tilde\partial\in \widetilde{D}_k$ such that
$\tilde\partial$ is sent to $\partial$ by the surjective map $\widetilde{D}_k\to D_k$. The difference $\delta-\theta_K{\big(\tilde\partial\big)}$ is a $k$-linear derivation from $R$ to itself, that is, it belongs to $$R\otimes_k D\cong\Der_k(R,R).$$ Put
$$
s(\partial):=\tilde\partial+f\Big(\delta-\theta_K{\big(\tilde\partial\big)}\Big),
$$
where~$f$ denotes also the map
$$
\begin{CD}
R\otimes_k\widetilde{D}_k@>f\cdot\id>>
\widetilde{D}_k.
\end{CD}
$$
By construction, $s(\partial)$ satisfies~\eqref{eq-defs}. The
uniqueness of~$s(\partial)$ follows from the fact that if
$s(\partial)$ and $s(\partial)'$ in~$\widetilde{D}_k$
satisfy~\eqref{eq-defs}, then  $s(\partial)-s(\partial)'$ belongs to
$$\Ker\Big(R\otimes_k\widetilde{D}_k\stackrel{f\cdot\id}\longrightarrow
\widetilde{D}_k\Big),$$ whose intersection with
$1\otimes\widetilde{D}_k$ is trivial. By construction, the obtained
map $s:D_k\to \widetilde{D}_k$ is $k$-linear. The uniqueness
of~$s(\partial)$ implies that $s$ is a morphism of Lie rings: given
$\partial_1,\partial_2\in D_k$, the commutator
$\big[s(\partial_1),s(\partial_2)\big]$ satisfies~\eqref{eq-defs}
with $\partial=[\partial_1,\partial_2]$. This shows that $(k,D_k)$
is relatively differentially closed in
$$(K,D_K)=\big(K,K\otimes_k\widetilde{D}_k\big).$$

\section{PPV extensions with non-closed constants}\label{section-nonclosedconstants}

 In this
section, we discuss two aspects of PPV extension for parameterized
differential fields over an arbitrary differential field $(k,D_k)$
(in contrast to the usual assumption~\cite{PhyllisMichael} that
$(k,D_k)$ be differentially closed).

\subsection{Galois correspondence}\label{section-Galois}

We establish the Galois correspondence for PPV extensions.
Basically, we use the classical differential Galois correspondence
for PV extensions. Also, we use the differential Tannakian
formalism, in particular, Theorem~\ref{theor-PPVdff}.

First, let us recall several notions concerning differential
algebraic groups. Let $(k,D_k)$ be a differential field and $G$ be a
{\it linear $D_k$-group}, that is, $G$ is a group-valued functor on
$\DAlg(k)$ corepresented by a $D_k$-finitely generated $D_k$-Hopf
algebra $U$ over $k$. A~{\it $D_k$-subgroup $H$ in $G$} is a
corepresentable group subfunctor $H$ in $G$ on the category
$\DAlg(k)$. By \cite[Theorem~15.3]{Water}, this corresponds to a
surjective morphism $U\to V$ between $D_k$-Hopf algebras over $k$.
Hence, $H$ is a linear $D_k$-group. 

Suppose that $G$ acts on a
$D_k$-algebra $A$, that is, we have a morphism of $D_k$-algebras
$m:A\to A\otimes_k U$ that satisfies the axioms of a comodule over a
Hopf algebra. Let $A$ be a domain and $L:=\Frac(A)$. We put
$$
L^G:=\left\{a/b\in L\:|\: a,b\in A,\,b\cdot m(a)=a\cdot m(b)\right\}.
$$
It follows that $L^G$ is a $D_k$-subfield in $L$.

Let $(K,D_K)$ be a parameterized differential field over $(k,D_k)$
and let $M$ be a finite-dimensional $D_{K/k}$-module. Suppose that
there exists a PPV extension $L$ for $M$.

\begin{definition}\label{defin-Galois}
The {\it parameterized differential Galois group of $L$ over $K$} is the group functor
$$
\Gal^{D_K}(L/K):\DAlg(k,D_k)\to \Sets,\quad R\mapsto
\Aut^{D_K}(R\otimes_k A/R\otimes_k K),
$$
where $A$ is a PPV ring associated with $L$ (Definition~\ref{defin-PPVring}) and we consider a $D_K$-structure on the extension of scalars
$R\otimes_k K$ as given by Definition~\ref{defin-extscaldiffmod}.
\end{definition}

\begin{lemma}\label{lemma-Galois}
The functor $\Gal^{D_K}(L/K)$ is corepresented by a $D_k$-finitely
generated $D_k$-Hopf algebra, that is, $\Gal^{D_K}(L/K)$ is a linear
$D_k$-group.
\end{lemma}
\begin{proof}
By Theorem~\ref{theor-PPVdff}, 
the PPV extension $L$ corresponds to
a differential functor
$$
\omega:\langle M\rangle_{\otimes,D}\to\Vect(k).$$
By Remark~\ref{rem-PPVring}, the functor $\Gal^{D_K}(L/K)$ 
is canonically isomorphic to the functor
$$
\DAlg(k,D_k)\to \Sets,\quad R\mapsto
\fIsom^{\otimes,D}(\omega_R,\omega_R).
$$
By Proposition~\ref{prop-reprdiffalg}, the latter functor is
corepresentable by a $D_k$-Hopf algebra $U$ over $k$. By
Proposition~\ref{prop-difffingencat}, $U$ is $D_k$-finitely
generated.
\end{proof}

Note that one can also prove Lemma~\ref{lemma-Galois} more
explicitly without using the Tannakian formalism.

\begin{remark}\label{rem-classGalois}
It follows from Proposition~\ref{prop-reprdiffalg},
Theorem~\ref{theor-PPVdff}, and \cite[9.6]{DeligneFS} that
$L$ is a union of (possibly, infinitely many) PV-extensions defined
by the $D_{K/k}$-modules $\big(\At^1\big)^{\circ i}(M)$ and
$\Gal^{D_K}(L/K)$ with forgotten $D_k$-structure is the differential
Galois group $\Gal^{D_{K/k}}(L/K)$.
\end{remark}

Recall the differential Galois correspondence in the case of
arbitrary constants from~\cite[Section 4]{Tobias2008}. Given a Hopf
algebra $U$, an {\it algebraic subgroup} $\Spec(V)$ in $\Spec(U)$ 
corresponds to a surjective homomorphism between Hopf algebras~$U\to V$.

\begin{proposition}\label{prop-Galois}
There is a bijective correspondence between algebraic subgroups
\mbox{$H\subset \Gal^{D_{K/k}}(L/K)$} and $D_{K/k}$-subfields
$K\subset E\subset L$ given by
$$
H\mapsto E:=L^H,\quad E\mapsto H:=\Gal^{D_{E/k}}(L/E).
$$
\end{proposition}

The Galois correspondence in the parameterized case is as follows.

\begin{proposition}\label{prop-Galoisparam}
There is a bijective correspondence between $D_k$-subgroups
$H\subset \Gal^{D_K}(L/K)$ and $D_K$-subfields $K\subset E\subset L$
given by
$$
H\mapsto E:=L^H,\quad E\mapsto H:=\Gal^{D_E}(L/E).
$$
\end{proposition}
\begin{proof}
By Proposition~\ref{prop-Galois} and Remark~\ref{rem-classGalois},
we only need to show that an algebraic subgroup
$$H\subset\Gal^{D_{K/k}}(L/K)$$ is a $D_k$-subgroup in
$\Gal^{D_{K}}(L/K)$ if and only if the corresponding
$D_{K/k}$-subfield $E\subset L$ is a $D_K$-subfield. Suppose that
$E$ is a $D_K$-subfield. Then $L$ is a PPV extension for the
$D_E$-module~$M_E$ over $E$, where $D_E:=E\otimes_KD_K$. Therefore,
the corresponding Galois group $H$ has a canonical $D_k$-structure
and corepresents a group subfunctor in~$G$ on $\DAlg(k)$ given by
Definition~\ref{defin-Galois}. Thus, $H$ is a $D_k$-subgroup in $G$.

Conversely, suppose that $H$ is a $D_k$-subgroup in
$G:=\Gal^{D_K}(L/K)$. Consider the extension of scalars $G_K$ from
$(k,D_k)$ to~$(K,D_K)$ for $G$
(Definition~\ref{defin-extscaldiffmod}). We have a $D_K$-subgroup
$H_K$ in $G_K$. By the adjunction between restriction and extension
of scalars (Definition~\ref{defin-extrestr}) and Definition~\ref{defin-Galois}, $G_K$ acts on the
$D_K$-field $L$ over $K$ and $L^H=L^{H_K}$. By
Proposition~\ref{prop-Galois}, we have $E=L^H$, whence, $E$ is an
$D_K$-subfield in $L$.
\end{proof}

The proof of the normal subgroup case uses the differential Tannakian formalism.

\begin{proposition}\label{prop-normal}
Under the correspondence from Proposition~\ref{prop-Galoisparam}, a
normal $D_k$-subgroup $H$ corresponds to a PPV extension~$E$ over
$K$, and we have an isomorphism of $D_k$-groups
$$
\Gal^{D_K}(L/K)/H\cong \Gal^{D_K}(E/K).
$$
\end{proposition}
\begin{proof}
For short, put $G:=\Gal^{D_K}(L/K)$. Let $$\omega:\langle
M\rangle_{\otimes,D}\to \Vect(k)$$ be the differential functor that
corresponds to the PPV extension $L$ by Theorem~\ref{theor-PPVdff}.
It follows from Theorem~\ref{theor-diffTanncorr} and the proof of
Lemma~\ref{lemma-Galois} that $\omega$ lifts up to an equivalence of
$D_k$-categories
\begin{equation}\label{eq:5194}
\langle M\rangle_{\otimes,D}\stackrel{\sim}\longrightarrow\Repf(G).
\end{equation}
Let $H$ be a normal $D_k$-subgroup in $G$. Then $\Repf(G/H)$ is a
full $D_k$-subcategory in $\Repf(G)$.
By~\cite[Proposition~15]{Cassidy}, $G/H$ is a linear $D_k$-group,
that is, there is a faithful finite-dimensional representation
of~$G/H$ over $k$. Let $N$ be the corresponding $D_{K/k}$-module
over $K$ under the equivalence~\eqref{eq:5194}. Taking the
restriction of $\omega$ to the subcategory $\langle
N\rangle_{\otimes,D}$ in $\langle M\rangle_{\otimes,D}$, by
Theorem~\ref{theor-PPVdff}, we obtain a PPV extension $E$ for~$N$,
which is embedded into $L$ as a $D_K$-subfield. Moreover, by
construction, we have an isomorphism
$$
G/H\cong \Gal^{D_K}(E/K).
$$
We need to show that $E$ is the subfield in $K$ associated with $H$ by the Galois correspondence, that is, $E=L^H$. By Proposition~\ref{prop-Galoisparam}, it is enough to show the equality $\Gal^{D_E}(L/E)=H$. It is implied by the fact that $\Gal^{D_E}(L/E)$
is the kernel of the restriction homomorphism
\begin{equation}\label{eq:5560}
\Gal^{D_K}(L/K)=G\to
\Gal^{D_K}(E/K)=G/H.
\end{equation}
Conversely, if $E$ is a PPV extension of $K$ in $L$, then $H:=\Gal^{D_E}(L/E)$ is the kernel of the group homomorphism~\eqref{eq:5560}, whence $H$ is normal.
\end{proof}

\subsection{Extension of constants in parameterized differential fields}\label{sec:extofscalarsforMichael}

Now let us consider the behavior of PPV extensions and the
corresponding differential categories under extensions of the
differential field $(k,D_k)$. Let $(K,D_K)$ be a parameterized
differential field over $(k,D_k)$ with $\Char k=0$. Let $l$ be a
$D_k$-field over $k$ (Definition~\ref{defin-strictalg}). In
particular, we have a differential field~$(l,D_l)$ with
$D_l:=l\otimes_k D_k$. Since $\Char k=0$, the field $k$ is
algebraically closed in~$K$ and the ring
$$
R:=l\otimes_k K
$$
is a domain (for example, see~\cite[Corollary~1, p.~203]{Jab}).
Denote  the fraction field of $R$ by $L$. By
Definition~\ref{defin-extscaldiffmod}, $R$ is a $D_K$-al\-ge\-bra
over~$K$. Therefore, $L$ is a $D_K$-field over $K$ and we have
morphisms of differential rings $$(l,D_l)\to (R,D_R)\to (L,D_L),$$
where $D_R:=R\otimes_K D_K$ and $D_L:=L\otimes_K D_K$. Also, we have
$$
D_{R/l}:={\rm Ker}(D_R\to R\otimes_l D_l)\cong R\otimes_K D_{K/k}\cong l\otimes_k D_{K/k},\ \ 
D_{L/l}:={\rm Ker}(D_L\to L\otimes_l D_l)\cong L\otimes_K D_{K/k},
$$
because the functors $R\otimes_K -$ and $L\otimes_K-$ are exact.
\begin{lemma}\label{lemma-diffidealsR}
The $D_{K/k}$-algebra $R$ over $K$ has no non-zero $D_{K/k}$-ideals besides $R$ itself.
\end{lemma}
\begin{proof}
Let $I$ be a non-zero $D_{K/k}$-ideal in $R$ and consider
$
0\ne
f\in I$ with $$f=\sum_{i=1}^n c_i\otimes f_i,\ \  0\ne c_i\in l,\ 0\ne
f_i\in K$$
such that $c_1\ldots,c_n$ are linearly
independent over $k$. Suppose that $f$ has the minimal possible
number $n$ among all non-zero elements in $I$. Take any $\partial\in
D_{K/k}$. Since $\partial f\in I$, we have $$g:=(1\otimes f_1)\partial f-
(1\otimes \partial f_1)f \in I.$$ On the
other hand,
$$
\partial f=\sum_{i=1}^n c_i\otimes \partial f_i,$$ hence, $$g=\sum_{i=2}^n c_i\otimes (f_1\partial
f_i-\partial f_1\,f_i)$$
has less summands than $f$. Therefore, $g=0$. Since $c_1,\ldots,c_n$
are linearly independent over $k$, we obtain that
$\partial{\left(f_i/f_1\right)}=0$ for all $i=2,\ldots,n$ and for all
$\partial\in D_{K/k}$. Hence, $$h_i:=f_i/f_1\in
k=K^{D_{K/k}}$$ and we have
$$
f=\left(\sum_{i=1}^nc_ih_i\right)\otimes f_1.$$
Thus, $f$ is invertible in $R$ and $I=R$.
\end{proof}

\begin{lemma}\label{lemma-integraldiffmorph}
Let $P$ and $P'$ be $D_{R/l}$-modules over $R$ such that $P$ is a
finitely generated $R$-module and let $\phi:P_L\to P'_L$ be a
morphism between the corresponding differential modules over
$(L,D_{L/l})$. Then we have $$\phi(P\otimes 1)\subseteq P'\otimes 1.$$
\end{lemma}
\begin{proof}
Consider the subset $I\subset R$ that consists of all  $f\in R$ such that, for all
$v\in P\otimes 1$, we have $$f\cdot\phi(v)\in P'\otimes 1.$$ It is readily seen that $I$ is an ideal in $R$. Moreover, since the module $P$ is finitely generated over $R$, the ideal $I$ is non-zero. Take any $\partial \in D_{K/k}$. Since the $R$-submodule $P\otimes 1\subset P_L$ is stable under $\partial$ and $\phi$ is $L$-linear and commutes with $\partial$, for all  $f\in I$ and $v\in P\otimes 1$, we get
$$\partial f\cdot\phi(v)=\partial{\left(f\cdot \phi(v)\right)}-f\cdot\phi(\partial v)\in P'\otimes 1.$$ Hence, $I$ is a non-zero $D_{K/k}$-ideal in $R$. By Lemma~\ref{lemma-diffidealsR}, we conclude that $I=R$.
\end{proof}

\begin{corollary}\label{corol-newsolutions}
For all finite-dimensional $D_{K/k}$-modules $M$ and $M'$ over $K$, the natural map
$$
l\otimes_k\Hom_{D_{K/k}}{\left(M,M'\right)}\cong \Hom_{D_{L/l}}{\left(M_L,M'_L\right)}
$$
is an isomorphism, where, for a differential ring $(A,D_A)$,
$\Hom_{D_A}(-,-)$ denotes morphisms between differential modules
over $(A,D_A)$. In particular, we have $$l\otimes_k
M^{D_{K/k}}=M_L^{D_{L/l}}\quad \text{and}\quad l=L^{D_{L/l}}.$$
\end{corollary}
\begin{proof}
First, by Lemma~\ref{lemma-integraldiffmorph} with $P=M_R$ and
$P'=M'_R$, the natural morphism $$
\Hom_{D_{R/l}}{\left(M_R,M'_R\right)}\to\Hom_{D_{L/l}}{\left(M_L,M'_L\right)}$$ is an
isomorphism. Since $D_{R/l}\cong l\otimes_k D_{K/k}$ acts trivially on $l$ and $M_R\cong
l\otimes_k M$, we have canonical isomorphisms:
\begin{align*}
\Hom_{D_{R/l}}(M_R,M'_R)=\Hom_R(M_R,M'_R)^{D_{R/l}}\cong
{\left(l\otimes_k\Hom_K(M,M')\right)}^{l\otimes_k D_{K/k}}\cong
 l\otimes_k \Hom_{D_{K/k}}(M,M').\quad\qedhere
\end{align*}
\end{proof}

Thus, we see that $(L,D_L)$ is a parameterized differential field over $(l,D_l)$, and that there are no non-trivial new solutions over $L$ of linear $D_{K/k}$-differential equations given over $K$.
The following result is implied directly by
Corollary~\ref{corol-newsolutions} (more precisely, by its last
assertion $l=L^{D_{L/l}}$).

\begin{corollary}\label{corol-Michael}
Let $M$ be a finite-dimensional $D_{K/k}$-module over $K$, $E$ be a PPV
extension for $M$, and let $A$ be the $D_k$-Hopf algebra of the parameterized Galois group of $E$ over $K$. Then the $D_L$-field $F:=\Frac(l\otimes_k E)$ is
a PPV extension for~$M_L$ and the $D_l$-Hopf algebra of the parameterized Galois group of $F$ over $L$ is $A_l$.
\end{corollary}

By Theorem~\ref{theor-PPVdff}, Corollary~\ref{corol-Michael} also
follows from the following categorical statement, which makes sense
without the assumption of the existence of a PPV extension (or,
equivalently, the existence of a differential functor) and has
interest on its own right.

\begin{proposition}\label{prop-extensioncat}
Let $M$ be a finite-dimensional $D_{K/k}$-module over $K$. Then the
differential functor from a $D_k$-category over $k$ to a
$D_l$-category over $l$ (Definition~\ref{defin-diffgenercat},
Proposition~\ref{prop-diffextscal}, and
Theorem~\ref{teor-paramAtiayh})
$$
\langle M\rangle_{\otimes,D}\to \langle M_L\rangle_{\otimes,D},\quad X\mapsto X_L
$$
induces an equivalence $D_l$-categories
$$
\Phi:l\otimes_k\langle M\rangle_{\otimes,D}\stackrel{\sim}\longrightarrow \langle M_L\rangle_{\otimes,D}.$$
\end{proposition}
\begin{proof}
It is known that Hom-spaces in the extension of scalars category
$l\otimes_k\langle M\rangle_{\otimes,D}$ are obtained by taking
$l\otimes_k-$ from the Hom-spaces in the category $\langle
M\rangle_{\otimes,D}$ (\cite[p.407]{MilneMotives},~\cite{Stalder}).
Thus, it follows from Corollary~\ref{corol-newsolutions} that $\Phi$
is fully faithful. Let us show that $\Phi$ is essentially
surjective. Any object $N$ in $\langle M_L\rangle_{\otimes,D}$ is a
subquotient of $Q_L$ for some object $Q$ in $\langle
M\rangle_{\otimes,D}$, that is, there are $D_{L/l}$-submodules
$$N_1\subset N_2\subset Q_L\quad \text{such that}\quad N\cong N_2/N_1.$$ Indeed, this
is true for $M_L$, and also this property is preserved under taking
direct sums, tensor products, duals, subquotients, and the functor
$\At^1$. Put
$$
P_i:=N_i\cap (l\otimes_k Q)\subset Q_L,\quad i=1,\:2,
\quad\text{and}\quad P:=P_1/P_2.$$
We have
$$
P_i=\varinjlim\left(P_i\cap (V\otimes_k Q)\right),$$ 
where the limit is taken over all finite-dimensional over $k$
subspaces $V$ in $l$. Recall that objects in $$l\otimes_k
\langle M\rangle_{\otimes,D}$$ are $l$-modules in the category of
ind-objects in $\langle M\rangle_{\otimes,D}$
(\cite[p.~407]{MilneMotives},~\cite{Stalder}). Therefore, $P_i$
and~$P$ are objects in $l\otimes_k \langle M\rangle_{\otimes,D}$.
Finally, $\Psi(P)=N$, because $L=\Frac(l\otimes_k K)$, whence
$L\otimes_R P\cong N$.
\end{proof}

\setcounter{secnumdepth}{-1}
\section{Acknowledgments} 
The authors thank P.~Cassidy, P.~Deligne,  L.~Di Vizio, C.~Hardouin, M.~Kamensky,
N.~Markarian, D.~Osipov, K.~Shramov, M.~Singer, D.~Trushin, and M.~Wibmer for
the very helpful conversations and comments. The authors also appreciate the suggestions of the referee.
During the work on a
part of the paper, the second author enjoyed hospitality of
Forschunginstitut f\"ur Mathematik (FIM), ETH, Z\"urich, which he is
very grateful for.

H. Gillet was supported by the grants NSF DMS-0500762 and DMS-0901373. S. Gorchinskiy was supported
by the grants RFBR 11-01-00145-a, NSh-4713.2010.1, MK-4881.2011.1,
and AG Laboratory GU-HSE, RF government grant, ag. 11
11.G34.31.0023. A. Ovchinnikov was supported by the grants: NSF  CCF-0952591 and  PSC-CUNY  No.~60001-40~41.

\setcounter{secnumdepth}{0}
\begin{appendix}
\def\appendixname{}
\section{Appendix}\label{sec:appendix}

Here we recollect several known definitions and results and fix
some notation we extensively use in the paper.

\setcounter{secnumdepth}{2}
\setcounter{section}{1}
\subsection{Hopf algebroids}\label{subsection-prelHopfalg}

There are many references concerning Hopf algebroids, for example,
see~\cite[1.6,~1.14]{DeligneFS}. Also, the book~\cite{deJong} is very
useful. A {\it Hopf algebroid} is a pair of rings $(R,A)$ with the
following data and properties. First, there are two ring
homomorphisms $l:R\to A$ and $r:R\to A$, that is, $A$ is an algebra
over $R\otimes R$. In particular,~$A$ is an $R$-bimodule with the
left and the right $R$-module structures given by the
homomorphisms~$l$ and $r$, respectively. Further, there are
morphisms of algebras over~$R\otimes R$:
$$
\Delta:A\to A\otimes_R A, \quad e:A\to R,\quad
\imath:A\to A^s.
$$
According to our notation, the tensor product $A\otimes_R A$
involves both left and right $R$-module structures on~$A$. The ring
$R$ is considered as an algebra over $R\otimes R$ via the
multiplication on $R$. In particular, we have the identities $e\circ
l=\id_R$ and $e\circ r=\id_R$. Also, $A^s$ denotes the same ring $A$
with the right and left $R$-module structures being the initial left
and right $R$-module structures on $A$, respectively, that is, we
have $$\imath\big(l(f)a\big)=\imath(a)\,r(f)\quad \text{and}\quad
\imath\big(ar(f)\big)=l(f)\imath(a)\quad \text{for all}\ \ f\in R,\ a\in A.$$ The
morphisms $(l,r,\Delta,e,\imath)$ are required satisfy the following set
of axioms, which are similar to the axioms in the definition of a
Hopf algebra. The coassociativity axiom requires the equality of the
compositions
$$
\begin{xy}(0,0)*+{A}="ab";(17,0)*+{A\otimes_R A}="a"; (50,0)*+{A\otimes_R A\otimes_R A}="b";
{\ar@<0ex>"ab";"a"^<<<<<{\Delta}};
{\ar@<0.7ex>"a";"b"^>>>>>>>>>{\Delta\otimes\id_A}};
{\ar@<-0.7ex>"a";"b"_>>>>>>>>>{\id_A\otimes\Delta}};
\end{xy}.
$$
The counit axiom requires that both compositions
$$
\begin{xy}(0,0)*+{A}="ab";(17,0)*+{A\otimes_R A}="a"; (40,0)*+{A}="b";
{\ar@<0ex>"ab";"a"^<<<<<{\Delta}};
{\ar@<0.7ex>"a";"b"^>>>>>>>>{e\otimes\id_A}};
{\ar@<-0.7ex>"a";"b"_>>>>>>>>{\id_A\otimes e}};
\end{xy}
$$
are equal to the identity. Finally, the antipode axiom requires that the following diagrams commute:
$$
\begin{CD}
A@>\Delta>> A\otimes_R A\\
@VVeV @VV\imath\cdot\id_A V\\
R@>r>>A,
\end{CD}\quad\quad\quad\quad
\begin{CD}
A@>\Delta>> A\otimes_R A\\
@VVeV @VV\id_A\cdot\imath V\\
R@>l>>A.
\end{CD}
$$
In particular, it follows that $\imath$ is an involution and that $e\circ\imath=e$. Also, $\imath$ is uniquely defined by $\Delta$ and $e$.
Note that a Hopf algebroid $(R,A)$ with $l=r$ is the same as a Hopf
algebra $A$ over $R$.

A Hopf algebroid $(R,A)$ defines a Hopf algebra $B$ over
$R$ by the formula
$$
B:=R\otimes_{(R\otimes R)}A.$$
Further, by the extension of scalars, $B$ defines a Hopf algebra
$B\otimes R$ over $R\otimes R$. It follows from the
definition of a Hopf algebroid that $\spec(A)$ is a pseudo-torsor
under the group scheme $\spec(B\otimes R)$ over
$\spec(R\otimes R)$ (for example, see Definition~\ref{def:pseudotorsor}).

A {\it Hopf algebroid over a ring $\kappa$} is a Hopf algebroid
$(R,A)$ such that $A$ and $R$ are $\kappa$-al\-gebras, the
morphisms~$l$ and $r$ are morphisms of $\kappa$-algebras and the
morphisms $(\Delta,e,\imath)$ are morphisms of algebras over
$R\otimes_{\kappa} R$. In this case, $\spec(A)$ is a pseudo-torsor
under the group scheme $\spec(R\otimes_{\kappa} B)$ over
$\spec(R\otimes_{\kappa} R)$, where $B$ is defined as above.

A {\it comodule} over a Hopf algebroid $(R,A)$ is an $R$-module $M$
together with a morphism of \mbox{$A$-mo\-dules}
$$
\epsilon_M:M\otimes_R A\to A\otimes_R M
$$
that satisfies two axioms, which are similar to the axioms in the
definition of a comodule over a Hopf algebra. The first axiom
requires the equality $$R\otimes_A\epsilon_M=\id_M,$$ where the
$A$-module structure on $R$ is defined by the ring homomorphism
$e:A\to R$ and we use that $$R\otimes_A(M\otimes_R A)\cong
R\otimes_A(A\otimes_R M)\cong M.$$ The second axiom requires the
equality of the composition
$$
\begin{CD}
M\otimes_R A\otimes_R A@>\epsilon_M\otimes_R A>> A\otimes_R
M\otimes_R A@>A\otimes_R\epsilon_M>>A\otimes_R A\otimes_R M
\end{CD}
$$
to the extension of scalars
$$
(A\otimes_R A)\otimes_A \epsilon_M:M\otimes_R A\otimes_R A\to
A\otimes_R A\otimes_R M,
$$
where the $A$-module structure on $A\otimes_R A$ is given by the
ring homomorphism $\Delta$ and we use that
$$
(A\otimes_R A)\otimes_A (M\otimes_R A)\cong M\otimes_R A\otimes_R A
\quad \text{and} \quad (A\otimes_R\nolinebreak A)\otimes_A
(A\otimes_R M)\cong A\otimes_R A\otimes_R M.
$$
One proves that $\epsilon_M$ is an isomorphism. By adjunction
between extension and restriction of scalars, one obtains a left
\mbox{$R$-linear} morphism $$\phi_M:M\to A\otimes_R M,$$ and one can give an
equivalent definition of a comodule in terms of $\phi_M$. Denote the
category of comodules over a Hopf algebroid $(R,A)$ by
$\Comod(R,A)$. Denote the full subcategory of comodules over~$(R,A)$
that are finitely generated as $R$-modules by $\Comodf(R,A)$.

\begin{remark}
Given a Hopf algebroid $(R,A)$ over a ring $\kappa$, the pair
$(\Spec(A),\Spec(R))$ defines a category~$\mathcal G$ fibred in
groupoids over $\kappa$-schemes. A comodule over $(R,A)$ is the same
as a quasi-coherent sheaf on $\mathcal G$, or, equivalently, a
morphism of fibred categories from $\mathcal G$ to the fibred
category of quasi-coherent sheaves, \cite[3.3]{DeligneFS}.
\end{remark}

Given a morphism of rings $R\to S$ and a Hopf algebroid
$(R,A)$, there is a canonical structure of a Hopf algebroid on the
extension of scalars $(S,{}_SA_S)$, where
$$
_SA_S:=(S\otimes S)\otimes_{(R\otimes R)} A.$$
The extension of scalars also induces a functor $$\Comod(R,A)\to\Comod(S,{}_SA_S),\quad
 M\mapsto S\otimes_R M.$$
If $(R,A)$ is a Hopf algebroid over a ring $\kappa$ and $\kappa'$ is
an algebra over $\kappa$, then $(R_{\kappa'},A_{\kappa'})$ is a Hopf
algebroid over $\kappa'$ with
$$R_{\kappa'}:=\kappa'\otimes_{\kappa}R\quad \text{and}\quad
A_{\kappa'}:={\kappa'}\otimes_{\kappa} A.$$

For a Hopf algebroid $(R,A)$ over a ring $\kappa$, suppose that $A$ is a
{\it faithfully flat} module over $R\otimes_{\kappa} R$, that is,
$$
N\mapsto A\otimes_{(R\otimes_{\kappa} R)} N
$$
is a faithful exact
functor on the category of modules over~$R\otimes_{\kappa} R$. Then
a very important fact is that any $R$-finitely generated comodule
$M$ in $\Comodf(R,A)$ is a projective~$R$-module,~\cite[1.9,3.5]{DeligneFS}. It follows that $\Comodf(R,A)$ is a
Tannakian category with the forgetful fiber functor
$$\omega:\Comodf(R,A)\to\Mod(R)$$ (Section~\ref{subsection-prelTann}). Further, given a morphism
of \mbox{$\kappa$-algebras} $R\to S$, the extension of scalars $_S
A_S$ is faithfully flat over $S\otimes_{\kappa} S$ and the functor
$$S\otimes_R-:\nolinebreak\Comodf(R,A)\to\nolinebreak\Comodf(S,{}_SA_S)$$
is an equivalence of categories,~\cite[1.8,3.5]{DeligneFS}.

\subsection{Tannakian categories}\label{subsection-prelTann}

General references for Tannakian categories are~\cite{DeligneFS} and
\cite{Deligne}; also, an outline is given in~\cite[B.3]{Michael}. By
a {\it tensor category}, we mean a category $\Cat$ together with a
functor
\begin{equation}\label{eq:tenprodfun}
\Cat\times\Cat\to \Cat,\quad
(X,Y)\mapsto X\otimes Y,
\end{equation}
a unit object $\uno_{\Cat}$, and functorial isomorphisms
$$
X\otimes \uno_{\Cat}\cong X,\quad X\otimes Y\cong Y\otimes X,\quad
X\otimes(Y\otimes Z)\cong(X\otimes Y)\otimes Z
$$
that satisfy a set of axioms, which can be found in the references
above. A functor $F:\Cat\to\Dop$ between tensor categories is {\it
tensor} if there are functorial isomorphisms
\begin{equation}\label{eq-tensfunct}
F(X)\otimes F(Y)\cong F(X\otimes Y),\quad
F(\uno_{\Cat})\cong\uno_{\Dop}
\end{equation}
that are compatible with the commutativity and associativity
isomorphisms above. A {\it morphism between tensor functors}, $F\to
G$, is a morphism between functors that commutes with the
isomorphisms~\eqref{eq-tensfunct} for $F$ and $G$. Denote the
category of tensor functors between tensor categories $\Cat$ and
$\Dop$ by $ \Fun^{\otimes}(\Cat,\Dop)$.

An {\it internal Hom object} $\Homc_{\Cat}(X,Y)$ in a tensor
category $\Cat$ is an object that represents the functor from $\Cat$
to the category of sets $ U\mapsto \Hom_{\Cat}(U\otimes X,Y)$, that
is, there is a functorial isomorphism
$$
\Hom_{\Cat}(U,\Homc_{\Cat}(X,Y))\cong \Hom_{\Cat}(U\otimes X,Y).
$$
An internal Hom object $\Homc_{\Cat}(X,Y)$ is unique up to a
canonical isomorphism if it exists. Denote the internal Hom object
$\Homc_{\Cat}(X,\uno)$ by $X^\vee$. An object $X$ in $\Cat$ is {\it
dualizable} if, for any object $Y$, there exists the internal Hom
object $\Homc_{\Cat}(X,Y)$ and the natural morphism $$ X^\vee\otimes
Y\to \Homc_{\Cat}(X,Y)$$ is an isomorphism,~\cite[2.3]{DeligneFS}. A tensor category $\Cat$ is {\it rigid} if all
objects in $\Cat$ are dualizable.

Let $\Cat$ and $\Dop$ be tensor categories. Then, for any tensor
functor $F:\Cat\to\Dop$ and any dualizable object~$X$ in $\Cat$, the
object~$F(X)$ is also dualizable and the natural morphism
$$
F(\Homc_{\Cat}(X,Y))\to\Homc_{\Dop}(F(X),F(Y))
$$
is an isomorphism for any object $Y$ in $\Cat$, \cite[2.7]{DeligneFS}. If $\Cat$ is rigid,
then any morphism between tensor functors from $\Cat$ to $\Dop$ is
an isomorphism,~\cite[2.7]{DeligneFS}.

Recall that, in an {\it abelian category}, morphisms between objects
form abelian groups, there is a zero object, there are finite direct
sums of objects, and there are kernels and cokernels of morphisms,
satisfying some conditions. In particular, an analogue of the
homomorphism theorem for groups is satisfied. Also, in an abelian
category, exact sequences are well-defined. A functor $F:\Cat\to\Dop$
between abelian categories is {\it (left, right)-exact} if it sends
(left, right)-exact sequences to (left, right)-exact sequences. 

By
an {\it abelian tensor category}, we mean a tensor category such that
the tensor product functor is additive and right-exact on both
arguments. Let $F:\Cat\to \Dop$ be a right-exact tensor functor
between abelian tensor categories with~$\Cat$ being rigid. Then, the
functor~$F$ is exact,~\cite[2.10(i)]{DeligneFS}, and {\it faithful},
that is, injective on morphisms with the same source and
target~\cite[2.13(ii)]{DeligneFS}.

For an abelian rigid tensor category
$\Cat$ and an object $X$ in $\Cat$, denote the minimal full rigid tensor subcategory in~$\Cat$ that contains $X$ and is closed under taking subquotients by $ \langle
X\rangle_{\otimes} $. We
say that {\it $\langle
X\rangle_{\otimes}$ is tensor generated by $X$}.
It follows that~$\langle X\rangle_{\otimes}$ is an abelian subcategory in $\Cat$.

Let $R$ be a commutative ring. An {\it $R$-linear category $\Cat$}
is an additive category $\Cat$ such that, for all objects $X$, $Y$ in
$\Cat$, the group of morphisms $\Hom_{\Cat}(X,Y)$ is given with an
$R$-module structure and the composition of morphisms is
$R$-bilinear, that is, induces morphisms of $R$-modules
$$
\Hom_{\Cat}(X,Y)\otimes_R \Hom_{\Cat}(Y,Z)\to \Hom_{\Cat}(X,Z)
$$
for all object $X$, $Y$, and $Z$ in $\Cat$. A functor
$F:\Cat\to\Dop$ between $R$-linear categories is {\it $R$-linear} if
it induces $R$-linear maps
$$
\Hom_{\Cat}(X,Y)\to\Hom_{\Dop}(F(X),F(Y)).$$
Denote
the category of $R$-linear functors between $R$-linear categories $\Cat$ and $\Dop$ by $
\Fun_R(\Cat,\Dop)$.

Given a tensor category $\Cat$, a ring homomorphism $R\to
\End_{\Cat}(\uno_{\Cat})$ induces an $R$-linear category structure
on $\Cat$. By an {\it $R$-linear tensor category} we mean a tensor
category with an $R$-linear structure obtained as above.
Equivalently, one requires that the tensor product
functor~\eqref{eq:tenprodfun} is $R$-linear in both variables. For
example, for a finite commutative group $G$, the tensor category
$\Rep(G)$ is $k[G]$-linear, but it is not a $k[G]$-linear tensor
category with the tensor structure given by the usual tensor product
of representations. On the other hand, $\Rep(G)$ is a $k$-linear
tensor category.

A {\it Tannakian category over a field $k$} is an abelian rigid
tensor category $\Cat$ with a fixed isomorphism
$\End_{\Cat}(\uno_{\Cat})\cong k$ such that there exist a
$k$-algebra $R$ and a right-exact $k$-linear tensor functor
$\omega:\Cat\to \Mod(R)$. The functor $\omega$ is called a {\it
fiber functor}. It follows from the above that $\omega$ is exact and
faithful. A Tannakian category $\Cat$ is {\it neutral} if, in the
above notation, one can take $R=k$, that is, there exists a fiber
functor \mbox{$\omega:\Cat\to\Vect(k)$}.

Given a Tannakian category $\Cat$ and two fiber functors
$$\omega,\eta:\Cat\to\Mod(R),$$ denote the set of all tensor
isomorphisms between $\omega$ and $\eta$ by
$\fIsom^{\otimes}(\omega,\eta)$. Given an $R$-algebra $S$, one has
the fiber functor
$$
\omega_S:\Cat\to\Mod(S),\quad X\mapsto S\otimes_R\omega(X).
$$
Note that the functor $\omega_S$ is denoted by $S\otimes_R\omega$ in \cite{DeligneFS}. It is more convenient for us to reserve the notation $S\otimes_R\omega$ for the extension of scalars of the functor defined in~Section~\ref{subsection-prelextscal}.
The functor
$$
\Isom^{\otimes}(\omega,\eta):\Alg(R)\to \Sets,\quad S\mapsto \fIsom^{\otimes}(\omega_S,\eta_S),
$$
is corepresented by an $R$-algebra $A$,~\cite{DeligneFS}. In particular, the identity map from $A$ to itself corresponds to a canonical isomorphism of tensor functors $\omega_A\stackrel{\sim}\longrightarrow\eta_A$.

\begin{proposition}\label{prop-Tanngen}
In the above notation, suppose that $\Cat$ is tensor generated by an object $X$. Then the $R$-algebra $A$ is generated by the matrix entries of the canonical isomorphism
$$
\omega(X)_A\stackrel{\sim}\longrightarrow \eta(X)_A
$$
and the matrix entries of its inverse with respect to any choice of
systems of generators of $\omega(X)_A$ and~$\eta(X)_A$ over $A$.
\end{proposition}
\begin{proof}
Let $B$ be a $k$-subalgebra in $A$ generated by the matrix entries as in the proposition. We need to show that $B=A$. Given projective $B$-modules $P$ and $Q$, the extension of scalars map
$$
\Hom_B(P,Q)\to \Hom_A(P_A,Q_A)
$$
is injective, because $P$ and $Q$ are direct summands in free
$B$-modules and $B$ is embedded into $A$. Therefore, by the
universal property of $A$, it is enough to prove that the canonical
isomorphisms $$\omega(Y)_A\stackrel{\sim}\longrightarrow \eta(Y)_A$$
are defined over $B$, where $Y$ runs through all objects in $\Cat$.
By the construction of $B$, this is true for the tensor generator
$X$. Further, this property is preserved under taking direct sums,
tensor products, and duals of objects in $\Cat$. 

It remains to show
that this property is preserved under taking subquotients, for which
it is enough only to consider subobjects. Assume that there is an
isomorphism of $B$-modules
\mbox{$\lambda:\omega(Y)_B\stackrel{\sim}\longrightarrow \eta(Y)_B$}
whose extension of scalars $\lambda_A$ is equal to the canonical
isomorphism $$\omega(Y)_A\stackrel{\sim}\longrightarrow
\eta(Y)_A.$$ Given a subobject $Z\subset Y$, consider the
composition
$$
\mu:\omega(Z)_B\to\omega(Y)_B\stackrel{\lambda}\longrightarrow \eta(Y)_B\to \eta(Y/Z)_B.
$$
Since $\mu_A=0$ and $\omega(Z)_B$, $\eta(Y/Z)_B$ are projective
$B$-modules, we have that $\mu=0$. Hence,
\mbox{$\lambda\big(\omega(Z)_B\big)=\eta(Z)_B$}, which implies the
needed condition for $Z$.
\end{proof}

\begin{theorem}\label{theor-HopfalgDeligne}
Let $\Cat$ be a Tannakian category over $k$ and let $\omega:\Cat\to
\Mod(R)$ be a fiber functor. Then there exists a Hopf algebroid
$(R,A)$ over $k$ such that $A$ is faithfully flat over $R\otimes_k
R$ and $\omega$ lifts up to a tensor $k$-linear equivalence of
tensor categories $$\Cat\stackrel{\sim}\longrightarrow\Comodf(R,A).$$ That
is, for any object $X$ in $\Cat$, there is a functorial in $X$
structure of a comodule over $(R,A)$ on $\omega(X)$ giving the above
equivalence (\cite[1.12]{DeligneFS}).
\end{theorem}

In particular, for a neutral Tannakian category $(\Cat,\omega)$,
there exists a Hopf algebra $A$ over~$k$ such that~$\omega$ lifts up
to an equivalence between $\Cat$ and $\Comodf(A)$ (equivalently,
there exists an affine group scheme~$G$ over $k$ such that $\omega$
induces an equivalence between $\Cat$ and $\Repf(G)$). The Hopf
algebroid $A$ from Theorem~\ref{theor-HopfalgDeligne} corepresents
the functor
$$
\Isom^{\otimes}({}_{R\otimes R}\,\omega,\omega_{R\otimes
R}):\Alg(R\otimes_k R)\to \Sets,
$$
where, as above, we put
$$
(_{R\otimes R}\,\omega)(X):=(R\otimes_k R)\otimes_R\omega(X)\cong
R\otimes_k \omega(X),\quad (\omega_{R\otimes
R})(X):=\omega(X)\otimes_R(R\otimes_k R)\cong\omega(X)\otimes_k R.
$$
\end{appendix}

\bibliographystyle{model1b-num-names}
\bibliography{difftanncat}

\end{document}